\documentclass[11pt,letterpaper]{amsart}
\usepackage{amsfonts, amssymb, mathtools, amscd, latexsym, graphicx, psfrag, color, float, enumitem}
\usepackage[all]{xy}
\usepackage{mathrsfs}
\usepackage[mathscr]{euscript}

\usepackage{tikz}

\usepackage[hyphens]{url}

\usepackage{todonotes}
\usepackage{hyperref}

\usepackage[hyphenbreaks]{breakurl}

\usepackage{tkz-euclide}
\usepackage{caption}

\newtheorem{dummy}{dummy}[section]
\newtheorem{lemma}[dummy]{Lemma}
\newtheorem{theorem}[dummy]{Theorem}
\newtheorem{innercustomthm}{Theorem}
\newenvironment{customthm}[1]
{\renewcommand\theinnercustomthm{#1}\innercustomthm}
  {\endinnercustomthm}

 \newtheorem{innercustomdef}{Definition}
\newenvironment{customdef}[1]
{\renewcommand\theinnercustomdef{#1}\innercustomdef}
  {\endinnercustomdef}

\newtheorem{conjecture}[dummy]{Conjecture}
\newtheorem{corollary}[dummy]{Corollary}
\newtheorem{proposition}[dummy]{Proposition}
\newtheorem*{theorem*}{Theorem}
\newtheorem{definition}[dummy]{Definition}
\newtheorem*{definition*}{Definition}

\newtheorem{remark}[dummy]{Remark}

\newtheorem{question}[dummy]{Question}

\newtheorem{property}{Property}

\newcommand{\mHH}{\mathrm{HH}}
\newcommand{\mHP}{\mathrm{HP}}

\newcommand{\bfA}{\mathsf{A}}
\newcommand{\bA}{\mathbb{A}}
\newcommand{\bC}{\mathbb{C}}
\newcommand{\bG}{\mathbb{G}}
\newcommand{\bL}{\mathbb{L}}
\newcommand{\bP}{\mathbb{P}}

\newcommand{\bZ}{\mathbb{Z}}

\newcommand{\cS}{\mathcal{S}}
\newcommand{\cF}{\mathcal{F}}
\newcommand{\cO}{\mathcal{O}}
\newcommand{\cL}{\mathcal{L}}
\newcommand{\cU}{\mathcal{U}}

\newcommand{\cX}{\mathcal{X}}
\newcommand{\cP}{\mathcal{P}}
\newcommand{\cHH}{\mathcal{HH}}
\newcommand{\cH}{\mathcal{H}}

\newcommand{\cY}{\mathcal{Y}}
\newcommand{\cHP}{\mathcal{HP}}

\newcommand{\scrC}{\mathscr{C}}
\newcommand{\scrR}{\mathscr{R}}

\newcommand{\ft}{\mathfrak{t}}

\newcommand{\Map}[2]{\underline{\mathrm{Map}}\left (#1,#2\right )}
\newcommand{\Mapo}[2]{\underline{\mathrm{Map}}^{0}\left (#1,#2\right )}

\newcommand{\fMapo}[2]{\underline{\widehat{\mathrm{Map}}}^{0}\left (#1,#2\right )}

\newcommand{\op}{\mathrm{op}}
\newcommand{\dSt}{\mathrm{dSt}}

\newcommand{\St}{\mathrm{St}}
\newcommand{\dAff}{\mathrm{dAff}}
\newcommand{\dStk}{\mathrm{dSt}_{k}}

\newcommand{\Aff}[1]{\mathrm{Aff}(#1)}
\newcommand{\Spec}{\mathrm{S}\mathrm{pec}\,}
\newcommand{\Pic}{\underline{\mathrm{Pic}}}
\newcommand{\Bun}{\underline{\mathrm{Bun}}}

\newcommand{\El}{\mathcal{E}ll}
\newcommand{\an}{\mathrm{an}}

\newcommand{\Zar}[1]{|#1|_{\mathrm{Zar}}}

\newcommand{\CAlg}{\mathrm{CAlg}}
\newcommand{\coCAlg}{\mathrm{coCAlg}}
\newcommand{\Fun}{\mathrm{Fun}}
\newcommand{\Id}{\mathrm{Id}}
\newcommand{\Mod}{\mathrm{Mod}}
\newcommand{\Tot}{\mathrm{Tot}}

\newcommand{\colim}{\mathrm{colim}}

\newcommand{\Perf}{\mathrm{Perf}}
\newcommand{\Frac}{\mathrm{Frac}}

 \newcommand{\twocell}[1]{\ar@{}[#1]^(.30){}="a"^(.70){}="b" \ar@{=>} "a";"b"}
 \newcommand{\ocell}[1]{\ar@{}[#1]^(.30){}="a"^(.70){}="b" \ar@{=} "a";"b"}

\newcommand{\Qcoh}{\mathrm{Qcoh}}
\newcommand{\QCoh}{\mathrm{QCoh}}

\newcommand{\Coh}{\mathrm{Coh}}

\begin{document}

\author[Sibilla]{Nicol\`o Sibilla}
\address{Nicol\`o Sibilla, SISSA\\ 
Via Bonomea 265\\ 34136 Trieste TS\\
Italy}
\email{\href{mailto:nsibilla@sissa.it}{nsibilla@sissa.it}}

\author[Tomasini]{Paolo Tomasini}
\address{Paolo Tomasini, SISSA\\
Via Bonomea 265\\ 34136 Trieste TS\\
Italy}
\email{\href{mailto:ptomasin@sissa.it}{ptomasin@sissa.it}}

\title[Equivariant Elliptic Cohomology and Mapping Stacks]{Equivariant Elliptic Cohomology and Mapping Stacks}


\begin{abstract}
We introduce a new cohomology theory for stacks called elliptic Hochschild homology, prove some fundamental properties and compute it in some classes of examples. We then introduce its periodic cyclic version and show that, over the complex numbers and for a quotient stack,   this recovers Grojnowski's equivariant elliptic cohomology of the analytification.
\end{abstract}

\maketitle

\tableofcontents
\setcounter{tocdepth}{3}

\section{Introduction}
In this paper we give a new construction of equivariant elliptic cohomology via derived algebraic geometry. The use of techniques from derived algebraic geometry has become pervasive in  the study of elliptic cohomology, especially after the groundbreaking work of Lurie \cite{ELL1, ELL2, ELL3}.  However our aims in this paper are more limited and somewhat different in spirit from the developments originating from Lurie's work. For starters, we are only interested in rational phenomena, and therefore we will work over a fixed ground field  of characteristic zero.   Our goal is providing a geometric interpretation of the   equivariant elliptic cohomology of complex algebraic varieties equipped with the action of an algebraic group $G$.    
We will show that the $G$-equivariant elliptic cohomology of a variety $X$   can be described in terms of functions over a certain substack of the mapping stack 
 $$
\Map{E}{[X/G]}
$$ 
where $E$ is an elliptic curve over $k$, and $[X/G]$ is the stacky quotient of $X$ by   $G$. 
Our approach is closely related to earlier works by other authors including Gorbounov--Malikov--Schechtman--Vaintrob \cite{GMSV1} \cite{GMSV2} \cite{GMSV3} \cite{GMSV4}, Berwick-Evans \cite{BE1} and Berwick-Evans--Tripathy \cite{BETI}, Costello \cite{Cos1} \cite{Cos2}, Huan \cite{Huan18a}\cite{Huan18b}, Rezk \cite{Rez20} and Spong \cite{Spo}\cite{Spo19}\cite{Spo21}.

 We will introduce the  
 substack  of $
\Map{E}{[X/G]}
$ parametrizing \emph{quasi-constant maps}  
to $[X/G]$. We will show that functions over the stack of quasi-constant maps define a  cohomology theory of stacks, which we call \emph{elliptic Hochschild homology}.  We will study the formal properties of elliptic Hochschild homology, and compute it in important classes of examples.  
In the second to last section of the paper we will prove that  elliptic Hochschild homology recovers Grojnowski's rationalized  equivariant elliptic cohomology. The comparison  requires  an additional step, namely passing to the Tate fixed points under a natural action.  This  is familiar from the theory of classical Hochschild homology, which recovers periodized de Rham cohomology only    after passing to the Tate fixed points for the natural  $S^1$-action. The relationship between elliptic Hochschild homology and elliptic cohomology is entirely parallel to this.  

Our definition of equivariant elliptic Hochschild homology is valid  for a general group $G$. However some of our constructions will be geared towards the case when $G=T$ is a torus. We conclude the paper with a detailed discussion of the case when $G$ is a general (reductive) group. We propose a double-sided approach: one is inspired by Grojnowski's original ideas, depends on a choice of maximal torus $T$, and is developed in full detail in this paper. From this perspective, we obtain a full comparison between our theory and the equivariant elliptic cohomology of the analytification. We also propose a purely intrinsic approach, based on quasi-constant maps to $[X/G]$. 


\subsection*{\emph{Equivariant elliptic cohomology}}
Elliptic genera were first introduced by Ochanine in the 80's. Subsequently Witten   introduced what is now called the Witten genus, which is a kind of universal elliptic genus,   as the index of a Dirac operator on the loop space \cite{WitDir}. The work of Witten showed that elliptic genera had deep ties to quantum field theory \cite{WitQuant}, and spurred a great deal of research in the area. Elliptic cohomology was introduced in the late $80$-s to provide a conceptual framework for the study of elliptic genera. Elliptic cohomologies are even periodic cohomology theories whose associated formal group law is isomorphic to the completion of an elliptic curve at the identity. They have been the focus of great interest within  homotopy theory  for the last thirty years. Giving a satisfactory construction of elliptic cohomology, and its universal variant Tmf, is highly non-trivial. The state-of-the art is provided by ongoing work of Lurie \cite{ELL1}, \cite{ELL2}, \cite{ELL3}, which depends in a crucial way on the comprehensive foundations for  $\infty$-categories and spectral geometry which he has been developing in a series of books \cite{HTT}, \cite{HA}, \cite{SAG}, and by independent work of Gepner and Meier \cite{GM1} in the equivariant case.

It was understood early on by Ginzburg--Kapranov--Vasserot and Grojnowski that rationalized equivariant elliptic cohomology should give rise to coherent sheaves over the elliptic curve itself \cite{GKV} \cite{groj}. This fits into a well established  paradigm, first evinced in Atiyah-Segal's work on equivariant K-theory, that turning on equivariance is closely related to decompleting. In particular, as the formal group law of elliptic cohomology theories is the completion of an elliptic curve $E$ at the identity, the equivariant elliptic cohomology of a space $X$ with an $S^1$-action should take values in coherent sheaves over $E$. Further, the  stalks of this coherent sheaf can also be understood geometrically: they compute the Borel equivariant cohomology of various fixed points loci of $X$. Equivariance with respect to general Lie groups can be also understood in similar terms, and gives rise to coherent sheaves over the moduli space of $G$-bundles over $E$. In the influential article \cite{groj}, Grojnowski gives a beautiful construction of rationalized equivariant elliptic cohomology which implements this picture. Our work in this paper provides in particular  a  geometric explanation of Grojnowski's construction in terms of the defomation theory of (quasi-constant) maps out of  elliptic curves. 

\subsection*{\emph{Elliptic cocycles, double loops and secondary Hochschild homology}}
One of the  main challenges in elliptic cohomology is providing a geometric description of elliptic cocycles. Several influential proposals have been put forward starting from Segal's famous 1988 lecture at the Bourbaki seminar \cite{Segal88}, and subsequent work of   Stolz--Teichner \cite{EllObj},  \cite{STII}.  An implementation of this circle of ideas based on the concept of conformal nets has been pursued by Douglas, Bartels and Henriques in a series of works:  we refer the reader to \cite{douglas2011topological}  for an overview of this important perspective.

A different, but related, point of view is that elliptic cohomology  should be in some sense a categorification of K-theory, and hence be related to the K-theory of the loop spaces. That fits with the heuristics that raising the chromatic level should be related to categorification and looping. 
The point of view was pioneered by Witten \cite{WitQuant}\cite{WitDir} and Taubes \cite{Taub}, and has been applied by Devoto \cite{Dev96} in the case of $G$-equivariance with respect to a finite group $G$ (see also \cite{Mor09}, \cite{Gan09}, \cite{Gan13}, \cite{Dov19}). Interesting  recent developments include the work of Huan \cite{Huan18a} \cite{Huan18b}. From a different perspective, approaches using the cohomology of double loop spaces have been explored by Rezk \cite{Rez20} and Spong \cite{Spo}\cite{Spo19}\cite{Spo21}.

The idea that elliptic cohomology should be a categorification of K-theory suggests in particular that elliptic cocycles could be represented by appropriate categorified bundles. These ideas have been explored in \cite{baas2004two}. Within algebraic geometry, this perspective has been taken up by To\"en--Vezzosi,  who introduced   \emph{secondary Hochschild homology} as a model of elliptic cohomology \cite{TVChernIIFR}. We refer the reader to the introduction of \cite{TVChernENG} for a beautiful discussion of these ideas. In this paper we do not attempt to provide a geometric  interpretation of elliptic cocycles, although this is one of the broader goals of our project. However the definition of secondary Hochschild homology was an important motivation for our work, and thus it is useful  to review it here  and compare it with our construction.

Recall that the Hochschild homology of a scheme $X$ is given by the global sections of the structure sheaf of the derived loop space of $X$,
$$
\mathrm{HH}_*(X) \simeq \cO(\cL X)
$$
Here the derived loop space $\cL X = \mathrm{Map}(S^1, X)$ is the stack of maps from $S^1$ to $X$. Secondary Hochschild homology is defined as the global sections of the structure sheaf of the \emph{double loop space} of $X$
\begin{equation}
\label{secondary}
\mathrm{HH}^{(2)}_*(X) = \cO(\cL\cL  X)
\end{equation}
As a model of elliptic cohomology, $\mathrm{HH}^{(2)}_*(X)$ has several desirable features. First, the double loop space 
$$
\cL\cL X = \mathrm{Map}(S^1 \times S^1, X)
$$
is the moduli space of maps out of a topological torus, which captures the underlying topology of an elliptic curve. Additionally  categorified bundles  yield cocycles in $
\mathrm{HH}^{(2)}_*(X)$, as the  heuristics on elliptic cohomology  would dictate. We refer the reader to \cite{HSS, HSSS} for additional information on secondary Hochschild homology and its properties. 

On the other hand secondary Hochschild homology is insensitive to the complex moduli of elliptic curves, and therefore cannot be hoped to fully capture elliptic cohomology. Our construction can be described as a variant of  (\ref{secondary}), where we promote  the topological torus $S^1 \times S^1$ to an elliptic curve $E$ over a field $k$. When the ground field $k$ is the field of complex numbers $\mathbb{C}$, and $X$ is a complex scheme with an action of an algebraic group $G$, the resulting theory is closely related to the complexified (equivariant) elliptic cohomology of the analytification of $X$; and recovers it after passing to Tate fixed points for an appropriate action. One key difference with (\ref{secondary})  is that the full mapping stack 
$$
\Map{E}{X}
$$
is too large, in general, as there might be topologically non-trivial maps between $E$ and $X$. From the perspective of elliptic cohomology, the only maps that contribute are the \emph{quasi-constant} ones. As we explain next, this concept can be easily formalized.

\subsection*{\emph{Quasi-constant maps}}
We fix a ground field $k$ of characteristic zero, and an elliptic curve $E$ over $k$. For ease of exposition,  for most of this introduction we will restrict attention to the case when $G=T$ is 
a torus.

Let  $X$ be a variety equipped with a $T$-action.    
We denote by 
\begin{equation}
\label{almostconstant}
\Mapo{E}{[X/T]} \subset \Map{E}{[X/T]}
\end{equation}
the smallest clopen substack containing the trivial maps, i.e. the maps factoring as 
$$
E \to \Spec(k) \to [X/T]
$$ 
We call $
\Mapo{E}{[X/T]}$ the stack of \emph{quasi-constant maps}. 
In fact, the correct notion of quasi-constant maps is slightly more involved; 
 we refer the reader to Section \ref{qcm} for a complete exposition of this point.


 Working with maps that are close to being constant is familiar from many geometric contexts.  The product structure on Chen--Ruan orbifold cohomology, for instance, is governed by quasi-constant maps out of  marked rational curves. 
 For an example which is closely related to our story recall that, 
in defining the Witten genus via Dirac operator on the loop space, Witten and Taubes actually work with \emph{small loops}: i.e. with the normal bundle to constant loops, rather than with the full loop space. Quasi-constant maps out of $E$ can be viewed as an analogue, in our setting, of Witten and Taubes' small loops.

A key property of the stack of quasi-constant maps is that it satisfies a form of Zariski codescent on the target.

\begin{customthm}{A}[Theorem \ref{proposition:codescentnormalvar}]
\label{introcodescent}
Let $U_i \to X$ be a $T$-equivariant Zariski open cover. Then the natural map
$$
\varinjlim_i \Mapo{E}{[U_i/T]} \to 
\Mapo{E}{[X/T]}
$$
is an equivalence.
\end{customthm}
Theorem \ref{introcodescent} plays a crucial role in our construction. Codescent fails in general for the full mapping stack, as topologically non-trivial maps will not factor through any equivariant Zariski open cover of $X$. The fact that codescent holds for quasi-constant maps should be viewed as a counterpart of the Mayer--Vietoris principle in elliptic cohomology. It is an interesting question to what extent codescent is a general feature of  stacks of quasi-constant maps from an arbitrary source;  we refer the reader to Remark \ref{sec:general} in the main text for additional comments on this point. 

\subsection*{\emph{Elliptic Hochschild homology}}

The stack of quasi-constant maps carries a structure morphism  
$$
\Mapo{E}{[X/T]} \to \Mapo{E}{[pt/T]} \simeq \Pic^0(E) \otimes_\mathbb{Z} \check{T}
$$
where $\check{T}$ is the cocharacter lattice of $T$. The stack $\Pic^0(E)$ is the connected component of the Picard stack of $E$ which parametrizes degree $0$ line bundles. We can  rewrite this as 
$$
\Pic^0(E)\otimes_\mathbb{Z} \check{T}  \simeq \big ( \mathrm{Pic}^0(E)\otimes_\mathbb{Z} \check{T}\big ) \times [pt/T] 
 \simeq \big ( E \otimes_\mathbb{Z} \check{T} \big )\times [pt/T]  \simeq E^n \times [pt/T] 
$$
where $\mathrm{Pic}^0(E)$ is the Picard scheme of $E$ and $n$ is the rank of $T$. In particular, we have a natural map
$$
p: \Mapo{E}{[X/T]} \longrightarrow \mathrm{Pic}^0(E)  \otimes_\mathbb{Z} \check{T}
$$
We set $E_T:=\mathrm{Pic}^0(E)  \otimes_\mathbb{Z} \check{T}$.

The following is the most important definition of this article.
\begin{customdef}{B}[Definition \ref{elliptichoch}]
\label{mainintro}
The $T$-equivariant \emph{elliptic Hochschild homology}  of $X$ is 
$$
\mathcal{HH}_E([X/T]):=p_*(\mathcal{O}_{\Mapo{E}{[X/T]}}) \in \QCoh(E_T
)
$$
We denote by $\mathrm{HH}_E([X/T])$ the global sections of $\mathcal{HH}_E([X/T])$. 
\end{customdef}
Note that, as a consequence of Theorem \ref{introcodescent}, the sheaf $\mathcal{HH}_E([X/T])$ satisfies $T$-equivariant Zariski descent on $X$. That is, if $\cU=\{U_i\}$ is a $T$-equivariant Zariski open cover of $X$
$$
\mathcal{HH}_E([X/T]) \simeq \varprojlim_i \mathcal{HH}_E([U_i/T])
$$

Let us briefly discuss the case of a general reductive algebraic group $G$. In that setting, 
defining the relevant stack of maps out of $E$ takes a little more care, as we need to take into account also a \emph{semi-stability} condition. This is however easily done, see Section \ref{sect:ReductiveG} for a discussion. The stack of semi-stable quasi-constants maps is an open substack of $\mathrm{Map}^{0}(E,[X/G])$
$$
\mathrm{Map}^{0,s.s.}(E,[X/G])\hookrightarrow\mathrm{Map}^{0}(E,[X/G])
$$
and is equipped with a map $p$ to $E_G$, the coarse moduli space of degree $0$, semi-stable $G$-bundles on $E$. 
Then one defines as above 
$$
\mathcal{HH}_E([X/G]):=p_*(\mathcal{O}_{\Mapo{E}{[X/G]}}) \in \QCoh(E_G)
)
$$
As we explain in Section \ref{sect:ReductiveG}, it is also possible to give a different definition of $G$-equivariant elliptic Hochschild homology based on a reduction to a maximal torus. We refer the reader to  Section \ref{sect:ReductiveG}  for a more detailed discussion.


As we show, this construction satisfies all functorialities expected from an equivariant cohomology theory: it is contravariant with respect to morphisms of $G$-varieties $X\to Y$, and has contravariant change of group maps with respect to group homomorphisms $H\to G$.

\begin{remark}
The terminology \emph{elliptic Hochschild homology} was already used by Moulinos-Robalo-To\"en in their beautiful  paper \cite{MRT} to refer to a seemingly different construction. There are differences between our setting  and theirs. They work over a p-adic ring of integers $R$; additionally they do not consider the equivariant setting, which is of primary importance for us. 

However the two notions are intimately related, and in fact equivalent when they overlap. 
 In this article we place ourselves over a field $k$ of characteristic $0$ because we are interested in establishing properties of $\mathcal{HH}_E([X/T])$ which only hold in that setting; and ultimately we want to set $k=\mathbb{C}$ and compare our theory with complexified equivariant elliptic cohomology of the analytification of $X$.  Note however that our Definition \ref{mainintro} does not depend on the choice of ground ring, and therefore makes  sense also over a ring of p-adic integers.  We claim  that if $X$ is a derived scheme over $R$ as considered in \cite{MRT}, and $T$ is the trivial group, then Definition \ref{mainintro} is equivalent to the elliptic Hochschild homology of $X$ as defined in \cite{MRT}. This comparison result  appears in the companion paper \cite{SST}. This justifies our usage of the term \emph{elliptic Hochschild homology}, as it is compatible with its earlier definition in  \cite{MRT}.
\end{remark}

In this article we  establish several  fundamental formal properties of elliptic Hochschild homology.  
Some of our main results   are a \emph{localization theorem} for elliptic Hochschild homology, and a calculation of its \emph{analytic stalks}. These results, which we will explain in the next section of this introduction, ought to be considered as direct analogues of the local structure of rationalized equivariant elliptic cohomology that was first described by Grojnowski in \cite{groj}. A more recent reference, which is more closely related to our work from a methodological standpoint, is the description of the local structure of Hochschild homology of global quotient stacks in \cite{HChen}. 
Our results give a complete description of the local behaviour of elliptic Hochschild homology of quotient stacks, and  are key  to establish the comparison with Grojnowski's rationalized equivariant elliptic cohomology.

Before presenting these results however,   let us explain  two classes of examples for which $\mathcal{HH}_E(X)$ can be explicitly computed. 
 The first  observation is that, as expected, in the absence of a group action elliptic Hochschild homology coincides with ordinary Hochschild homology. This is an analogue of the fact that, in the non-equivariant regime, rationalization collapses all cohomology theories to singular cohomology.   
 
\begin{customthm}{C}
[Corollary \ref{corollary:LoopstackequalsMapo}]
\label{introtrivialaction}
Let $T$ be the trivial group. Then there is an equivalence 
$$
\mathrm{HH}_E(X) \simeq \mathrm{HH}_*(X)
$$
\end{customthm}

Next, let us consider the case  when $X$ is a smooth toric variety equipped with the action of the maximal torus. Toric actions on toric varieties are treated at length in Section \ref{section:mainthmproof}, for several reasons. First, calculations on affine spaces and projective spaces equipped with a torus actions are the cornerstone of our general structure results, as varieties equipped with a $T$-action admit  an equivariant local embedding in affine space. Second, toric varieties provide fully computable examples of our theory owing to the codescent Theorem \ref{introcodescent}. In particular, when $X$ is a smooth toric variety equipped with the action of a maximal torus $T$ we have the following result.    

Assume that $k=\mathbb{C}$ and let $X$ be a smooth toric variety equipped with the action of the maximal torus $T$. We denote by $\mathcal{E} ll^0_{T^{an}}(X^{an})$ the degree zero complexified $T$-equivariant elliptic cohomology of the analytification of $X$, viewed as a coherent sheaf over $E_T$  
$$
\mathcal{E} ll^0_{T^{an}}(X^{an}) \in \mathrm{Coh}(E_T)
$$
\begin{customthm}{D}
[Theorem \ref{theorem:ellipticohoftoricvar}]
\label{toricintro}
There is an equivalence in $\mathrm{Coh}(E_T)$
$$
 \mathcal{HH}_E([X/T])  \simeq \mathcal{E} ll^0_{T^{an}}(X^{an})
$$
\end{customthm}
Theorem \ref{toricintro} follows from two ingredients. The first is the calculation of the elliptic Hochschild homology of the affine space $\mathbb{A}^N$ under an arbitrary torus action, which  plays a key role in the proof of  Theorems \ref{introloc} and \ref{introcompletion}. The second is   codescent  for quasi-constant maps, Theorem \ref{introcodescent}. 


\subsection*{\emph{Main theorems}} Our main results  are contained in the last two Sections of the article, Section \ref{section:HELL} and \ref{section:Tate}. 
In Section  \ref{section:HELL}  we establish two structure theorems that describe the local behaviour of elliptic Hochschild  homology. They are exactly parallel to features of ordinary equivariant elliptic cohomology. As explained by Grojnowski, the local structure of elliptic cohomology is governed by information coming from the cohomology of fixed point loci. Further, on sufficiently small   neighbourhoods of points of $E_T$, elliptic cohomology is equivalent to  ordinary Borel equivariant cohomology of fixed point loci. In our setting, these two claims translate into the statements of Theorem \ref{introloc} and Theorem \ref{introcompletion} respectively.   Analogous statements for ordinary Hochschild homology were proved by Chen in \cite{HChen}, which was an important inspiration for our work.

As first suggested by Grojnowski, and further clarified  by work of Ro\c{s}u, we can associate to each closed point $x$ of $E_T$ a subgroup $T(x)$ of $T$.  When $T$ is of rank one, $T(x)$ is  equal to $T$ if $x$ is non-torsion and is equal to 
$\mu_n$  
if $x$ is torsion of (exact) order $n$. We denote $X^{T(x)}$ the derived fixed locus of $X$ under the induced $T(x)$-action. The classical fixed locus is given by the truncation $t_0(X^{T(x)})$.  

\begin{customthm}{E}
[Theorem \ref{proposition:localization}]
\label{introloc}
Let $X$ be a  smooth variety over $k$. Then for any closed point $x \in E_T$ there exists a Zariski open neighborhood $U$ of $x$ such that the natural map
\begin{equation}\label{eq:locmap}
\Mapo{E}{[t_0X^{T(x)}/T]}\times_{E_T} U\rightarrow\Mapo{E}{[X/T]}\times_{E_T} U
\end{equation}
induced by the inclusion $t_0 X^{T(x)}\rightarrow X$, is an equivalence.
\end{customthm}

Recall that $T\cong \Spec \mathrm{HH}_\ast([\ast/T])$. If $Y$ is a stack with a $T$-action, the Hochschild homology $
\mathrm{HH}_\ast([Y/T])$ carries an action of 
$\mathrm{HH}_\ast([\ast/T])$. 

\begin{customthm}{F}[Theorem \ref{theorem:completionsgeneral}]
\label{introcompletion}
Let $X$ be a smooth quasi-projective variety. The \'etale stalk of $\mathcal{HH}_E([X/T])$ at a point $x$ of $E_T$  is equivalent to the  completion of $\mathrm{HH}_\ast([t_0 X^{T(x)}/T])$ at  $$1\in T\cong \Spec \mathrm{HH}_\ast([\ast/T]).$$
\end{customthm}
 
Theorem \ref{introloc} and \ref{introcompletion} are key ingredients in the proof of the comparison between $\mathcal{HH}_E([X/T])$ and  equivariant elliptic cohomology. As we already discussed, the first step is   to introduce a \emph{periodic cyclic} variant of  $\mathcal{HH}_E([X/T])$, which we call  \emph{elliptic periodic cyclic homology} and denote  $\mathcal{HP}_E([X/T])$. Elliptic Hochschild homology, just as classical Hochschild homology, is  not homotopy invariant. Thus it cannot be hoped to coincide with elliptic cohomology on the nose. The fact that this discrepancy can be obviated by keeping track of an extra piece of data, in the form of a differential or of a cyclic action, is familiar from the classical story. However what exactly the  cyclic action might be in the elliptic setting is a somewhat  subtle issue. 

In the ordinary setting periodic cyclic homology is calculated by passing to the Tate fixed points under  the natural action of $S^1$ on Hochschild homology.  Via the identification 
 $$
\mathrm{HH}_*(X) \simeq \cO(\cL X)
$$
this action comes from the action of $S^1$ on $\cL X$ via  loop rotation. 
However elliptic Hochschild homology $\mathcal{HH}_E([X/T])$ does not carry in a natural way an  action of $S^1$.  It is defined instead in terms of functions on a  stack parametrizing maps out of the elliptic curve $E$
$$
\mathcal{HH}_E([X/T]):=p_*(\mathcal{O}_{\Mapo{E}{[X/T]}}) \in \QCoh(E_T
). 
$$
Thus it is endowed with a natural $E$-action. It is not difficult to see however  that this action is non-trivial only along the fibers of the structure map 
$$
\Mapo{E}{[X/T]} \to E_T
$$
This simple observation allows us to reinterpret the $E$-action as a hidden $S^1$-action, \emph{relative} to the base scheme $E_T$. The possibility to interpolate between a (cohomological) $E$-action and an action of the circle depends in a crucial way on the characteristic 0 assumption. Indeed it leverages the formality of the coherent cohomology of $E$, which is thus   equivalent to the coalgebra of cochains on $S^1$ (since the latter is also formal). We stress that both of these   formality statements fail   away from characteristic zero.  These identifications allow us to make sense of the Tate fixed points of elliptic Hochschild homology and we set  $$
\cHP_E([X/T]):=\mathcal{HH}_E([X/T])^{\mathrm{Tate}}$$ 
If $X$ is a variety with a $T$-action, its elliptic periodic cyclic homology becomes an object in the Tate fixed points of the trivial $S^1$-action on $\Perf(E_T)$. The latter coincides with the  \emph{$\mathbb{Z}_2$-folding} of $\Perf(E_T)$, i.e. the category obtained by collapsing the natural   $\mathbb{Z}$-grading on $\Perf(E_T)$ to a $\mathbb{Z}_2$-grading
$$
\cHP_E([X/T]) \in \Perf(E_T)^{\mathrm{Tate}} \simeq \Perf(E_T) \otimes_k k[u, u^{-1}]
$$
where $u$ is in degree $2$. 

\begin{customthm}{G}[Theorem  \ref{theorem:comparisongeneral}]\label{comparisonintro}
Let $k=\bC$. Let $T$ be an algebraic torus of rank $n$ acting on a smooth quasi-projective variety $X$.
We have an isomorphism of $\bZ_2$-periodic perfect complexes on $E$
$$\cHP_E([X/T])\simeq\El_{T^{\an}}(X^{\an})$$
where $\El_{T^{\an}}(X^{\an})$  is  the complexified $T^{\an}$-equivariant elliptic cohomology of the analytification of $X$.  \end{customthm}

An analogue of Theorem \ref{comparisonintro} for classical Hochschild homology is due to Halpern-Leistner and Pomerleano. Their Theorem 2.17 in \cite{HLPom} shows that the periodic cyclic homology of suitable quotient stacks $[X/G]$ over the complex numbers is equivalent to the $G^\an$-equivariant K-theory of the analytification $X^\an$ of $X$.

Let us comment briefly on the proof of Theorem \ref{comparisonintro}. It follows from   Theorem \ref{introcompletion} that the completions of $\cHP_E([X/T])$ and $\El_{T^{\an}}(X^{\an})$ at points of $E_T$ match. The formal completions at points of a smooth variety $X$ can be thought of as algebraic analogues of small complex balls covering $X$. This might suggest that this information is sufficient to establish a global equivalence between $\cHP_E([X/T])$ and $\El_{T^{\an}}(X^{\an})$. 
This is roughly correct, but  
the notion of completion has to be substantially enhanced. Completions at points  are by themselves insufficient. The correct notion is provided by the \emph{adeles} of a coherent sheaf, which are kinds of completed stalks labelled by flags of subvarieties. The adeles come with natural ``restriction maps'' which relate them, and give the adeles of a sheaf the structure of a cosimplicial complex. We review the theory of adelic descent in Section \ref{sec:adelic}. The theory of adelic descent goes back to  classic results of Weil  for curves, while its generalization to higher dimensions is due to Parshin and Beilinson. We will mostly follow the modern formulation of  Groechening \cite{Groech}, which is particularly convenient for our purposes.

Adelic methods alone are not quite sufficient to prove Theorem \ref{comparisonintro}. The reason is that, in general, the adeles of $\cHP_E([X/T])$ and $\El_{T^{\an}}(X^{\an})$ are hard to compute owing to the fact that $E_T$ is not affine. 
We work around this issue, by combining adelic descent with an induction on the rank of the torus $T$ acting on $X$. This uses in a crucial way our second structure result on the local behaviour of elliptic Hochschild homology, i.e.  Theorem \ref{introloc}.  

\begin{remark}
It is a natural question, see Ro\c{s}u \cite{Ros3}, whether  equivariant cohomology theories can be defined via adelic methods. Grojnowski's original approach involves a careful choice of analytic open cover, and this limits its applicability to the complex setting. A fully adelic treatment would have several benefits, and in particular would work over an arbitrary characteristic zero base. This is the subject of a recent paper of the second author \cite{EllAd}. 
\end{remark}

In the last section of the paper we  focus on the general case of a reductive group  $G$. We give two definitions of $G$-equivariant elliptic Hochschild homology, and its periodic cyclic version. The first definition closely mimics Grojnowski's original approach and  involves the choice of a maximal torus $T$ of $G$. Let $W$ be the Weyl group. We set
$$\cHH_E([X/G]):=\cHH_E([X/T])^W \quad \cHP_E([X/G]):=\cHP_E([X/T])^W$$
Using this definition we can prove the following comparison theorem.
\begin{customthm}{H}[Corollary  \ref{cor:comparisonG}]\label{groupGintro}
Let $k=\bC$. Let $G$ be a reductive group acting on a smooth quasi-projective variety $X$.
We have an isomorphism of $\bZ_2$-periodic perfect complexes on $E$
$$\cHP_E([X/G])\simeq\El_{G^{\an}}(X^{\an})$$
where $\El_{G^{\an}}(X^{\an})$ is  the complexified $G^{\an}$-equivariant elliptic cohomology of the analytification of $X$.  \end{customthm}
We also propose a second fully intrinsic definition of 
$G$-equivariant elliptic Hochshchild homology. This has the advantage not depending on the choice of a maximal torus (so, in particular, this definition makes sense for any algebraic group and not just for reductive ones). In this approach, $G$-equivariant elliptic Hochshchild homology is defined in terms of functions on the \emph{semi-stable} locus of the stack of quasi-constant maps from $E$ to $[X/G]$. We remark that this definition trivially agrees with the previous one in the case when $G=T$ is a torus. We conclude the paper with a sketch of the  proof of a comparison between these two definitions of  $G$-equivariant equivariant Hochschild homology. 

%
%

{\bf Acknowledgements:} 
This project started from an observation of Sarah Scherotzke, which appears in this article as Theorem \ref{introtrivialaction}. We thank Sarah for  generously sharing her ideas and for her encouragement. In the early stages of  this project we had a very enlightening email exchange with Jacob Lurie,  that made clear that he was already aware of the results that we present in this paper. Pavel Safronov offered key insights at various points during our work. We benefitted from conversations with many members of the community.  We thank Dragos Fratila, David Gepner, Sam Gunningham, Lennart Meier, Hayato Morimura, Tasos Moulinos, James Pascaleff, Emanuele Pavia, Mauro Porta,  Vesna Stojanoska, Mattia Talpo and Bertrand To\"en for useful exchanges and for their encouragement. The first author wants to thank  David Ben-Zvi for very  inspiring conversations during a visit to Austin in 2019, which sparked his interest in elliptic cohomology.   A substantial part of this project was carried out while the second author was in residence at the Hausdorff Research Institute for Mathematics, in Bonn, as a participant in the Program ``Spectral Methods in Algebra, Geometry, and Topology.'' We  thank the HIM for excellent working conditions.  

\section{Preliminaries}
\label{section:preliminaries}
Throughout the paper, $k$ is a fixed field of characteristic zero. We use the term \emph{variety} to mean a $k$-scheme which is integral, separated and of finite type. Unless we explicitly state otherwise, all geometric objects in the paper are implicitly assumed to be defined over $k$.  

In this paper we use the language of $\infty$-categories and derived algebraic geometry as developed by Jacob Lurie in \cite{HTT} and \cite{SAG}. We will mostly work over a field of characteristic zero, where derived rings can be modelled equivalently by simplicial commutative algebras or by commutative differential-graded algebras. In the setting of commutative cdga-s, foundations for derived algebraic geometry were developed by To\"en and Vezzosi in \cite{HAGI} and \cite{HAGII}.
\subsection{Derived stacks}
 Let $\CAlg$ be the $\infty$-category of simplicial commutative rings, and $$\dAff:=\CAlg^\op$$ be the $\infty$-category of derived affine schemes. We also use the name \emph{derived rings} for simplicial commutative rings. Constant simplicial commutative rings embed fully faithfully in simplicial commutative rings, and the embedding has a left adjoint corresponding to the  connected components functor, which is denoted $\pi_0$. We will refer to constant simplicial commutative rings as \emph{underived, classical}, or \emph{discrete} rings interchangeably.

The $\infty$-category of derived prestacks is the $\infty$-category of functors $$\cP(\dAff):=\Fun(\dAff^\op,\cS)$$ from simplicial commutative rings to the $\infty$-category of spaces. The $\infty$-category of derived stacks is   the $\infty$-category of (hypercomplete) sheaves on $\dAff$ with respect to the étale topology, $\dSt$. The category of derived stacks is naturally an $\infty$-topos. We denote by $\mathrm{Spec}$ the Yoneda embedding 
$$\mathrm{Spec}: 
\CAlg^\op\rightarrow\cP(\CAlg^\op).$$

Analogously, let $\cP(\mathrm{Aff}):=\Fun(\mathrm{CRing},\cS)$ be $\infty$-category of presheaves over classical commutative rings. The $\infty$-category of hypercomplete sheaves on $\mathrm{CRing}$ with respect to the \'etale topology is  the category of \emph{higher stacks}, $\St$. These embed fully faithfully in derived stacks; the embedding has a right adjoint called the \emph{truncation} functor and denoted by $t_0$. We refer to derived stacks which are equivalent to their truncation as \emph{underived, classical} or \emph{discrete}.

In a relative setting, given a simplicial commutative ring $R$, we define $\CAlg_R$ to be the $\infty$-category of commutative simplicial $R$-algebras.  We denote the $\infty$-category of derived prestacks over $R$, $\cP(\CAlg_R^\op)$, and the $\infty$-category of stacks over $R$, $\dSt_R$. If $R=k$ is a field of characteristic zero, the $\infty$-category $\CAlg_k$ is equivalent to the $\infty$-category $\mathrm{dg-cAlg}_{k}^{\le 0}$ of connective commutative dg algebras, i.e. concentrated in nonpositive degrees in cohomological indexing convention. In this paper we mostly work in this setting.

\subsection{Effective epimorphisms, geometricity, connected components}
\emph{Effective epimorphisms} are the natural notion of surjective maps in an  $\infty$-topos. Effective epimorphisms can be characterized as follows

\begin{definition}[\cite{HTT}, Corollary 6.2.3.5]
A morphism $f:X\rightarrow Y$ in an $\infty$-topos is an \emph{effective epimorphism} if one of the following two equivalent conditions is satisfied:
	\begin{enumerate}
		\item $f$ is a $(-1)$-truncated object in the $\infty$-topos $\dSt_{/Y}$ of derived stacks over $Y$ 
		\item The $\check{C}$ech nerve $\check{C}(f)$ is a simplicial resolution of $Y$.
	\end{enumerate}
\end{definition}
This definition, in the special case of the $\infty$-topos of derived stacks, becomes equivalent to the following property:
\begin{property}
A map of derived stacks $f:X\rightarrow Y$ is an effective epimorphism if for any representable $\Spec S$ and any map $\Spec S\rightarrow Y$, there exists an etale cover $\{S_i\}$ of $S$ such that for all $i$ the composition $\Spec S_i\rightarrow\Spec S\rightarrow Y$ admits a lift $\Spec S_i\rightarrow X$.
\end{property}
In other words, an effective epimorphism of derived stacks is an étale locally surjective map, just as in the classical theory of sheaves a surjective map of sheaves is a locally surjective map.

There is an important class of derived stacks called \emph{geometric} (derived) stacks. This notion is   a generalization to the setting of derived algebraic geometry of the more classical notion of Artin stack. Geometric stacks are characterized by an integer number called the \emph{geometric level}. The definition is stated inductively on the level. We will state the definition in the context of smooth maps, for a discussion in full generality see for example \cite{PORTA2016351}.

\begin{definition}
Let $X$ be a derived stack.
\begin{itemize}
\item $X$ is $(-1)$-geometric if it is an affine derived scheme. A map $f:X\rightarrow Y$ of derived stacks is $(-1)$-representable if for every map $$\Spec A\rightarrow Y$$ from a $(-1)$-geometric stack $\Spec A$, the fiber product $\Spec A\times_Y X$ is $(-1)$-geometric. A map is $(-1)$-smooth if it is $(-1)$-representable, and the induced map $$\Spec A\times_Y X\rightarrow\Spec A$$ is a smooth map of affine derived schemes.

\item 
$X$ is $n$-geometric if the diagonal map $X\rightarrow X\times X$ is $(n-1)$-representable and there exists an effective epimorphism $\coprod\Spec A_i\rightarrow X$, called an $n$-atlas for $X$, such that each map $\Spec A_i\rightarrow X$ is $(n-1)$-smooth. We say that $X$ is geometric if it is $n$-geometric for some $n$.
\item A map of derived stacks $X\rightarrow Y$ is $n$-representable if for every $(-1)$-geometric $\Spec A$ and any map $\Spec A\rightarrow Y$, the fiber product $\Spec A\times_Y X$ is $n$-geometric. 
\item A map $X\rightarrow Y$ is $n$-smooth if it is $n$-representable, and for any $\Spec A\rightarrow Y$ there exists a $n$-atlas $\coprod\Spec B_i$ of $\Spec A\times_Y X$ such that for all $i$ the composition $$\Spec B_i\rightarrow \Spec A\times_Y X\rightarrow \Spec A$$ is smooth.
\end{itemize}
\end{definition}

There is a notion of open and closed immersion of geometric stacks. 
\begin{definition}
Let $f:X\rightarrow Y$ be a morphism of derived geometric stacks. We say that  $f$ is an \emph{open} (resp. \emph{closed}) \emph{immersion} if for any map $g:\Spec S\rightarrow Y$ the fiber product $\Spec S\times_Y X$ is a derived scheme, and the induced map $\Spec S\times_Y X\rightarrow\Spec S$ is an open (resp. closed) immersion of derived schemes.
\end{definition}
In the context of derived schemes, the above definition is equivalent to the definitions below.
\begin{definition}[\cite{PanVez}, Definition 4.2]
 A map of derived schemes $f:X\rightarrow\Spec S$ with affine target is an \emph{open immersion} if there exist affine derived schemes $\Spec A_i\rightarrow X$ over $X$ such that the composite map $\coprod_i\Spec A_i\rightarrow X\rightarrow\Spec S$ is an effective epimorphism, and each composite $\Spec A_i\rightarrow X\rightarrow\Spec S$ exhibits $A_i$ as a localization of $S$.

Let $f:X\rightarrow Y$ be a morphism of derived schemes. $f$ is an \emph{open immersion} if for any affine derived scheme $\Spec S$, the induced map $f_S:\Spec S\times_Y X\rightarrow \Spec S$ is an open immersion of derived schemes with affine target. 
\end{definition}

\begin{definition}
Let $f:X\rightarrow Y$ be a morphism of derived schemes. We say that $f$ is a \emph{closed immersion} if the map on the underlying classical schemes $t_0 f: t_0 X\rightarrow t_0 Y$ is a closed immersion of classical schemes.
\end{definition}

We will also need the notion of connected component of a point in a derived geometric stack $X$. Let $K$ be a field, and let $x:\Spec K\rightarrow X$ be a $K$-point. Consider the full subcategory of $\dSt_{/X}$ of open and closed maps to $X$ whose image contains the $K$-point $x$:
$$
\mathrm{Clopen}_{x}{X} =\{  a:G\rightarrow X\mbox{ clopen map in }\dSt  \mbox{ such that }x\mbox{ factors through }a\}
$$
\begin{definition}
The connected component of the point $x$ in $X$, $X^{(x)}$, is an initial object in the $\infty$-category $\mathrm{Clopen}_{x}{X}$ (which always exists). 
\end{definition}

It will be important to consider quasi-coherent sheaves on derived (pre)stacks. 
\begin{definition} Let $X$ be a prestack. 
The \emph{$\infty$-category of quasi-coherent sheaves} on $X$, $\QCoh(X)$, is defined as the limit 
$$
\QCoh(X):=\mathrm{lim}_{\{\Spec A\rightarrow X\}}\Mod_A
$$
over the $\infty$-category of derived affine schemes with a map to $X$. Here $\Mod_{A}$ denotes the $\infty$-category of  $A$-modules.
\end{definition}
\subsubsection{$\bZ_2$-folding of quasi-coherent sheaves}
We will be interested in dealing with a $\bZ_2$-periodic version of the $\infty$-category of quasi-coherent sheaves. We review this object following Preygel \cite{PreyTate}. There, he introduces a Tate construction on $\infty$-categories with an action of $S^1$ and this formalism recovers in particular $\mathbb{Z}_2$-folding, which is what we are interested in.  
\begin{definition}[Definition 1.2.3 in \cite{PreyTate}]
Let $\mathscr{C}$ be a small stable $k$-linear idempotent complete $\infty$-category with an action of $S^1$, where $k$ is a field of characteristic zero. Then the Tate construction of $\scrC$ with respect to this $S^1$-action is the tensor product of small stable $\infty$-categories
$$\scrC^{\mathrm{Tate}}:=\scrC^{S^1}\otimes_{\Perf(k[[u]])} \Perf(k((u)))$$
where $u$ is a variable of homological degree $-2$. 
\end{definition}
In the following, we adopt Preygel's notation and omit $\Perf$, i.e. we shall write simply
$$\scrC^{\mathrm{Tate}}:=\scrC^{S^1}\otimes_{k[[u]]}k((u))$$

If $\scrC$ is not small but is equipped with a coherent t-structure (see Definition 4.2.7 in \cite{PreyTate}), Preygel defines its Tate fixed points via a \emph{regularization} procedure. The regularization of $\scrC$, $\scrR(\scrC$), is defined as the ind-completion of the full subcategory of \emph{coherent} objects, i.e. bounded above objects $C$ whose $r$-truncation, $\tau_{\leq r}C$, is compact for all $r$ (see Definition 4.2.2 in \cite{PreyTate}). Then Preygel defines (see Definition 1.3.4 in \cite{PreyTate})
$$\scrC^{t\mathrm{Tate}}:=\scrR(\scrC^{S^1})\otimes_{k[[u]]}k((u))$$
Following Preygel, we refer to $k((u))$-linear $\infty$-categories as $\bZ_2$-periodic. 

If $X$ is a Noetherian geometric stack, the standard t-structure on $\QCoh(X)$ is coherent (see Proposition 4.4.1 in \cite{PreyTate}). In particular, we can consider the Tate construction with respect to a trivial $S^1$-action on $\QCoh(X)$.
\begin{definition}
The \emph{$\bZ_2$-folding} of $\Perf(X)$ is the category 
$$
\Perf(X)_{\mathbb{Z}_2} := \Perf(X)^{\mathrm{Tate}}
$$
The \emph{$\bZ_2$-folding} of $\QCoh(X)$ is the category $$\QCoh(X)_{\mathbb{Z}_2}:=\QCoh(X)^{t\mathrm{Tate}}$$
\end{definition}
\subsection{Betti stacks and affinization}\label{subsection:affinization}
We define an important class of derived stacks, called \emph{Betti stacks}, which correspond to spaces. 
\begin{definition}
Let $X\in\cS$ be a space.
The \emph{Betti stack} associated to the space $X$ is the sheafification of the constant presheaf $\underline{X}:\dAff^\op\rightarrow\cS$ sending any derived affine $\Spec A$ to $X$. We abuse notation by denoting the Betti stack associated to a space $X$ again by $X$. 
\end{definition}

The \emph{affinization} of a derived stack is a fundamental notion, and is strictly related to the concept of \emph{affine stack} which we review below. The construction has been introduced in the nonderived setting by To\"en in \cite{ToenPhD}, and studied in the derived context by Ben-Zvi and Nadler. A review is in \cite{MRT}. We recall the construction here for the reader's convenience.

Let $k$ be a discrete commutative ring. We denote by $\coCAlg_k$ the $\infty$-category of cosimplicial commutative algebras over $k$, and $\mathrm{St}_k\subset\dSt_k$ the $\infty$-category of classical higher stacks.
\begin{definition}
Denote by 
\begin{equation*}
\Spec^{\Delta}:\coCAlg\rightarrow\mathrm{St}_k
\end{equation*}
the functor that sends a coconnective commutative $k$-algebra $A$ to the functor it corepresents, i.e. the functor sending a simplicial commutative ring $B$ to the space of maps $\mathrm{Map}_{\coCAlg}(A,B)$. Stacks of the form $\Spec^{\Delta}A$ for some cosimplicial commutative $k$-algebra $A$ will be called \emph{affine stacks}.\footnote{The same notion is referred to as \emph{coaffine stacks} by Lurie in \cite{DAGVIII}.}
\end{definition}
The functor $\Spec^{\Delta}$ has a left adjoint $\cO:\mathrm{St}_k\rightarrow\coCAlg$. The composition $$\Spec^{\Delta}\cO:\mathrm{St}_k\rightarrow\mathrm{St}_k$$ is called the \emph{affinization} functor.

We will consider the affinization of an elliptic curve $E$ over a field $k$ of characteristic zero, i.e. the affine stack $\Spec^{\Delta}\cO(E)$. Over $k$ the cdga $\cO(E)$ is formal, and thus isomorphic to its cohomology, which is given by 
\begin{equation*}
H^i(E;\cO_E)=
\begin{cases}
k\mbox{ if $i=0$ or $i=1$}\\
0\mbox{ else}
\end{cases}
\end{equation*}
The cdga $\cO(S^1)=C^\ast(S^1;k)$ is also formal, which implies 
\begin{equation*}
\Aff{E}\simeq\Aff{S^1}
\end{equation*}
Since the algebra of derived global functions on $S^1$ is isomorphic to the commutative dg-algebra $k[\epsilon]$ with the variable $\epsilon$ in (cohomological) degree $1$, the affine stack $\Aff{S^1}$ is sometimes denoted also by $\bA^1[1]$ and referred to as the \emph{shifted affine line}. 

\begin{remark}
Working over a field $k$ of characteristic zero we can model simplicial commutative algebras over $k$ via connective commutative dg-algebras. In this setting, Ben-Zvi and Nadler develop in \cite{BZNLoopConn} a similar construction. They consider the functor
\begin{equation*}
\Spec:\mathrm{dg-cAlg}_{k}\rightarrow\dSt_k
\end{equation*}
sending $A$ to the functor  mapping $B\in\mathrm{dg-cAlg}_{k}^{\le 0}$ to $\mathrm{Map}_{\mathrm{dg-cAlg}_{k}}(A,B)$. This functor is right adjoint to $\cO$, and the affinization functor is $\Spec\cO:\dSt_k\rightarrow\dSt_k$.  
\end{remark}

\subsection{Tangents and loops}\label{subsection:tangents,loops,maps}
Given two derived stacks $X$ and $Y$, we denote their mapping stack as $\Map{X}{Y}$.   
As a presheaf on $\CAlg^\op$, the mapping stack is characterized by 
\begin{equation*}
\Map{X}{Y}(S)=\mathrm{Map}(X\times\Spec S,Y)
\end{equation*}
The mapping stack can also be defined relative to some base derived stack $B$: for $X$ and $Y$ over $B$, the mapping stack $\Map{X}{Y}_{B}$ relative to $B$ is given by 
\begin{equation*}
\Map{X}{Y}_{B}(S)=\mathrm{Map}_{/B}(X\times_{B}\Spec S,Y)
\end{equation*}

Given a derived geometric stack $X$, we can build new derived stacks via universal constructions. In this paper we will consider the shifted tangent $T_X[-1]$, the unipotent loop space $\cL^u X$, the loop space $\cL X$ and the derived stack of quasi-constant maps $\Mapo{E}{X}$ from an elliptic curve $E$. The first three constructions are well documented in the literature, and we will recall them in this section, while the latter is new and is one of the main objects of study of this paper. 

If $X$ is a derived geometric stack   its \emph{cotangent complex} is the quasi-coherent sheaf $$\bL_X\in\Qcoh(X)$$ corepresenting the functor of derivations. 
\begin{definition}
The \emph{shifted tangent bundle} of $X$, $T_X[-1]$, is the relative spectrum over $X$ $$T_X[-1]:=\Spec_{\cO_X}\mathrm{Sym}\,\bL_X[1]$$
The \emph{unipotent loop space} of $X$, $\cL^u X$, is the mapping stack $$\cL^u X:=\Map{\Aff{S^1}}{X}$$
The \emph{loop space} of $X$, $\cL X$, is the mapping stack $$\cL X:=\Map{S^1}{X}$$
\end{definition}
When $X$ is a derived scheme, the shifted tangent bundle, the loop space and the unipotent loop space are all equivalent. This follows from the fact that these three objects are cosheaves over the Zariski site of a derived scheme, as explained in \cite{BZNLoopConn}, and  they coincide   when the target $X$ is affine. As an example, we recall the relevant codescent statement for the loop space:
\begin{lemma}[Lemma 4.2, \cite{BZNLoopConn}]
Let $X$ be a derived scheme. The functor
$$
\Zar{X}\rightarrow \dStk \, , \quad  
U \in \Zar{X} \mapsto\cL U \in  \dStk
$$
is a cosheaf.
\end{lemma}

All these constructions have formal counterparts, where one formally completes the stacks at a trivial locus corresponding to $X$. We start by  recalling the notion of formal completion in derived algebraic geometry. A summary of this and related notions can be found for example in \cite{HChen}.
\begin{definition}
Let $X$ and $Y$ be derived stacks, and let $f:X\rightarrow Y$ be a map. We define the \emph{formal completion} of $Y$ at $X$, $\widehat{Y}_X$, as the derived stack whose functor of points is the following: for every ring $S$, the space $\widehat{Y}_X(S)$ is the space of commutative diagrams
$$
\xymatrix{
\Spec \pi_0(S)^\mathrm{red}\ar[r]\ar[d] & X\ar[d]^{f} \\
\Spec S\ar[r] &Y
}
$$
where $\pi_0(S)^{\mathrm{red}}$ is the reduction of the discrete ring $\pi_0(S)$ considered as a constant simplicial commutative ring.
\end{definition}
We can describe formal completions alternatively using the \emph{de Rham} stack of a derived stack $Y$. Let $\mathrm{CRing^{\mathrm{red}}}$ be the category of reduced classical commutative rings, which embeds in the $\infty$-category of simplicial commutative rings $\CAlg$ as constant simplicial rings. 
In particular, we get a restriction functor
\begin{equation*}
i^\ast:\cP(\dAff)\rightarrow\cP(\mathrm{CRing^{\mathrm{red}}}^\op)
\end{equation*}
from derived prestacks to presheaves on $\mathrm{CRing^{\mathrm{red}}}^\op$. We can construct a right adjoint to this functor by sending a presheaf on $\mathrm{CRing^{\mathrm{red}}}^\op$ to its right Kan extension along $i$. We call this functor $i_\ast$. 
\begin{definition}
Let $Y$ be a derived prestack. Its \emph{de Rham} prestack is the derived prestack 
$$Y_{dR}:=i_\ast i^\ast Y$$
\end{definition}
By definition, given a derived ring $S$, an $S$-point of $Y_{dR}$ is a $\pi_0(S)^\mathrm{red}$-point of $Y$. 
Observe that the unit of the adjunction $\Id\rightarrow i_\ast i^\ast$ gives us a  map $Y\rightarrow Y_{dR}$. 
Via the de Rham stack we can describe the formal completion of $Y$ at $X$ as the pullback 
$$
\xymatrix{
\widehat{Y}_X\ar[r]\ar[d] & X_{dR}\ar[d]\\
Y\ar[r] & Y_{dR}
}
$$

Now we can define the formal completions at $X$ of the objects we described earlier.
\begin{definition}\label{def:formalloops}
The \emph{formal shifted tangent bundle} of $X$ is the formal completion of the shifted tangent bundle of $X$ at the zero section, and is denoted by $\widehat{T}_X[-1]$. The \emph{formal loop space} of $X$ is the formal completion of the loop space of $X$ at the constant loops, and is denoted by $\widehat{\cL}X$.
\end{definition}

\subsection{Quasi-constant maps}
\label{qcm}
We can now introduce one of the main objects of the paper, the \emph{derived stack of quasi-constant maps} $\Mapo{Y}{X}$ between two derived stacks $Y$ and $X$.  

Let $A$ be a finite abelian group isomorphic to a product of groups of roots of unity
\begin{equation}\label{eq:A}
 A \simeq \prod_{i=1, \ldots, r} \mu_{n_i}
\end{equation} 
For each map
$$
\alpha: Y \to [\Spec k/A] \to X
$$
let $(\Map{Y}{X})^{(\alpha)}$ be the connected component of $\Map{Y}{X}$ containing $\alpha$.
\begin{definition}
\label{quasiconstant}
 The \emph{derived stack of quasi-constant maps} is the union \begin{equation*}
\Mapo{Y}{X}:=\bigcup_{\alpha}\Map{Y}{X}^{(\alpha)}
\end{equation*}
\end{definition}

\begin{remark}
The structure map $Y \to \Spec k$ gives a closed embedding 
$$
X \simeq \Map{\Spec k}{X} \to \Map{Y}{X}
$$
which factors through $\Mapo{Y}{X}$. 
\end{remark}  
\begin{remark}
Definition \ref{quasiconstant} is designed to work in the setting where the target is the quotient   $[X/G]$ of a variety by the action of a (reductive) group, which is the framework we place ourselves in in this paper. Though adequate for what we plan to do in this article, Definition \ref{quasiconstant} is somewhat ad hoc. A more conceptual definition can be obtained by considering maps which are in a precise sense \emph{of degree zero}. Consider the map
\begin{equation}
\label{dr}
\Map{Y_{dR}}{X}\to \Map{Y}{X}
\end{equation}
induced by the unit $Y\to Y_{dR}$.
We believe that in general quasi-constant maps should be definable as the image of this map, i.e. the smallest clopen subset of $\Map{Y}{X}$ such that the map (\ref{dr}) factors through it. 
As a reality check, consider the case when $X$ is a classifying stack $[\Spec k/G]$, where $G$ is a reductive algebraic group. The image of (\ref{dr}) is the stack classifying $G$-bundles admitting a flat connection. These are degree 0 semi-stable $G$-bundles; further, the smallest clopen containing them is given by degree 0 $G$-bundles, i.e. in our terminology quasi-constant maps to $BG$.   
\end{remark}
We will mostly use Definition \ref{quasiconstant} in the situation where the source $Y$ is an elliptic curve $E$ over a field $k$ of characteristic zero. 
\begin{remark}
The bundles classified by $\Bun^0_T(E)$ are exactly the degree zero $T$--bundles on $E$, as those are the ones in the connected components of the $T$--bundles whose structure group can be reduced to $A$, for $A$ as in \eqref{eq:A}.
\end{remark}
\begin{remark}
The variant of Definition \ref{quasiconstant} involving only connected components of constants maps is in general insufficient for our purposes. The issue arises from quotients $[X/T]$ that have points with finite non-trivial stabilizers.
Consider $X=\bG_m$ with an action of $T=\bG_m$ with weight $w\neq 1$. The quotient $[X/T]$ is isomorphic to $[\Spec k/\mu_{|w|}]$. In this case, a simple  equivariant elliptic cohomology computation dictates that there should be an isomorphism  
$$
\Mapo{E}{[\mathbb{G}_m/\mathbb{G}_m]} \cong E[|w|] \times [\Spec k/\mu_{|w|}]
$$
where $E[|w|]$ denotes the $|w|$-torsion points in $E$. In particular, this stack has $|w|$-many connected components.  
Definition \ref{quasiconstant} is designed precisely so as to reproduce this expected behaviour. 
\end{remark}

In the following Propositions we give sufficient conditions under which it is enough to look at connected components of constant maps.
\begin{proposition}\label{example:MapoVar}
Let $Y$ and $X$ be  derived schemes. Then the stack of quasi-constant maps $\Mapo{Y}{X}$ coincides with the union of the connected components of the constant maps. In particular, if $X$ is connected, $\Mapo{Y}{X}$ is connected. 
\end{proposition}
\begin{proof}
This is a direct consequence of the full embedding of derived schemes in derived stacks.
\end{proof}
\begin{proposition}
\label{connectedstab}
Let $X$ be a variety with an action of  an algebraic torus $T$, and let $E$ be an elliptic curve over $k$. Assume the $T$-action on $X$ is such that the $T$-orbits in $X$ have connected stabilizers. Then the stack of quasi-constant maps $\Mapo{E}{[X/T]}$ coincides with the union of the connected components of the constant maps. In particular, if $X$ is connected, $\Mapo{E}{[X/T]}$ is connected. 
\end{proposition}
\begin{proof}
We will assume for simplicity that $T$ is rank one. The general argument is a simple extension of the rank one case. This proof requires the codescent property, which will be proved in Section \ref{section:proofcodescent} as Theorem \ref{proposition:codescentnormalvar}. It follows from that  codescent result that points in $\Mapo{E}{[X/T]}$ correspond to maps whose image is contained in a single $T$-orbit in $X$, which under our assumptions is either free or a fixed point.  

Let $x \in X$ be a fixed point for the $T$-action. 
We need to consider maps from $E$ to $[x/T]$ factoring through $[\Spec k/\mu_n]$ for all positive integers $n$. 
These maps classify $T$-bundles on $E$ admitting a reduction of their structure group to $\mu_n$. But these all lie in the connected component of $\Pic(E)$ classifying degree zero bundles. 

Now consider a free orbit $O$. 
Since $T$ acts freely on $O$, we have that $[O/T]=\Spec k$. So any map to $[O/T]$ factors necessarily through the point.  This concludes the proof.
\end{proof}

An important class of examples satisfying the assumptions of Proposition \ref{connectedstab} is given by toric varieties with the standard torus action. 
\begin{corollary}
Let $X$ be a smooth toric variety equipped with standard action by the torus $T$ and let $E$ be an elliptic curve. Then the stack of quasi-constant maps $\Mapo{E}{[X/T]}$ coincides with the connected component of the constant maps.
\end{corollary}

Just as in the case of the shifted tangent bundle and of the loop space, we can consider a formal completion of the derived stack of quasi-constant maps. 
\begin{definition}
The derived stack of \emph{formal maps}, denoted by $\fMapo{Y}{X}$, is the formal completion of $\Mapo{Y}{X}$ at the constant maps 
$$
X \to \Mapo{Y}{X}
$$
\end{definition}

\begin{proposition}
Let $E$ be an elliptic curve over $k$. Then, if the target stack $\cX$ is a finitely presented variety $X$ or a quotient  stack $[X/T]$ with $X$ a variety, the stack of quasi-constant maps $\Mapo{E}{\cX}$ is $1$-geometric.
\end{proposition}
\begin{proof}
Note that $X$ is in particular finitely presented over $k$, thus this is a direct application of Theorem 5.1.1 in \cite{HLP}. 
\end{proof}
 

\subsection{Equivariant elliptic Hochschild homology}
In this section we define equivariant elliptic Hochschild homology,  which is our main object of study in this article. 
\begin{definition}
Let $G$ be a smooth reductive algebraic group and let $Y$ be a scheme.
\begin{itemize}
\item The \emph{derived stack of principal $G$-bundles} on   $Y$   is the mapping stack 
$$\Bun_G(Y):=\Map{Y}{[\Spec k/G]}$$ 
The \emph{derived stack of principal $G$-bundles of degree zero} on $Y$ is 
$$\Bun^0_G(Y):=\Mapo{Y}{[\Spec k/G]}$$
\item Let $G=T$ be an algebraic torus of rank $n$. We set 
$$\Pic(Y)_T:=\Bun_T(Y)$$
$$\Pic^0(Y)_T:=\Bun^0_{T}(Y)$$
\end{itemize}
\end{definition}
\begin{remark}
When $Y$ is connected, $\Pic^0(Y)_T$ is the connected component of the trivial rank $n$ bundle over $Y$.
\end{remark}
\begin{remark}
When $T$ is of rank $1$, we simply write $\Pic(Y)$ in place of $\Pic_T(Y)$ and  $\Pic^0(Y)$ in place of $\Pic^0(Y)_T$.
\end{remark}

We will apply these definitions in the case when $Y$ is an elliptic curve $E$ over a field $k$ of characteristic zero. 
Let $\check{T}$ be the cocharacter lattice of $T$. 
We have decompositions  
$$
\Pic(E)_T\simeq (\mathrm{Pic}(E)\otimes_{\bZ}\check{T})\times [\Spec k/T], \quad \Pic(E)^0_T\simeq E\otimes_{\bZ}\check{T}\times [\Spec k/T]
$$
where $\mathrm{Pic}(E)$ denotes the Picard scheme of $E$. In particular, these stacks are underived. In the rank 1 case, 
$$\Pic(E)\simeq\mathrm{Pic}(E)\times[\Spec k/\bG_m]$$ 
$$\Pic^0(E)\simeq \mathrm{Pic}^0(E)\times [\Spec k/\bG_m]\simeq E\times [\Spec k/\bG_m]$$

Although we are interested in torus actions, a few steps of our argument  will depend on considering more general actions where $T$  is the product of a torus $T'$ and of a finite abelian group $A$ isomorphic to a product of groups of roots of unity
$$
A \simeq \prod_{i=1, \ldots, r} \mu_{n_i}
$$
Note that there is an equivalence
$$
\Bun_A(E) \simeq \Bun^0_A(E) \simeq \prod_{i=1, \ldots, r} 
\big ( E[n_i] \times [\Spec k/\mu_{n_i}] \big )
$$
where $E[n_i]$ denotes the $n_i$-torsion points in $E$. This induces an equivalence 
$$
\Bun^0_T(E) \simeq \big ( E\otimes_{\bZ}\check{T'}\times [\Spec k/T'] \big ) \times \big ( \prod_{i=1, \ldots, r} E[n_i] \times [\Spec k/\mu_{n_i}] \big )
$$
This stack carries a map towards its coarse moduli space
$$
\Bun^0_T(E) \to E\otimes_{\bZ}\check{T'} \times \big ( \prod_{i=1, \ldots, r} E[n_i]   \big )
$$
which we denote  $E_T$. When $T$ is a torus, fixing an isomorphism $T\simeq(\bG_m)^n$ identifies this scheme with a product of $rk(T)$ copies of $E$.

Consider a variety $X$ over $k$ with the action of an abelian group $T$ which decomposes as the product of a torus $T'$ and a finite abelian group $A$ as above. The structure map $X\rightarrow\Spec k$ induces a map
$$
p':\Mapo{E}{[X/T]}\rightarrow\Mapo{E}{[\Spec k /T]}
$$ 
We denote $p$ the composition of $p'$ with the projection $\Bun^0_T(E)\rightarrow E_T$.
\begin{definition}
\label{elliptichoch}
The $T$-\emph{equivariant elliptic Hochschild homology} of $X$ is 
$$\cHH_E([X/T]):=p_\ast\cO_{\Mapo{E}{[X/T]}} \in \Qcoh(E_T) $$
We denote by $\mathrm{HH}_E([X/T])$ the global sections of the sheaf $\cHH_E([X/T])$. We refer to $\cHH_E([X/T])$ also as the \emph{elliptic Hochschild homology} of $[X/T]$. 
\end{definition}

\begin{remark} We can say that Elliptic Hochschild homology  defines a cohomology theory for quotient stacks, in the sense that it satisfies all  the expected properties. In particular, it follows from the definition  that a $T$-equivariant map $Y \to X$ induces algebra maps
$$
\cHH_E([X/T]) \to \cHH_E([Y/T]) \quad \text{and} \quad 
 \mHH_E([X/T]) \to \mHH_E([Y/T])$$
Further, a group homomorphism $f:T\to T'$ induces a map $\phi_f:E_T\to E_{T'}$, and there is a natural map
\[
\cHH_E([X/T'])\to (\phi_f)_\ast\cHH_E([X/T])
\]
and similarly for the global sections.
Additionally $\cHH_E(-)$ satisfies a form of Mayer--Vietoris. This is a key property, which will be proved in Theorem \ref{proposition:codescentnormalvar}.
\end{remark}

In fact, we can define more generally $G$-equivariant elliptic Hochschild homology for any (reductive) algebraic group $G$. This will spelled out in Section  \ref{sect:ReductiveG}. 

\subsection{Complexified Equivariant Elliptic Cohomology}\label{subsection:PerlimGroj}
Here we present a short review of rationalized equivariant elliptic cohomology. This object was axiomatically defined by Ginzburg--Kapranov--Vasserot in \cite{GKV} and constructed by Grojnowski in \cite{groj}.
We review Grojnowski's construction following mostly the more recent exposition found in \cite{Gan14} and \cite{SchSib}. Other reviews closer in style to the original can be found in \cite{And1}, \cite{Ros3} and \cite{GreenRat}. We remark that Grojnowski's paper only sketches the construction, and that the details were carried out by Ro\c{s}u in \cite{Ros1}. 

Let $X$ be a finite $T$-CW-complex, where $T$ is a torus of rank $n$. We construct complex $T$-equivariant elliptic cohomology of $X$ as an object in the $\bZ_2$-periodic $\infty$-category of \emph{complex analytic coherent} sheaves over the complex analytic variety $E_T := E\otimes_{\bZ}\check{T}$, which is then viewed as an algebraic coherent complex via standard GAGA arguments, yielding
$$
\El_T(X) \in \Perf(E_T)_{\bZ_2}
$$
 
\begin{remark}
As $E_T$ is a smooth Noetherian underived scheme, $\Perf(E_T)_{\bZ_2}\simeq\Coh(E_T)_{\bZ_2}$. $\Coh(E_T)$ is the \emph{$\infty$-category of coherent sheaves}, i.e. the full subcategory of $\QCoh(E_T)$ spanned by bounded complexes having coherent homotopy sheaves.
\end{remark}

First, we set up some notation. 
Let $C_T^\ast(X)$ be the $T$-equivariant singular cochains on $X$, i.e. the singular cochains of the Borel construction $C^\ast(X//T)$. 
The \emph{sum}-$\bZ_2$-periodization of the $T$-equivariant singular cochains, denoted by $C_T^{\oplus,\ast}(X)$, is defined as
$$C_{T}^{\oplus,\ast}(X):=\bigoplus_{i\in\bZ}C_T^{\ast+2i}(X;\bC)$$
Analogously, we introduce the \emph{product}-$\bZ_2$-periodization as
$$C_{T}^{\prod,\ast}(X):=\prod_{i\in\bZ}C_T^{\ast+2i}(X;\bC)$$

Grojnowski's  insight is that complexified equivariant elliptic cohomology is locally controlled by the singular equivariant cohomology of loci in $X$ fixed by subgroups of $T$.

\begin{definition}\label{definition:T(x)}
Let $e$ be a closed point of $E_T$.  Let $S(e)$ be the set of subtori $T'\subset T$ such that $e$ belongs to $E_{T'}\subset E_T$. Then set 
$$T(e):=\bigcap_{T'\in S(e)}T'$$
\end{definition}
\begin{remark}\label{rmk:T(x)nonclosed}
In Section \ref{section:Tate} we will need an extension of this notion to  points of $E_T$ which are not necessarily closed. Let $x\in E_T$ be any point.
The subgroup $T(x)$ of $T$ associated to $x$ is the smallest subgroup of $T$ such that $E_{T(x)}$ contains the closure of $x$, $\overline{\{x\}}$, i.e. 
$$T(x):=\bigcap_{K\subset T|\overline{\{x\}}\subset E_K} K$$
Moreover, we also set 
$$T'(x):=T/T(x)$$
\end{remark}

\begin{remark}
Note that, if $x$ is a non-closed point, then $\mathrm{rk}(T'(x))\leq\mathrm{rk}(T)-1$. This will be relevant in our inductive proof of the comparison theorem between our theory and Grojnowski's equivariant elliptic cohomology of the analytification,  Theorem\ref{theorem:comparisongeneral}.
\end{remark}
\begin{remark}
Fixing an isomorphism $T\cong \bG_m^n$ induces an isomorphism $E_T \cong E^n$. Under this identification, we can characterize the subgroup $T(e)$ for a closed point $e$ as follows. Let $$e=(e_1,\dots, e_n) \in E^n$$ and assume that
\begin{itemize}
\item $e_{i_1}, \ldots, e_{i_l} \in E$ are torsion  with order $|e_{i_j}|=n_j$
\item For all $k \notin \{i_1, \ldots, i_l\}$, 
$e_k \in E$ is not torsion 
\end{itemize}
Then, up to shuffling the factors,  
$$
T(e) = \big ( \prod_{i=i_1}^{i_l}\mu_{n_i} \big ) \times (\bG_m)^{n-l} \subset T
$$ 
\end{remark}

We are now ready to construct complexified equivariant elliptic cohomology. First of all, recall that $C_T^{\oplus,\ast}(X)$ is a module over $C_T^{\oplus,0}(\ast)$. This is a formal commutative dg-algebra concentrated in degree zero, and in particular we have an equivalence  
$$C_T^{\oplus,0}(\ast)\simeq H_T^{\oplus,0}(\ast)=\bC[u_1, \dots, u_{\mathrm{rk}(T)}]$$ 
Equivalently, we can regard the module $C_T^{\oplus,\ast}(X)$ as an object in $\Perf(E_T)_{\bZ_2}$ (as $X$ is a finite $T$-CW-complex).
Let us call $\mathcal{H}_T(X)$ this object. By definition $\cH_T(X)$ is a sheaf of algebras over $\Spec C_T^{\oplus,\ast}(\ast)\simeq\ft_\bC$, which is the complexified Lie algebra of the torus $T$. Denote by  $\mathcal{H}_T^{\an}(X)$ its analytification, i.e. the coherent analytic sheaf 
$$\mathcal{H}_T^{\an}(X):=\mathcal{H}_T(X)\otimes_{\cO_{\ft_\bC}}\cO_{\ft_\bC}^\an$$ 

There is a quotient map
$$\mathrm{exp^2}:\bA^n_\bC\to E_T$$
which is an isomorphism if restricted to sufficiently small analytic disks $U$ in $E_T$. Let us call $\mathrm{log^2}$ its local inverse. Moreover, the group structure on $E_T$ induces translation maps
\begin{align*}
\tau_e: & E_T\to E_T \\ & f\mapsto fe
\end{align*}
for all closed points $e$ in $E_T$ (we use multiplicative notation for the group operation on $E_T$).
Then, for a closed point $e\in E_T$ and a sufficiently small analytic neighbourhood $U_e$ of $e$, so that $U_1\subset \tau_{e^{-1}}(U_e)$, we set
$$\El_T^{\an}(X)|_{U_e}:=(\tau_e\circ\mathrm{exp^2})_{\ast}\cH_T^{\an}(X^{T(e)})|_{\mathrm{log^2}(e^{-1}U_e)}$$
As summarized in \cite{Gan14}, these open sets cover $E_T$ and transition isomorphisms between $\El_T^{\an}(X)|_{U_e}$ and $\El_T^{\an}(X)|_{U_e'}$ can be defined for all closed points $e$ and $e'$ in terms of the localization theorem in equivariant cohomology. These isomorphisms satisfy the cocycle identities and thus give rise to a complex-analytic sheaf denoted by $\El_T^{\an}(X)$.

We reserve the name $\El_T(X)$ for the algebraic sheaf obtained from $\El_T^{\an}(X)$ via GAGA.
\begin{remark}
Grojnowski's original construction involves singular cohomology rather than singular cochains. His construction can be obtained from ours by taking the cohomology sheaves of $\El_T(X)$.
\end{remark}

The completions of the periodic version of Grojnowski's sheaf over closed points $x$ of $E_T$ are given by a \emph{product}-$\bZ_2$-periodization of $T$-equivariant singular cohomology:
$$\El_T(X)_{\widehat{x}}\simeq C_{T}^{\prod,\ast}(X^{T(x)})\simeq C_{T}^{\oplus,\ast}(X^{T(x)})\otimes_{C_T^{\oplus,0}(\ast)}\cO_{E_T,\widehat{x}}$$
where $\cO_{E_T,\widehat{x}}$ is a module over $C_T^{\oplus,0}(\ast)\simeq\cO(\ft_\bC)$ via the completed multiplication map $$\widehat{\mu}_x:E_{T,\widehat{1}}\simeq E_{T,\widehat{x}}$$
and the identification $E_{T,\widehat{1}}\simeq \ft_{\bC,\widehat{0}}$.

\subsection{Adelic descent}
\label{sec:adelic}
In the last section of this paper we make extensive use of adelic descent theory for $n$-dimensional schemes. This theory was first introduced by Parshin \cite{ParAd} and Beilinson \cite{BeilAd}, the reader can consult \cite{HubAd} and \cite{MorAd} for a review. We will follow the modern treatment given in \cite{Groech}, which is  the main reference for the short reminder below. 

Let $X$ be a Noetherian scheme. For two points $x$ and $y$ we say $x\geq y$ if $y\in\overline{\{x\}}$.
We let $|X|_k$ denote the set of \emph{$k$-chains} on $X$, i.e. sequences of $k+1$ ordered points $(x_0\geq\dots\geq x_k)$ in $X$. If $k=0$, we equivalently write $|X|=|X|_0$.
Finally, given a subset $T\subset|X|_k$, we set
$$_{x}T:=\{\Delta\in |X|_{k-1}|(x\geq \Delta)\in T\}$$

This notation allows us to define sheaves of ad\`eles on $X$ for a subset $T\subset|X|_{k}$. The ad\`eles are the unique family of exact functors $$\bfA_{X}(T,-):\QCoh(X)\to\Mod_{\cO_X}$$ satisfying the following properties:
\begin{itemize}
	\item $\bfA_{X}(T,-)$ commutes with directed colimits;
	\item if $\cF$ is coherent and $k=0$, $\bfA_{X}(T,\cF)=\prod_{x\in T}\lim_{r\geq 0}\tilde{j}_{rx}\cF$;
	\item if $\cF$ is coherent and $k>0$, $\bfA_{X}(T,\cF)=\prod_{x\in|X|}\lim_{r\geq 0}\bfA_{X}(_{x}T,\tilde{j}_{rx}\cF)$.
\end{itemize}

In the above, $\tilde{j}_{rx}$ denotes the functor $j_{rx\,\ast}j_{rx}^{\ast}$, where 
$$j_{rx}:\Spec \cO_{X,x}/\mathfrak{m}_x^r\to X$$
is the canonical immersion of an $r$-thickening of the point $x$. Here $\cO_{X,x}$ is the local ring at $x$ and $\mathfrak{m}_x$ is its maximal ideal.

\begin{definition}
The global sections $\Gamma(X, \bfA_{X}(T,\cF))$ are denoted by $\bA_{X}(T,\cF)$ and are called the groups of ad\`eles. 
\end{definition}

The sets $|X|_k$ can be assembled into a simplicial set: face and degeneracy maps are defined, respectively, by removing or repeating a point in a chain.  
We denote this simplicial set  by $|X|_{\bullet}$. In particular,  for any $T_\bullet\subset|X|_\bullet$, the sheaves of ad\`eles assemble into a cosimplicial sheaf of $\cO_{X}$--modules $\bfA_{X}(T_\bullet,\cF)$. If $T_\bullet$ coincides with $|X|_\bullet$ we denote this cosimplicial sheaf by $\bfA_{X}^{\bullet}(\cF)$ and its global sections by $\bA_{X}^{\bullet}(\cF)$. If $\cF=\cO_X$, we denote the cosimplicial sheaves and groups of ad\`eles by $\bfA_{X}^{\bullet}$ and $\bA_{X}^{\bullet}$ respectively.

Similarly, there is a cosimplicial sheaf given by products of ``local'' ad\`eles
$$[n]\mapsto\prod_{\Delta\in|X|_n}\bfA_{X}(\Delta,\cF)$$
Theorem 2.4.1 in \cite{HubAd} tells us that the natural inclusion of the full ad\`eles into the product of local ad\`eles respects the cosimplicial structures.

We conclude this section with two theorems that allow to reconstruct sheaves from their adelic decomposition.

\begin{theorem}[Theorem 3.1 in \cite{Groech}]
Let $X$ be a Noetherian scheme. Then there is an equivalence of symmetric monoidal $\infty$-categories
$$\Perf^{\otimes}(X)\simeq\Tot\Perf^{\otimes}(\bA_{X})$$
\end{theorem}

The following theorem due to Beilinson appears as Theorem 1.16 in \cite{Groech}.
\begin{theorem}[Beilinson \cite{BeilAd}]
Let $\cF$ be a quasi-coherent sheaf on $X$. The augmentation $$\cF\to\bfA_{X}^{\bullet}(\cF)$$ is a resolution of $\cF$ by flasque $\cO_X$-modules. In particular, the totalization of the ad\`eles $\Tot \bA_{X}^{\bullet}(\cF)$ computes the cohomology of $\cF$.
\end{theorem}

The objects we will consider in Section \ref{section:Tate} belong rather to the \emph{$\bZ_2$-periodic} categories of perfect complexes. The arguments made by Groechenig in \cite{Groech} also hold in this context, leading to completely parallel statements involving the $\bZ_2$-periodic categories.
\begin{remark}
Let us remark that in Section \ref{section:Tate} we will use a variant of Beilinson's theorem for perfect complexes, i.e. that the adelic descent data of a perfect complex, computed as in Beilinson's definition where we interpret the operations in the derived sense, recovers the original perfect complex after totalization. This follows from Theorem 3.1 in \cite{Groech}.
\end{remark}

\section{Codescent for quasi-constant maps}
\label{section:codescentpropstatement}
\label{section:proofcodescent}
In this section we prove that the stack of quasi-constant maps is Zariski local on the target. This behaviour is in sharp contrast with the full mapping stack, where locality on the target is essentially never satisfied.  
We will  give a proof of this statement in the case when the source is an elliptic curve, which is the case that is most relevant for our applications, but we will also comment on  extensions of our results to more general settings (see Section \ref{sec:general}). One of the ingredients in our argument is a simple criterion that allows us to detect when an open immersions of geometric stacks is an equivalence, Proposition \ref{proposition:pointwisecriterion2} below.

We start by proving a few simple general properties of the stack of quasi-constant maps.

Let $S$ be a derived stack. Consider the functor 
$$
\Mapo{S}{-}: \dStk \to \dStk
$$ 
If $f:X \to Y$ is a map in $\dStk$, we denote by $r_f$ the induced map
$$
r_f: \Mapo{S}{X} \to \Mapo{S}{Y}
$$
\begin{lemma}\label{remark:Mapolimitpreserving}
The functor  
$$
\Mapo{S}{-}: \dStk \to \dStk
$$  preserves limits. 
In particular, if $f: X \to Y$ is a map in $\dStk$,  there is an equivalence of simplicial objects in $\dStk$
$$
\Mapo{S}{\check{\mathrm{C}}(f)} \simeq  \check{\mathrm{C}}(r_f) 
$$
\end{lemma}
\begin{proof}
The functor from pointed stacks to stacks sending a pair $(F, x)$ to the connected component of $x$ preserves limits. Thus the statement follows from the analogous statement for $\Map{S}{-}$, which is obvious. The second part of the claim is a formal consequence of the first one.
\end{proof}
 
\begin{lemma}
\label{lemma:openinclusion}
The functor  
$$\Mapo{S}{-}: \dStk \to \dStk$$
preserves both open and closed immersions of derived stacks. That is, if $i:Y \to X$ is an open (resp. closed) immersion of derived stacks, then 
$$r_i: \Mapo{S}{Y}\rightarrow\Mapo{S}{X}$$
is an open (resp. closed) immersion of derived stacks.
\end{lemma}
\begin{proof} 
We prove the statement for open immersions, as the proof for closed immersions is the same. 
We need to show that for every affine scheme $\Spec A$ and for every map $$\Spec A\rightarrow\Mapo{S}{X}$$ the pullback map $\Mapo{S}{Y}\times_{\Mapo{S}{X}}\Spec A\rightarrow\Spec A$ is an open  immersion of derived schemes. 

Set $Z := \Mapo{S}{Y}\times_{\Mapo{S}{X}}\Spec A$. Then by the universal property of the mapping stack, we  obtain a pullback diagram 
$$
\xymatrix{
Z\times S \ar[d] \ar[r] & Y \ar[d] \\
\Spec A\times S \ar[r]  & X
}
$$
Since the map $Y\rightarrow X$ is an open immersion, the map $Z\times S \rightarrow \Spec A\times S$ must be an open  immersion. This map is the identity on $S$, so we conclude that $Z\rightarrow \Spec A$ is an open immersion; in particular $Z$ is necessarily a derived scheme. 
\end{proof}

\begin{proposition}[Point-wise criterion]\label{proposition:pointwisecriterion2}
Let $X$ be an $n$-geometric derived stack and let 
$$
\phi:\coprod_{\alpha\in I}U_\alpha\rightarrow X
$$ 
be a coproduct of open immersions of derived stacks. Assume that for every field extension $K$ of $k$ and any map $f:\Spec K\rightarrow X$ there exists a lift 
$$
\xymatrix{
& \coprod_{\alpha\in I} U_\alpha \ar[d]^-\phi \\
\Spec K\ar[r]^{f}\ar@{-->}[ur] & X
}
$$
Then the map $\phi$ is an effective epimorphism.
\end{proposition}
\begin{proof} 
Since $X$ is $n$-geometric, the map $f$ factors through one of the affine schemes $$\Spec A\rightarrow X$$ which compose the chosen atlas of $X$ (up to trading $K$ for a field extension). Open immersions of derived stacks are stable under base change, and thus the base change of $\phi$ along $\Spec A\rightarrow X$
$$
\Big ( \coprod_{\alpha\in I} U_\alpha \Big ) \times_X \Spec A \simeq  \coprod_{\alpha\in I} \Big (  U_\alpha  \times_X \Spec A  \Big ) \to \Spec A  
$$
is also a coproduct of open immersions. The equivalence above is a consequence of the universality of colimits in $\infty$-topoi.
Each summand $V_\alpha:=U_\alpha  \times_X \Spec A$ is an open substack of $\Spec A$, and is therefore a derived scheme. 
Thus we can reduce to proving the claim  when $X = \Spec A$  is an affine derived scheme and 
$$
\coprod_{\alpha\in I} V_\alpha \to X = \Spec A
$$
is a disjoint union of open subschemes. Up to refining the cover $\{V_\alpha\}_{\alpha \in I}$, by taking affine open covers of each scheme $V_\alpha$, we can also assume that the $V_\alpha$-s are affine. Set $V_\alpha=\Spec A_\alpha$.

The existence of lifts in the affine situation is equivalent to the statement that the collection of maps of simplicial $k$-algebras $$\{A\rightarrow A_\alpha\}_{\alpha\in I}$$ is a Zariski cover of the simplicial $k$-algebra $A$, i.e.
\begin{itemize}
\item all the  $k$-algebras $A_\alpha$ are localizations of the algebra $A$ at some elements $a_\alpha\in \pi_0 A$;
\item the collection of the elements $\{a_\alpha\}_{\alpha\in I}$ generates $\pi_0 A$.
\end{itemize} 
 
But this implies that the map $\coprod_{\alpha\in I} \Spec A_\alpha\rightarrow\Spec A$ is an effective epimorphism.
\end{proof}
\subsection{Quasi-constant maps to varieties}
\label{subsection:codescentnoaction}
Let $X$ be a variety and $E$ an elliptic curve over $k$. 
In this section we prove the following statement:
\begin{proposition}
\label{affinization}
The morphism 
$$
\Mapo{\mathrm{Aff}(E)}{X} \to \Mapo{E}{X}
$$
induced by the unit $E \to \mathrm{Aff}(E)$ of the affinization is an equivalence of derived stacks.
\end{proposition}
The codescent property for $\Mapo{E}{X}$ will follow immediately from Proposition \ref{affinization}. 

\begin{lemma}
\label{extension}
Let  $S \subset E$ be an affine open subset. Let $U$ be an affine variety, and fix a locally closed embedding $U \subset \mathbb{P}^n$. Denote by $\overline{U}$ the closure of $U$. Then there are natural monomorphisms of stacks
$$
\Map{S}{U} \stackrel{\alpha}\longrightarrow \Map{E}{\overline{U}} \stackrel{\beta}\longrightarrow \Map{E}{\bP^n}
$$
That is, for every affine scheme $Y$ the induced  maps of sets
$$
\pi_0 \mathrm{Map}(S\times Y,U)   \stackrel{\pi_0\alpha} \to \pi_0 \mathrm{Map}(E\times Y, \overline{U})  \stackrel{\pi_0\beta}  \to \pi_0\mathrm{Map}(E\times Y, \bP^n)
$$
are injective. Restricting to quasi-constant maps yields monomorphisms
$$
\Mapo{S}{U} \stackrel{\alpha}\longrightarrow \Mapo{E}{\overline{U}}  \stackrel{\beta}\longrightarrow \Mapo{E}{\bP^n}
$$
\end{lemma}
\begin{remark}
The notion of monomorphism we refer to in Lemma \ref{extension} is the notion of monomorphism in an $\infty$-category appearing in \cite[p.~575]{HTT}.
\end{remark}

\begin{proof}[Proof of Lemma \ref{extension}]
The inclusion  $\overline{U} \subset \bP^n$ determines a map 
$\Map{E}{\overline{U}} \stackrel{\beta}\longrightarrow \Map{E}{\bP^n}$ which has the desired properties. Let us define the map $\alpha$. Let $T$ be a  proper and separated derived scheme. As $E$ is a discrete one-dimensional scheme, the valuative criterion of properness\footnote{For a reference on the valuative criterion of properness in the derived setting, see for instance {\tt https://www.preschema.com/lecture-notes/kdescent/lect6.pdf}.} implies that there is an equivalence $
\mathrm{Map}(E, T) \stackrel{\simeq}\to \mathrm{Map}(S, T)
$.  Defining $\alpha$ requires defining maps
$$
\alpha_Y: \mathrm{Map}(S \times Y, U) \to \mathrm{Map}(E \times Y, \overline{U})
$$
for every $Y \in \mathrm{dAff}$, that are natural in $Y$. We define $\alpha_Y$ as the composition 
 $$
\mathrm{Map}(S \times Y, U) \stackrel{(a)}\to \mathrm{Map}_{\mathrm{dSt}/Y}(S \times Y, U \times Y) \stackrel{(b)}\to 
\mathrm{Map}_{\mathrm{dSt}/Y}(E \times Y, \overline{U} \times Y) \stackrel{(c)}\to \mathrm{Map} (E \times Y, \overline{U})
$$
where  
\begin{itemize}
\item  $\mathrm{Map}_{\mathrm{dSt}/Y}(-,-)$ denotes the mapping space in the over-category $\mathrm{dSt}/Y$ 
\item on connected components, the map (a) is the assignment
$$(S \times Y \stackrel{f} \to U) \mapsto  (S \times Y \stackrel{f \times \mathrm{pr}_Y}\to U \times Y)$$
\item the map (b) is the morphism on mapping spaces given by the valuative criterion of properness, relative to the base scheme $Y$, as $E \times Y \to Y$ is a discrete curve over $Y$, and $\overline{U} \times Y \to Y$ is proper
\item the map (c) is induced by the projection $\overline{U} \times Y \to \overline{U}$
\end{itemize}
The fact that $\alpha$ is natural in $Y$ is clear. Also, it is easy to see that $\alpha$ induces section-wise injections on connected components. As for the last statement, it follows from the fact that $\alpha$ and $\beta$ preserve constant maps. Indeed, as the image of a connected stack under any map is connected, the maps $\alpha$ and $\beta$ restrict to maps between the connected components of the constant maps, and by Proposition \ref{example:MapoVar} the stack of quasi-constant maps with target a variety is the connected component of the constants.
\end{proof}

\begin{remark}
Let us make some more comments on the map $\alpha$ defined in Lemma \ref{extension}. 
The map $\alpha$ can be factored as 
$$
\xymatrix{
\Map{S}{\overline{U}}  \ar[dr]^-\simeq & &  \Mapo{S}{\overline{U}} \ar[dr]^-\simeq & \\
\Map{S}{U} \ar[r]^-\alpha \ar[u] & \Map{E}{\overline{U}}  & \Mapo{S}{U}  \ar[r]^-\alpha \ar[u]& \Mapo{E}{\overline{U}}
}
$$
Let us focus on then diagram on the left, as the one on the right is just obtained by restricting to quasi-constant maps. The vertical arrow is induced by the inclusion $U \to \overline{U}$, and thus is an open embedding as explained in Lemma \ref{lemma:openinclusion}. The diagonal arrow is an equivalence. It is the inverse of the natural map $\Map{E}{\overline{U}} \to \Map{S}{\overline{U}}$ given by restriction to $S$. As  explained in the proof of Lemma \ref{extension}, the fact that this map is an equivalence follows from the valuative criterion for properness.
\end{remark}

\begin{lemma}
\label{affopen}
Let $\mathrm{Aff}(E)$ be the affinization of $E$. Then the map
$$
\Mapo{\mathrm{Aff}(E)}{X} \to \Mapo{E}{X} 
$$
is an open embedding.
\end{lemma}
\begin{proof}
Recall from Section \ref{section:preliminaries}, that there is an equivalence 
$$
\mathrm{Aff}(E) \simeq \mathrm{Aff}(S^1)
$$
This implies  that $\Map{\mathrm{Aff}(E)}{X}$ is equivalent to the derived loop space $\cL X$; in particular, the stack $\Map{\mathrm{Aff}(E)}{X}$ is connected, and thus there is an identification
$$\Map{\mathrm{Aff}(E)}{X} = \Mapo{\mathrm{Aff}(E)}{X}$$  
Now this also implies that $\Mapo{\mathrm{Aff}(E)}{X}$ satisfies Zariski codescent on $X$. Indeed, this is easily proved for the loop space  $\cL X$; a reference  is, for instance, Lemma 4.2 of \cite{BZNLoopConn}. 

Consider an affine open cover $\{U_i\}_{i \in I}$ of $X$. By Lemma \ref{lemma:openinclusion}, for every $i \in I$ the vertical arrows in the commutative diagram below are open inclusions 
$$
\xymatrix{
\Mapo{\mathrm{Aff}(E)}{X} \ar[r] & \Mapo{E}{X} \\
\Mapo{\mathrm{Aff}(E)}{U_i} \ar[ur]^-\subset \ar[r]^-\simeq \ar[u] & \Mapo{E}{U_i} \ar[u]
}
$$
Further, the bottom arrow is an equivalence, by the universal property of the affinization. It follows that we have an open embedding 
$\Mapo{\mathrm{Aff}(E)}{U_i} \to \Mapo{E}{X}$. Thus the map from the realization of the \v{C}ech nerve of the open substacks $\Mapo{\mathrm{Aff}(E)}{U_i}$  
$$
\Mapo{\mathrm{Aff}(E)}{X} \simeq |\Mapo{\mathrm{Aff}(E)}{U_i} | \longrightarrow \Mapo{E}{X} 
$$
is also an open embedding.
\end{proof}
 
 Next we show Proposition \ref{affinization} in the case when $X$ is the projective space, as a stepping stone to the proof in the general case.
 \begin{lemma}
 \label{Pn}
 The morphism 
$$
\Mapo{\mathrm{Aff}(E)}{\bP^n} \to \Mapo{E}{\bP^n}
$$
 induced by the unit map $E \to \mathrm{Aff}(E)$, is an equivalence of derived stacks.
 \end{lemma}
 \begin{proof}
Let $\phi_{\cO(1)}: \mathbb{P}^n \to [\Spec k/\bG_m]$ be the classifying map of the bundle $\cO(1)$ over $\mathbb{P}^n$. By evaluating the composition 
  $$
  \Mapo{E}{\bP^n} \times \Map{\bP^n}{[\Spec k/\bG_m]} \longrightarrow \Map{E}{[\Spec k/\bG_m]} = \underline{\mathrm{Pic}}(E)
  $$
  at the point $\phi_{\cO(1)}$ in the second factor, we obtain a map of stacks 
$$
\Mapo{E}{\bP^n} \to \underline{\mathrm{Pic}}(E)
$$
As connected stacks map to connected stacks, this map must factor through the inclusion $$\underline{\mathrm{Pic}}^0(E) \subset \underline{\mathrm{Pic}}(E)$$This implies that for every map $f: E_{K} \to \bP^n$ parametrized by a point of 
$
\Mapo{E}{\bP^n}$, the bundle $f^*\cO(1)$ has degree 0. Further $f^*\cO(1)$ must have non-trivial global sections (as it is the pullback of a very ample bundle). These two properties imply that $f^*\cO(1) \simeq \cO_E$. 

By Lemma \ref{affopen}, the statement we need to prove can be checked via Proposition \ref{proposition:pointwisecriterion2}. That is, we need to show that every map $f:E_{K} \to \mathbb{P}^n$ factors through some affine open subset $U \subset \mathbb{P}^n$. To show this, it is enough to check the set-theoretic condition that $f$ induces a constant map between the set of geometric points of $E$ and the set of geometric points of $\mathbb{P}^n$. Then it will be enough to choose as $U$ an affine open neighbourhood of $f(p)$, where $p \in E$ is any geometric point. The fact that $f$ is constant on geometric points follows immediately from the  fact that  $f^*\cO(1) \simeq \cO_E$. Indeed, the sections of  $\cO_E$ are constants, and therefore do not distinguish points. 
\end{proof}

 \begin{proof}[Proof of Proposition \ref{affinization}]
As in the case of $\bP^n$, we start by observing that by Lemma \ref{affopen} the natural map
 $$
\Mapo{\mathrm{Aff}(E)}{X} \to \Mapo{E}{X}
$$
 is an open embedding. Thus to prove that it is an equivalence we can use the point-wise criterion. Consider a map 
 $ \, 
 f: E_K \to X
 \, $ 
corresponding to a closed point $\Spec K\to\Mapo{E}{X}$. We need to show that $f$ factors as 
$$
\xymatrix{\mathrm{Aff}(E)_K \ar[dr]^-{g} & \\
E_K \ar[u] \ar[r]^-f & X }
$$
for some $g: \mathrm{Aff}(E)_K \to X$. Let $U \subset X$ be an affine open subset such that 
$S=E_K \times_U X$ is non-empty. Note $S \subset E_K$ is an affine open subset. Now fix a locally closed embedding $U \subset \mathbb{P}^n$, and let $\overline{U}$ be the closure of $U$. As $\mathbb{P}^n$ is proper there is a unique map $h: E_K \to \mathbb{P}^n$ which makes the following diagram commute
\begin{equation}
\label{diacom}
\begin{gathered}
\xymatrix{
U \ar[r] & \mathbb{P}^n \\
S \ar[r] \ar[u]^-{f|_S} & E_K \ar[u]_-h
}
\end{gathered}
\end{equation}
In fact more is true, namely $S$ is the truncation of the fiber product: 
$ \, 
S \cong t_0(U \times_{\bP^n} E_K)
$. 
By Lemma \ref{extension}, the map $h$ is parametrized by a closed point in 
$\Mapo{E}{\bP^n}$. Thus, by  Lemma \ref{Pn}, $h$ factors through the affinization of $E_K$. We can complete diagram (\ref{diacom}) to the following commutative diagram
$$
\xymatrix{
\overline{U} \ar[dr] & & \\
U \ar[r] \ar[u] & \mathbb{P}^n & \\
S \ar[r] \ar[u]^-{f|_S} & E_K \ar[u]^-h \ar[r] & \mathrm{Aff}(E)_K \ar[ul] \ar@/_2pc/[uull]_-{h'}
}
$$
The commutativity of the external  triangle with edge $h'$ follows  from the fact that the image of $h$  must be contained in the closure of the image of $f$. In order to conclude we need to show that $h'$ factors through $U$
$$
\xymatrix{
\overline{U}  & \\
U  \ar[u] &  \mathrm{Aff}(E)_K \ar[ul]_-{h'} \ar[l]_-{g}
}
$$

This is easy to check. Indeed as $U$ and $\overline{U}$ are schemes, we can replace  $\mathrm{Aff}(E)_K$ with $S^1_K$ and think in terms of loop spaces. It is a consequence of Zariski codescent for the loop space that if $\gamma: S^1_K \to \overline{U}$ is a loop, and $U \subset \overline{U}$ is an affine open subset, then the following two facts are equivalent
\begin{enumerate} 
\item The fiber product $S^1_K \times_{\overline{U}}U$ is non-empty 
\item The loop $\gamma$ factors through $U$
\end{enumerate}
By construction, the first condition is satisfied in our case. Thus $h'$ factors through $U$, and this implies that $h$ also factors through $U$. As a consequence $f$   factors also through $U$, and this concludes the proof.
\end{proof}

Proposition \ref{affinization} has the following useful consequence.

\begin{corollary}\label{corollary:LoopstackequalsMapo}
Let $X$ be a variety, and let $\Zar{X}$ denote its small Zariski site. 
\begin{enumerate}
\item The assignment mapping an open subset 
$
U \in \Zar{X}$ to the stack of quasi-constant maps $\Mapo{E}{U} \in \dStk$  
defines a cosheaf on $\Zar{X}$
$$
\Mapo{E}{-}: \Zar{X} \rightarrow\dStk
$$
\item The natural map $\cL U \to \Mapo{E}{U}$ defines an equivalence of $\dStk$-valued cosheaves on $\Zar{X}$
$$
\cL(-) \simeq \Mapo{E}{-}:  \Zar{X} \rightarrow\dStk
$$
\end{enumerate}
\end{corollary}
\begin{proof}
This an immediate consequence of Proposition \ref{affinization}.  Indeed by Proposition \ref{affinization} for every $U \in \Zar{X}$ there is an equivalence 
$$
\cL(U) \simeq \Mapo{E}{U}
$$
which is natural for maps in $\Zar{X}$. This implies that $\Mapo{E}{-}$ is equivalent to $\cL(-)$ as $\dStk$-valued precosheaves on $\Zar{X}$. As $\cL(-)$ is a cosheaf, it follows that $\Mapo{E}{-}$ is also a cosheaf. 
\end{proof}
 
\subsubsection{Some generalizations of Proposition \ref{affinization}}
\label{sec:general}
We formulated Proposition \ref{affinization} for an elliptic curve $E$, as this is the case we will be interested in the remainder of the paper. However the statement holds more generally for any smooth and proper curve $C$ over $k$, namely the unit map $C \to \mathrm{Aff}(C)$ induces  an equivalence 
$$
\Mapo{\mathrm{Aff}(C)}{X} \to \Mapo{C}{X}
$$
The proof we have given for the case $C=E$ extends without variations to this more  general setting.  

Corollary \ref{corollary:LoopstackequalsMapo} also generalizes to the case of a general smooth and proper curve $C$.   
In characteristic 0, there is an equivalence between $\mathrm{Aff}(C)$ and the affinization of the Betti stack of a wedge of $g$ circles $S_g$, where $g$ is the genus of $C$
$$
\mathrm{Aff}(C) \simeq  \mathrm{Aff}(S_g) \simeq k[\varepsilon_1, \ldots, \varepsilon_g]
$$  
where 
\begin{itemize}
\item $k[\varepsilon_1, \ldots, \varepsilon_g]$ denotes the square-zero extension of $k$ by $g$ generators with $\mathrm{deg}(\varepsilon_i)=1$
\item the last equivalence follows from the fact that, in characteristic zero, the cdga-s 
$$\mathrm{Hom}(\cO_C, \cO_C) \quad \text{and} \quad C^*_{sing}(S_g)$$ 
are both formal and therefore quasi-equivalent to their cohomology.
\end{itemize}
 This implies in particular that there are equivalences 
 $$
 \Mapo{C}{X} \simeq \Mapo{S_g}{X} \simeq \cL X \times_X \ldots \times_X \cL X
 $$
 where the last  one follows from the presentation of $S_g$ as an iterated push-out  
 $$
 S_g \simeq S^1 \coprod_{pt} S^1 \coprod_{pt} \ldots \coprod_{pt} S^1
 $$
A simple  observation which we have used repeatedly is that $\cL X$ is local with respect to the Zariski topology on $X$. The argument in Lemma 4.2 of \cite{BZNLoopConn} immediately extends to show that $\Mapo{S_g}{X}$ also defines a cosheaf on $\Zar{X}$. It follows that the conclusions of Corollary \ref{corollary:LoopstackequalsMapo} apply to the case of a general smooth and proper curve, and in particular $\Mapo{C}{X}$ defines a cosheaf on $\Zar{X}$. 

It is natural to ask whether Proposition \ref{affinization}  and Corollary \ref{corollary:LoopstackequalsMapo} are in fact general features of quasi-constant maps, beyond the curve case. As this will not play any role in the sequel, we   leave these as open questions without attempting to answer them in this article.
\begin{question}
Let $X$ be a scheme and let $S$ be any derived stack.
\begin{enumerate}
\item Under what assumptions on $S$ does $\Mapo{S}{X}$ define a cosheaf of stacks on $|X|_{Zar}$? 
\item Under what assumptions on $S$  is there an equivalence of stacks
$$
\Mapo{\mathrm{Aff}(S)}{X} 
 \simeq  \Mapo{S}{X}  \, \, ?
$$
\end{enumerate}
\end{question}

\begin{remark}
It should be possible to extend Proposition \ref{affinization}  to the setting where the target $X$ is an algebraic space satisfying suitable properties. The key observation should be that, although in general  loops are not local for the \'etale topology, they are when the target is a scheme or an algebraic space. As this extra generality is not essential for our intended applications, we will not pursue this  any further.
\end{remark}

\subsection{Quasi-constant maps to global quotient stacks}
\label{label:globalquotient}
Let $X$ be a variety over $k$, and assume that $X$ carries an action of an algebraic torus $T$.  
\begin{definition} The small \emph{$T$-equivariant Zariski site} of $X$, which we denote by   
$\Zar{X}^T$, is   the site having 
\begin{enumerate}
\item as objects, schemes $U$ equipped with an action of $T$, and a $T$-equivariant open immersion  $U \to X$;
\item as morphisms, $T$-equivariant open immersions $V \to U$ over $X$;
\item  as covering families, jointly surjective families of $T$-equivariant immersions.
\end{enumerate}
\end{definition}
The main result of this Section is Theorem \ref{proposition:codescentnormalvar} below, which generalizes the first part of Corollary \ref{corollary:LoopstackequalsMapo} to the setting of stacks that are global quotients of varieties by a $T$-action. We remark that the second part of Corollary \ref{corollary:LoopstackequalsMapo} fails in the presence of a $T$-action, and this is a key feature differentiating  elliptic Hochschild homology from ordinary Hochschild homology. Another important difference with Corollary \ref{corollary:LoopstackequalsMapo}  is that in the statement of Theorem \ref{proposition:codescentnormalvar} we require $X$ to be normal. The reason is that the argument we will give relies on Sumihiro's Theorem \cite{sumihiro1974equivariant}, which does not hold in general without the normality assumption. 
  
\begin{theorem}\label{proposition:codescentnormalvar}
Assume that $X$ is normal. Then the assignment mapping  
$U \in \Zar{X}^T$ to the stack of quasi-constant maps $$\Mapo{E}{[U/T]} \in \dStk$$  
defines a cosheaf on $\Zar{X}^T \, ,$
$$
\Mapo{E}{[-/T]}: \Zar{X}^T \rightarrow\dStk
$$
\end{theorem}
 If $X$ satisfies the conclusion of Theorem \ref{proposition:codescentnormalvar}, we say that $
\Mapo{E}{[X/T]}$ satisfies $T$-equivariant Zariski codescent on $[X/T]$. In its main lines the proof of Theorem \ref{proposition:codescentnormalvar} closely parallels the argument  given in the   Section \ref{subsection:codescentnoaction} in non-equivariant case. However there are a few minor  subtleties that arise when taking into account the $T$-action.  
\begin{remark}
Let us stress that Theorem \ref{proposition:codescentnormalvar} fails if we consider finer topologies such as the \'etale or smooth topology on $[X/T]$. Consider for instance the case when $T=\bG_m$ and $X = \Spec k$. Then $\Spec k \to [\Spec k/\bG_m]$ is a smooth cover, but 
$$
\Mapo{E}{\Spec k} \simeq \Spec k \to \Mapo{E}{[\Spec k/\bG_m]} \simeq \Pic^0(E)
$$
clearly is not.
\end{remark}

    We start by proving Theorem \ref{proposition:codescentnormalvar}  in two important special cases, namely when $X$ is isomorphic to an affine space or to a projective space.  Up to fixing an isomorphism $$T \cong (\mathbb{G}_m)^n$$ and applying a change of basis to $\bA^N$, we can diagonalize the  action of $T$, which can then be written in standard form as follows.   
  For every   $\lambda = 
    (\lambda_1, \ldots, \lambda_n) \in (\mathbb{G}_m)^n \cong T$,  and 
    $(z_1 \ldots z_N)  \in \bA^N$
    $$  
(\lambda_1, \ldots, \lambda_n) \cdot (z_1 \ldots z_N) =    (\prod_{i=1}^{n}\lambda_i^{w^1_i}z_1,\dots,\prod_{i=1}^{n}\lambda_i^{w^N_i}z_N)   
    $$     
for an appropriate collection of integers  $\{ w_i^j \}$, called \emph{weights}.
	\begin{lemma}
    \label{lemma:Ancodescent}
The stack of quasi-constant maps $\Mapo{E}{[\mathbb{A}^N/T]}$ satisfies $T$-equivariant Zariski codescent on $[\mathbb{A}^N/T]$. 
\end{lemma}
\begin{proof}
We fix an isomorphism $T \cong (\mathbb{G}_m)^n$ and a diagonalization of the $T$-action. Maps to $[\bA^N/(\bG_m)^n]$ classify the datum of 
\begin{enumerate}
\item  $n$ line bundles $\cL_i$, $i \in \{1, \ldots, n \}$
\item and $N$ sections $s_j \in H^0(\otimes_{i=1}^{i=n}\cL_i^{w^j_i})$, $j \in \{1, \ldots, N\}$ 
\end{enumerate} 
There is a natural map 
$$
 \Mapo{E}{[\bA^N/T]} 
 \to \Mapo{E}{[\Spec k/T]} \simeq \Pic^0(E)^N 
$$
which forgets the information on sections. 

By Proposition \ref{proposition:pointwisecriterion2}, it is sufficient to show  that given a $T$-equivariant open cover $\{U_i\}_{i \in I}$ of $\mathbb{A}^N$ and 
a map
$$
f: E \times_{\Spec k}\Spec K = E_K \to [\mathbb{A}^N/T]
$$
corresponding to a closed point of $\Mapo{E}{[\bA^N/T]}$, there is an open subset $U_i$ such that $f$ factors as 
$$
\xymatrix{
& [U_i/T] \ar[d] \\
E_K \ar@{-->}[ur] \ar[r]^-f & [\bA^N/T]}
$$

We will prove this by showing the stronger claim that $f$ has to factor through the image of a $T$-orbit in $[\bA^N/T]$. 
As we discussed, giving the map $f$ is the same as giving $n$ line bundles $\cL_i$ on $E_K$ and $N$ sections $s_j$ of appropriate tensor powers of the $\cL_i$-s. As $f$ is parametrized by a point in $\Mapo{E}{[\bA^N/T]}$ all the $\cL_i$-s have degree zero. Thus the sections $s_j$ are all constant. This immediately implies that the image of $f$ is a single orbit of the $T$-action. 
\end{proof}


Next, let us consider a $T$-action on the projective space $\mathbb{P}^N$. Up to fixing an isomorphism $T \cong (\bG_m)^n$ and applying an  automorphism of $\mathbb{P}^N$, we can put the $T$-action in the following   standard form. For  every $ \lambda = 
    (\lambda_1, \ldots, \lambda_n) \in (\mathbb{G}_m)^n \cong T$,  and $[z_0, \ldots, z_N] \in \bP^N$
\begin{equation}
\label{standardform}
\lambda  \cdot [z_0, z_1 \ldots z_N] =    [z_0, \prod_{i=1}^{n}\lambda_i^{w^1_i}z_1,\dots,\prod_{i=1}^{n}\lambda_i^{w^N_i}z_N] \in \bP^N  
\end{equation} 
for an appropriate collection of integers  $\{ w_i^j \}$.  
In particular, 
 we can assume that the standard toric affine open cover of $\mathbb{P}^N$ is $T$-equivariant.

\begin{lemma}\label{lemma:Pncodescent}
The stack of quasi-constant maps $\Mapo{E}{[\mathbb{P}^N/T]}$ satisfies $T$-equivariant Zariski codescent on $[\mathbb{P}^N/T]$.
\end{lemma}
\begin{proof}
The proof strategy is the same as for Lemma \ref{lemma:Ancodescent}. Namely, consider a map $$f:E_K \to [\bP^N/T]$$  classified by a point of 
$\Mapo{E}{[\bP^N/T]}$. 
We need to show that $f$ factors through the image of a $T$-orbit in 
$[\mathbb{P}^N/T]$. Now we make the following observations:
\begin{enumerate}
\item all line bundles on $\mathbb{P}^N$ admit a $T$-equivariant structure;
\item there exists a $r > 0$ such that there exists a non-vanishing $T$-equivariant section $\sigma$ of $\cO_{\mathbb{P}^N}(r)$;
\item we can further assume that, in point $(2)$, $r=1$ on condition of replacing, if needed, $\mathbb{P}^N$ with a larger projective space $\mathbb{P}^M$ equipped with a $T$-action and such that there is a $T$-equivariant Veronese embedding 
$$
\mathbb{P}^N  \to \mathbb{P}^M
$$
\end{enumerate}
Although these are all standard facts, let us sketch a proof. We start with $(1)$. We identify $\mathbb{P}^N$ with the projectivization $\mathbb{P}(V)$ of the vector space $V$ which is  dual to $H^0(\bP^N,\cO_{\bP^N}(1))$. Note that we can lift the $T$-action on $\bP^N$ to a linear action on $V$. This turns $V$ into a $T$-representation. Taking the dual representation gives a  $T$-action on $H^0(\bP^N,\cO_{\bP^N}(1))$.  This induces a $T$-equivariant structure on the line bundle $\cO_{\bP^N}(1)$. Taking tensor powers and duals generates $T$-equivariant structures on all the bundles $\cO_{\bP^N}(m)$, for all integers $m$.

Let us consider $(2)$ next. The existence of such a section $\sigma$ for some $r$ is equivalent to fact that the GIT quotient $\bP^N//T$ is non-empty, i.e. it is equivalent to the existence of a semistable point of $\bP^N$ with respect to the $T$-action. A semistable point in $\bP^N$ with respect to a linear action is by definition a semistable point of the lift of the action to the vector space $V \cong k^{N+1}$ such that 
$\mathbb{P}^N \cong \mathbb{P}(V)$, i.e. a point in $V$ such that the closure of its orbit under the $T$-action does not contain $0$. The existence of a semistable point is clear after we write the action in standard form (see equation (\ref{standardform}) above). Indeed, the point $(1,0,0,\dots,0)$ in $k^{N+1}$ is fixed by the lift of the action, and hence it is in particular  semistable, as the closure of its orbit does not contain the point $0$. This implies that the point $[1:0:0:\dots:0]$ in $\bP^N$ is also semistable, hence the GIT quotient $\bP^N//T$ is nonempty. 

As for point $(3)$, i.e. the reduction to the case $r=1$, it is sufficient 
to linearize the action of $T$ on $\bP^N$ by choosing an equivariant structure on $\cO_{\bP^N}(r)$ which makes $\sigma$ into an equivariant section. As familiar from GIT theory, the choice of linearization yields a $T$-equivariant embedding
$$\bP^N \to \bP(H^0(\cO_{\bP^N}(r))^\vee) \cong \bP^M$$ 
In particular, we obtain an isomorphism of $T$-modules
\begin{equation*}
    H^0(\bP^M,\cO_{\bP^M}(1)) \cong H^0(\bP^N,\cO_{\bP^N}(r))
\end{equation*}
which restricts to  the spaces of equivariant sections
\begin{equation*}
    H^0(\bP^M,\cO_{\bP^M}(1))^{T} \cong H^0(\bP^N,\cO_{\bP^N}(r))^{T}
\end{equation*}
hence a bijective correspondence between the $T$-equivariant sections of the two bundles. Note that $f$ factors through the image of a $T$-orbit if and only the composite map
$$
E_K \to [\bP^N/T] \to [\bP^M/T]
$$ 
factors through the image of a $T$-orbit. Thus, from the perspective of the argument we are carrying out, we can harmlessly replace $\bP^N$ with $\bP^M$.  We do this implicitly in the sequel, and in particular assume that $r=1$ and $\sigma$ is linear, but we will not rename either $\bP^N$ or $\sigma$.

Consider the map $\phi$  classifying the line bundle $\cO_{[\bP^N/T]}(r)$, $\, 
\phi: [\bP^N/T] \to [\Spec k/\bG_m].$   
If the map $f$ is classified by a geometric point  of $\Mapo{E}{[\bP^N/T]}$, the pullback line bundle $f^*(\cO_{[\bP^N/T]}(r))$ is classified by the composition 
$$
\Spec(K) \stackrel{f}\to \Mapo{E}{[\bP^N/T]} \stackrel{\phi}\to \Mapo{E}{[\Spec k/\bG_m]} \cong \Pic^0(E)
$$
In particular, $f^\ast(\cO_{[\bP^N/T]}(r))$ is a degree-zero line bundle on $E$. 

Let $H \subset \bP^N$ be the zero locus of $\sigma$, and consider its complement $U$. As $\sigma$ is linear $H$ is a hyperplane and $U$ is a $T$-equivariant open subset of $\bP^N$ isomorphic to $\mathbb{A}^N$. We claim that the image of $f$ is entirely contained either in $H$ or in $U$. Indeed, assume to the contrary that the image of  $f$ intersects both $H$ and $U$. The pullback section $f^\ast\sigma \in H^0(E, f^\ast\cO_{[\bP^N/T]}(r))$ is not constant, as the image of $f$ intersects both the zero-locus of $\sigma$ and its complement. However the line bundle $f^\ast\cO_{[\bP^N/T]}(r)$ must be of degree zero, and therefore its sections are necessarily constant. Thus as we claimed $f$ factors either through 
$$[\bA^N/T] \cong [U/T] \subset [\bP^N/T] \quad \text{ or through} \quad [H/T] \subset [\bP^N/T]$$ 
In the first case, we are done by Lemma \ref{lemma:Ancodescent}. In the second case, $f$ factors through the lower dimensional  space $[H/T] \cong [\bP^{N-1}/T]$, and we can conclude by induction on the dimension $N$: note indeed  that the base case of the statement $N=1$ is clear, as the complement of a $T$-invariant open subset $\bA^1 \cong U \subset \bP^1$ is a fixed point, which is a $T$-orbit. 
\end{proof}

We are now ready to prove Theorem \ref{proposition:codescentnormalvar}.
\begin{proof}[Proof of Theorem \ref{proposition:codescentnormalvar}]
This proof is very similar to that of Proposition \ref{affinization}. By Sumihiro's Theorem \cite{sumihiro1974equivariant} there exist a $T$-equivariant open cover 
$\{U_i\}_{i \in I}$ of $X$ such that each $U_i$ is affine. Further any $T$-equivariant open cover of $X$ can be refined to such a affine $T$-equivariant open cover. Thus  we can restrict, without loss of generality, to covers in $\Zar{X}^T$ given by disjoint unions of affines equipped with a $T$-action.

Let $\{U_i \to X\}_{i \in I}$ be such a cover. By the pointwise criterion, it is enough to show that for every quasi-constant  map $f:E_K \to [X/T]$ there exists an $i \in I$ such that there is a factorization
$$
\xymatrix{
& [U_i/T] \ar[d] \\
E_K \ar@{-->}[ur] \ar[r]^-f & [X/T]}
$$
Let $U_i$ be such that the fiber product 
$$
S:=[U_i/T] \times_{[X/T]} E_K    
$$
is non-empty, and fix a $T$-equivariant embedding
$$
U_i \to \bA^N \subset \bP^N
$$
where $\bA^N$ are $\bP^N$ equipped with a suitable $T$-action; the existence of such an embedding is an application of Lemma 5.2 in \cite{Luna}. Note that $[\bP^N/T]$ satisfies the valuative criterion for properness: as a consequence, there always exists a (non unique) arrow $h$ making the diagram 
$$
\xymatrix{
[U_i/T] \ar[r] & [\mathbb{P}^n/T] \\
S \ar[r] \ar[u]^-{f|_S} & E_K \ar[u]_-h \ar@{-->}[lu]
}
$$
commute. Thus, by Lemma \ref{lemma:Pncodescent}, $h$ admits a lift represented by the dashed arrow in the diagram. 
\end{proof}

\begin{remark}
\label{valuativeprop}
$[\bP^N/T]$ satisfies the valuative criterion of properness as $T$-torsors are locally trivial in the Zariski topology, and thus every $T$-torsor on the fraction field of a DVR  has a section. Indeed, let $R$ be a DVR and  consider a map $m: \Spec \Frac R \to [\bP^N/T]$. Let us show that it admits a lift 
$$
\bar m: \Spec(R) \to [\bP^N/T]
$$ 
The map $m$ classifies a diagram of the form 
$$
\xymatrix{
P \ar[r] \ar[d] & \bP^N \ar[d] \\
\Spec(\mathrm{Frac}(R)) \ar[r]^-m \ar[ur]^-n \ar@/^2pc/@{-->}[u]  & [\bP^N/T]}
$$
where $P$ is a $T$-torsor; as $P$ is trivial we can pick a section (represented by a dashed arrow) and this induces a lift $n$. Applying the ordinary valuative criterion for properness to $\bP^n$, we obtain a map $\bar n$ which makes the following diagram commute
$$
\xymatrix{
 \Spec(R) \ar[r]^-{\bar{ n}}  & \bP^N \ar[d] \\
\Spec(\mathrm{Frac}(R)) \ar[u] \ar[r]^-n      & \bP^N}
$$
Then the  composite map 
$$
\bar m: \Spec(R) \stackrel{\bar n} \rightarrow \bP^N  \rightarrow [\bP^N/T]
$$
is the desired lift of $m$.  
\end{remark}

\section{The local model.}\label{section:mainthmproof} 
In this section we compute the elliptic Hochschild homology of quotient stacks of the form 
\begin{equation}
\label{model}
 [\mathbb{A}^l \times \mathbb{G}_m^{k}/T]
\end{equation}
where $T$ is an algebraic torus. This calculation is relatively straightforward, but important. It is obtained by combining the results of propositions \ref{remark:Kunneth}, \ref{proposition:toricopen1} and \ref{proposition:toricopen2}. It will also play a role in the next section as, via Luna slice theorem, we will often be  able to reduce our arguments to this local case.
Then, using codescent, we compute the elliptic Hochschild homology of smooth toric varieties with their standard torus action. This will give us a broad supply of geometrically interesting examples for which  elliptic Hochschild homology can be explicitly described. Finally, in Theorem \ref{theorem:ellipticohoftoricvar} we show that when $k=\mathbb{C}$ this coincides with the complexified equivariant cohomology of the analytification of $X$. 
  
We begin this section with the following simple observation, which is an analogue of the K\"unneth isomorphism for mapping stacks. 
\begin{proposition}[K\"unneth formula]\label{remark:Kunneth}
Let $G$ be an algebraic group and let $X$ and $Y$ be $G$-varieties. Then  
\begin{equation*}
\Mapo{E}{[X\times Y/G]}\simeq\Mapo{E}{[X/G]}\times_{\Mapo{E}{[\Spec k/G]}}\Mapo{E}{[Y/G]}
\end{equation*}
where $G$ acts on $X\times Y$ diagonally. Similarly, if we equip the product $X\times Y$ with the product action, we have that
\begin{equation*}
\Mapo{E}{[X\times Y/G\times G]}\simeq\Mapo{E}{[X/G]}\times\Mapo{E}{[Y/G]}
\end{equation*}
\end{proposition}
\begin{proof}
Note that if $X$ and $Y$ are two $G$-varieties, then 
$$
[X\times Y/G]\simeq [X/G]\times_{[\Spec k/G]}[Y/G]
$$
for the diagonal $G$-action, and 
$$
[X\times Y/G\times G]\simeq [X/G]\times [Y/G]
$$
for the product $G$-action.
The formulas follow from the fact that $\Mapo{E}{-}$ preserves limits.
\end{proof}
 
 \subsection{The local model} 
In this section we consider quotient stacks of the form $[\mathbb{A}^l \times \mathbb{G}_m^{k}/T]$. 
Iterated applications of the K\"unneth formula allow us to break down the computations for product $\bA^l\times\bG_m^k$ to the cases of $\bA^1$ and $\bG_m$, and thus we will limit ourselves to describe these explicitly. This will be achieved in a sequence of Propositions. 
Let us  consider trivial actions first.
\begin{proposition}
Let $X$ be a variety over $k$ and let $T$ be an algebraic torus acting trivially on $X$. Then 
\begin{equation*}
\cHH_E([X/T]) \simeq \mathrm{HH}_\ast(X)\otimes_{k}\cO_{E_T}
\end{equation*}
\end{proposition}
\begin{proof}
Since the action is trivial, we have that $[X/T]\simeq X\times BT$. The proposition follows because   $\Mapo{E}{-}$ preserves products. \end{proof}
 
\begin{proposition}\label{proposition:toricopen1}
Let $T$ be an algebraic torus of rank $r$ acting on $\bA^1$. Assume that the $T$-action is non-trivial. 
 Then there is an equivalence 
\begin{equation*}
    \cHH_E([\bA^1/T]) \simeq \cO_{E_T}
\end{equation*}
\end{proposition}
\begin{proof}
Choosing an isomorphism $T\cong \bG_m^r$, the action in coordinates becomes
$$\lambda\cdot z=\prod_{i=1}^{r}\lambda_i^{w_i}z$$
Recall that the stack $\Mapo{E}{[\bA^1/\bG_m^r]}$ classifies $r$-tuples $\{\cL_i\}_{i=1}^{r}$ of degree zero line bundles on $E$ together with a section $s\in H^0(\otimes_{i=1}^{r} \cL_i^{w_i})$. In particular, we obtain the following description of $\Mapo{E}{[\bA^1/T]}$, when the action is non-trivial. Let $Z$ be the closed subscheme of $E_T$ cut out by the equation
\begin{equation}\label{eq:e}
\prod_{i=1}^{r} e_i^{w_i}=1
\end{equation}
Then we have a push-out diagram
$$
\xymatrix{
[Z\times\Spec k/T]\ar[d]^{0}\ar[r] & \Pic^0(E)_T\ar[d]\\
[Z\times\bA^1/T]\ar[r] & \Mapo{E}{[\bA^1/T]}\\
}
$$
where the left vertical map is the zero section of the projection $[Z\times\bA^1/T]\to[Z/T]$, and $T$ acts trivially on $E_T$ and in particular on $Z$. Note that if $T$ is rank 1,  $Z$ is a finite subset of the torsion points in $E_T$.
In particular, the coarse moduli space of the stack $\Mapo{E}{[\bA^1/T]}$ is given by $E_T$, which implies that 
$$p_\ast\cO_{\Mapo{E}{[\bA^1/T]}} \simeq \cO_{E_T}$$
and this concludes the proof.
\end{proof}

\begin{proposition}\label{proposition:toricopen2}
Let $T$ be an algebraic torus of rank $r$ acting on $\bG_m$.
Let $$p:\Mapo{E}{[\bG_m/T]}\rightarrow E_T$$ be the 
structure map. Let $Z \subset E_T$ be 
the closed susbscheme of $E_T$ cut out by equation \eqref{eq:e}. 
Then, as soon as the action is nontrivial, we have an equivalence 
\begin{equation*}
    \cHH_E([\bG_m/T])\simeq\cO_Z
\end{equation*}
as quasi-coherent sheaves on $E_T$. 
\end{proposition}
\begin{proof}
The proof of this proposition is the same as the proof of Proposition \ref{proposition:toricopen1}, and in fact the geometry is simpler. The stack $\Mapo{E}{[\bG_m/T]}$ classifies 
 $r$-tuples $\{\cL_i\}_{i=1}^{r}$ of degree zero line bundles on $E$, and \emph{trivializations} of the tensor product 
\begin{equation}
\label{triv}
 \bigotimes_{i=1}^{r} \cL_i^{w_i} \simeq \cO_E
\end{equation}
 where the $w_i$-s are the weights of the action. Thus  the stack 
 $\Mapo{E}{[\bG_m/T]}$ sits over the locus $Z \subset E_T$ of those bundles satisfying  (\ref{triv}).
\end{proof}
 
\subsection{The case of maximal tori}
In this subsection we compute the equivariant elliptic Hochschild homology of smooth toric varieties equipped with the toric action of their maximal torus.
In particular, we prove the following theorem:
\begin{theorem}\label{theorem:ellipticohoftoricvar}
Let $X$ be a smooth (normal) toric variety over $\bC$, and let the maximal torus $T$ act on $X$ either with the standard toric action or with weights $\{w_1,\dots,w_n\}$ all non-zero. Then there is an isomorphism of coherent sheaves on $E_T$
\begin{equation*}
\cHH_E([X/T])\simeq\El_T^0(X^\an)
\end{equation*}
where $X^\an$ denotes the analytification of $X$, and $\El_T^0(X^\an)\simeq\pi_0\El_T(X^\an)$ is 0-th homotopy sheaf of the complexified $T$-equivariant elliptic cohomology of $X^\an$ as defined in Section \ref{subsection:PerlimGroj}. 
\end{theorem}

\begin{proof}[Proof of theorem \ref{theorem:ellipticohoftoricvar}]
As $X$ is a toric variety, we can find a $T$-equivariant open cover $\cU$ by products $\bA^n\times\bG_m^k$. We assume for simplicity that $k=0$ as, by the K\"unneth formula, the case $k\neq 0$ can be treated in a very similar way. 
By the codescent property we know that the natural map
\begin{equation*}
\coprod_{U_i\in\cU}\Mapo{E}{[U_i/T]}\rightarrow\Mapo{E}{[X/T]}
\end{equation*}
induced by the cover $\cU$ is an effective epimorphism, and that the functor
$$
\Mapo{E}{[-/T]}: \Zar{X}^T \rightarrow\dStk
$$
is a cosheaf, as explained in theorem \ref{proposition:codescentnormalvar}. This implies that the functor
\begin{equation*}
\cHH_E([-/T]): \Zar{X}^T \rightarrow\Qcoh(E_T)
\end{equation*}
is a sheaf, and in particular $\cHH_E([X/T])$ is obtained as the totalization of the cosimplicial object $\cHH_E([\cU_\bullet/T])$, where $\cU_\bullet$ is the Cech nerve of the cover $\cU$.
Similarly, $T$-equivariant elliptic cohomology satisfies Mayer-Vietoris which implies that we can compute the quasi-coherent sheaf $\El_T^0(X)$ as the totalization of the cosimplicial object $\El_T^0(\cU_\bullet)$. 
Propositions \ref{proposition:toricopen1} and \ref{proposition:toricopen2} imply that this two cosimplicial objects coincide, hence they have the same totalization. 
\end{proof}

\section{Equivariant Elliptic Hochschild Homology}\label{section:HELL}
In this section we study the local behaviour of  equivariant elliptic Hochschild homology $$\cHH_E([X/T])$$ when $X$  is a smooth variety over $k$ equipped with an action of $T$. We relate this to Hochschild homology and equivariant elliptic cohomology. 
In particular, this involves completing the quasi-constant maps at the constant maps and comparing this   with the completion of the loop space at the constant loops. A localization phenomenon, combined with the group structure on the elliptic curve $E$, allows for the computation of the completions over all closed points of $E_T$. Our main results in this section are  analogues of theorems established by Chen in \cite{HChen}  in the setting of  ordinary Hochschild homology. 

First, we need the following definition.
\begin{definition}
Let $G$ be a derived group scheme acting on a derived stack $X$. The \emph{derived fixed locus} of the $G$-action  is the following fiber product of derived stacks 
$$
\xymatrix{
X^G \ar[r] \ar[d] & X\times G \ar[d]^{(\pi,a)} \\
X \ar[r]^{\Delta}  & X\times X} 
$$
where $\pi$ is the projection and $a$ is the action map. 
\end{definition}
From now on all fixed loci are assumed to be derived. 

\subsection{The localization formula for quasi-constant maps}\label{subsection:locformula}
In this section we establish a localization theorem for quasi-constant maps. This is an analogue of Theorem 3.1.12 in \cite{HChen}. 

\begin{theorem}[Localization formula]\label{proposition:localization}
Let $X$ be a smooth variety over $k$ equipped with an action of an algebraic torus $T$. Then for any closed point $e\in E_T$ there exists a Zariski open set $U\subset E_T$ containing $e$ such that the natural map
\begin{equation}\label{eq:locmap}
\Mapo{E}{[t_0 X^{T(e)}/T]}\times_{E_T} U\rightarrow\Mapo{E}{[X/T]}\times_{E_T} U
\end{equation}
induced by the inclusion of the classical fixed locus $t_0 X^{T(e)}\rightarrow X$, is an equivalence. 
\end{theorem}
In the above, $T(e)$ is the subgroup of $T$ as in Definition \ref{definition:T(x)}.
To prove this theorem we need to establish some preliminary results first.

\begin{lemma}[Localization for the affine space]
\label{lemma:Anlocalization}
Consider the $n$-dimensional affine space $\bA^n$, equipped with an action of a torus $T$ of rank $k$. Fix a closed point $e$ in $E_T$. Then there exists a Zariski open $U(e)$ in $E_T$ such that the natural map
\begin{equation*}
\phi:\Mapo{E}{[t_0 (\bA^{n})^{T(e)}/T]}\times_{E_T} U(e)\rightarrow\Mapo{E}{[\bA^n/T]}\times_{E_T} U(e)
\end{equation*}
induced by the inclusion of the classical fixed locus $t_0 (\bA^n)^{T(e)}\rightarrow \bA^n$ is an equivalence.
\end{lemma}
\begin{proof}
In the case of $[\bA^n/T]$ we can describe explicitly the open $U(e)$. First of all, observe that, since the action of $T$ on $\bA^n$ is linear, the fixed loci will be a linear subspace, in particular there exists a natural number $p$ such that $t_0(\bA^n)^{T(e)} \cong \bA^p$.
We distinguish two situations:
\begin{enumerate}
	\item $e=1$ in $E_T$. In this case, the statement is true for any $U(e)$;
	\item $e\neq 1$. In this case, the open set $U(e)$ can be described in terms of the weights $\{w_i^j\}$ of the action. 
\end{enumerate}
Let us focus on the second case.
Recall that the stack $[\Spec k/T]$ classifies principal $T$-bundles or equivalently $k$-tuples of line bundles $\{\cL_i\}_{i=1}^{k}$, while $[\bA^n/T]$ classifies such $k$-tuples $\{\cL_i\}_{i=1}^{k}$ together with an $n$-tuple of sections $s_j\in H^0(\otimes_{i=1}^{k} \cL_i^{w_i^j})$ for all values of $j$ in $\{1,\dots,n\}$. Since degree zero bundles on elliptic curves have sections if and only if they are trivial, the stacks $\Mapo{E}{[\bA^n/T]}$ and $\Mapo{E}{[\bA^p/T]}$ will differ only over the locus of those points $f=(f_1,\dots,f_n)$ in $E_T$ such that $\Sigma_{i=1}^{k}f_i^{w_i^j}=0$ for more than $p$ values of $j$ in $\{1,\dots,n\}$. This is because, for points $f$ of this kind, there will be more then $p$ line bundles of the form $\otimes_{i=1}^{k}\cL_i^{w_i^j}$ (indexed by $j$) that admit non-vanishing sections, hence $n$-tuples of sections may differ from $p$-tuples of sections.
Then, the open $U(e)$ is defined by removing from $E_T$ the locus of the points $f$ having this property. 
\end{proof}

\begin{remark}
The lemma above can be viewed as a statement about deformation theory of bundles with sections on elliptic curves. In particular, it is possible to compute the relative cotangent complex of the map $$\Mapo{E}{[t_0 (\bA^{n})^{T(e)}/T]}\rightarrow\Mapo{E}{[\bA^n/T]}$$ Its vanishing on closed points lying over the Zariski open $U(e)$ depends on the fact that nontrivial degree zero line bundles on $E$ have no non-zero sections.
\end{remark}
\subsubsection{The localization theorem on closed points}
\label{remark:bijonpoints}
Before proceeding with the proof of Theorem \ref{proposition:localization}, we will explain why the statement is true on closed points. This will clarify the geometry underlying Theorem \ref{proposition:localization}. Further the partial results we will obtain in this section will actually be needed in the course of the proof of Theorem \ref{proposition:localization}, which we will present in the next section.

In section \ref{section:proofcodescent} we have shown  that the derived stack of quasi-constant maps from $E$ satisfies a codescent property with respect to equivariant Zariski open covers. We  proved this by showing that the images of the total spaces of principal $T$-bundles are always contained inside $T$-orbits in the target space $X$. This property of quasi-constant maps allows us to show that the map \eqref{eq:locmap} is a homotopy equivalence on geometric points, as the type of $T$-orbit selects which bundles are allowed to map to it. 

\begin{proposition}[Localization formula on geometric points]\label{remark:bijonpoints}
Let $X$ be a smooth variety over $k$ equipped with an action of an algebraic torus $T$. Then for any closed point $e\in E_T$ there exists a Zariski open set $U\subset E_T$ containing $e$ such that the natural map
\begin{equation}\label{eq:locmap}
\Mapo{E}{[t_0 X^{T(e)}/T]}\times_{E_T} U   \rightarrow\Mapo{E}{[X/T]}\times_{E_T} U
\end{equation}
is a homotopy equivalence on geometric points.
\end{proposition}
\begin{proof}
Let us choose a $T$-orbit in $X$, $O$, generated by a closed point $x$ with stabilizer $T_x$. Then considering the mapping stack to $O\simeq T/T_x$ we obtain
\begin{equation*}
\Mapo{E}{[O/T]}\simeq\Mapo{E}{[\Spec k/T_x]}
\end{equation*}
which is the classifying stack of degree zero $T_x$-bundles on $E$. Hence, we conclude that only the principal $T$-bundles that admit a reduction of the structure group from $T$ to $T_x$ are allowed to map to the orbit $O$. In particular, as the subscheme $t_0 X^{T(e)}$ is a union of orbits of the form $T/S$, where $S$ is a subgroup of $T$ containing $T(e)$, the bundles that admit a reduction of the structure group to $T(e)$ are allowed to map to the complement of $t_0(X^{T(e)})$. 
In order to have that the map \eqref{eq:locmap} is a homotopy equivalence on $K$-points for an algebraically closed field $K$, we need to remove maps from those bundles. To do so, it is sufficient to remove the bundles that admit a reduction of the structure group to a subgroup of $T(e)$, and this is implemented by restricting the mapping stack to a Zariski open $U$ of $E_T$.
\end{proof}
 
\subsubsection{The proof}
We are now ready to prove Theorem  \ref{proposition:localization}. The key ingredient in the proof is Luna's slice theorem \cite{Luna}, which allows us to reduce to the case of a linear action on affine space, which was treated in Lemma \ref{lemma:Anlocalization}. 
\begin{proof}[Proof of Theorem  \ref{proposition:localization}] 
By codescent we can assume $X$ is affine. Our goal is to prove that the map
\begin{equation}\label{eq:mapproof}
\phi:\Mapo{E}{[t_0 X^{T(e)}/T]}\times_{E_T} U\rightarrow\Mapo{E}{[X/T]}\times_{E_T} U
\end{equation}
is an equivalence, for a Zariski open subset $U$ of $E_T$. We choose a $U$ that makes Proposition \ref{remark:bijonpoints} hold.

First we prove that this map is étale. Choose a point $x$ of $t_0(X^{T(e)})$ such that the orbit $Tx$ is closed. The Luna slice theorem applied to $x$ gives us a locally closed smooth subvariety $V$ of $X$ closed under the action of the stabilizer $T_x$ of $x$ such that the natural $T$-equivariant map $\psi: T\times^{T_x}V\rightarrow X$ is \'etale and has image given by a Zariski open $Z$ of $X$. We have an induced commutative diagram
$$
\xymatrix{
\Mapo{E}{[t_0(T\times^{T_x}V)^{T(e)}/T]}\times_{E_T}U \ar[r]^-{\phi_V} \ar[d]^{i_e} & \Mapo{E}{[T\times^{T_x}V/T]}\times_{E_T}U\ar[d]^{i}\\
\Mapo{E}{[t_0(X)^{T(e)}/T]}\times_{E_T}U \ar[r]^-{\phi} & \Mapo{E}{[X/T]}\times_{E_T}U
}
$$
for the mapping stacks. First note that the map \eqref{eq:mapproof} is locally finitely presented (see Remark \ref{locallyfinpreseq}), hence we only need to show it is formally étale, i.e. its relative cotangent complex vanishes.
To do so, we first observe that the vertical maps have vanishing relative cotangent complex. The argument is the same for the left and the right one. As for the right vertical composition, choose an $S$-point $$x:\Spec S\rightarrow\Mapo{E}{[T\times^{T_x}V/T]}$$ 
We need to show that the relative cotangent vanishes at any such $S$-point.  
We apply Halpern-Leistner and Preygel's Proposition 5.1.10 in \cite{HLP}:
$$\bL_{\Map{X}{Y},f}\simeq \pi_+ f^\ast\bL_{Y}$$
where the $S$-point $f:\Spec S\rightarrow\Map{X}{Y}$ is viewed as a map $f:\Spec S\times X\rightarrow Y$ and $\pi_+$ is a left adjoint to the pullback along the projection $\pi:\Spec S\times X\rightarrow\Spec S$. 

The pullback $x^\ast\bL_{i}$ is
$$x^\ast\bL_{i}\simeq \pi_+x^\ast\bL_{\psi}\simeq 0$$
as Luna's slice theorem guarantees that the map $\psi$ is étale.
In particular, we obtain that the map $i$ (and similarly $i_e$) is formally étale. One consequence of this fact is that we have an equivalence 
\begin{equation}\label{eq:eqv}
i_e^\ast\bL_{\phi}\xrightarrow{\simeq}\bL_{\phi_V}
\end{equation}
To see this, recall that for a commutative triangle
$$
\xymatrix{
X \ar[r]^{f}\ar[dr]^{h} & Y\ar[d]^{g} \\
 & Z
}
$$
of derived stacks, there is an induced cofiber   sequence of the relative cotangent complexes:
$$
f^{\ast}\bL_{g}\rightarrow\bL_{h}\rightarrow\bL_{f}
$$
(see for example Corollary $1.44$ in \cite{DAGIV}).

We get two cofiber sequences:
$$\phi_V^\ast\bL_i\rightarrow\bL_{i\circ\phi_V}\rightarrow\bL_{\phi_V}$$
$$i_e^\ast\bL_{\phi}\rightarrow\bL_{\phi\circ i_e}\rightarrow\bL_{i_e}$$
Since we know that the two relative contangent complexes $\bL_i$ and $\bL_{i_e}$ vanish, we get equivalences
$$\bL_{i\circ\phi_V}\xrightarrow{\simeq}\bL_{\phi_V}$$
$$i_e^\ast\bL_{\phi}\xrightarrow{\simeq}\bL_{\phi\circ i_e}$$
Since there is an equivalence $\phi\circ i_e\simeq i\circ\phi_V$, we conclude that the equivalence \eqref{eq:eqv} holds.

We now show $\phi_V$ is formally étale. Recall that, for smooth closed points $x\in X$, the Luna étale slice theorem gives us an additional map $V\rightarrow T_{V,x}$ from $V$ to its tangent space at $x$ which is $T_x$-equivariant and étale onto its image, which is an open subscheme of $T_{V,x}$. We have a further commutative diagram
$$
\xymatrix{
\Mapo{E}{[t_0(T\times^{T_x}V)^{T(e)}/T]}\times_{E_T}U \ar[r]^-{\phi_V} \ar[d]^{j_e} & \Mapo{E}{[T\times^{T_x}V/T]}\times_{E_T}U\ar[d]^{j}\\
\Mapo{E}{[t_0(T\times^{T_x}T_{V,x})^{T(e)}/T]}\times_{E_T}U \ar[r]^-{\phi_x}  & \Mapo{E}{[T\times^{T_x}T_{V,x}/T]}\times_{E_T}U
}
$$
and, reasoning as in the previous case, we obtain an equivalence 
$$j_e^\ast\bL_{\phi_x}\xrightarrow{\simeq}\bL_{\phi_V}$$
But the map
$$
\Mapo{E}{[t_0(T\times^{T_x}T_{V,x})^{T(e)}/T]}\times_{E_T}U \rightarrow \Mapo{E}{[T\times^{T_x}T_{V,x}/T]}\times_{E_T}U
$$
is an equivalence by an application of lemma \ref{lemma:Anlocalization}. In particular, we conclude that the map $\phi_V$ is formally étale and deduce that $i_e^\ast\bL_{\phi} \simeq 0$ from \eqref{eq:eqv}.
 
Let us observe that, in the case of algebraic actions of tori on affine varieties, there is a sufficient supply of points with closed orbit, i.e. there is a collection of closed points in $X$ such that the orbit they span is closed, and the images of the Luna slice maps $\psi$ at each of these points form an open cover of $X$. 
Indeed, recall that in an affine variety with an action of an algebraic group $G$, for every orbit $O$ there is a unique closed orbit in the complement $\overline{O}-O$, where $\overline{O}$ is the closure of $O$.
Moreover, the Zariski open sets $Z$ given by images of the étale slice maps $\psi:T\times^{T_x}V\to X$ are saturated, that means that given a point $z\in Z$ and any other point $x\in X$, if the intersection $\overline{Tz}\bigcap\overline{Tx}$ of the closure of the orbits is non-empty, then $x\in Z$. In particular, given an orbit $O$ in $X$, there always exists a point $x\in X$ whose orbit is closed, and such that the image $Z$ of the Luna slice at the point $x$ contain the orbit $O$. As a consequence, it is always possible to cover $X$ with images of Luna slices.

Now we can conclude: for each induced map $i_e$ relative to each of these points we know that $i_e^\ast\bL_\phi$ vanishes, and the coproduct of all the maps $i_e$ is an étale effective epimorphism by equivariant Zariski codescent. This is enough to prove that $\bL_{\phi}=0$, and since $\phi$ is locally finitely presented we obtain that it is étale, as we desired to show.

We now argue that \eqref{eq:mapproof} is an equivalence.
Since it is étale and a closed immersion, it is also an open immersion; in particular it exhibits $\Mapo{E}{[t_0 X^{T(e)}/T]}$ as a union of connected components of $\Mapo{E}{[X/T]}$. Since $\Mapo{E}{[X/T]}$ is a union of connected components of $\Map{E}{[X/T]}$ by definition, checking that the map \eqref{eq:mapproof} is an equivalence amounts only to understanding if its image contains all such connected components, which can be checked on geometric points. But Proposition  \ref{remark:bijonpoints} tells us that such map is a homotopy equivalence on the spaces of closed points, and this concludes the proof.
\end{proof}

\begin{remark}
\label{locallyfinpreseq}
We make the following observation. Let $f:X\rightarrow Y$ is a locally finitely presented map of derived stacks. Then the induced map 
$\Mapo{E}{X}\to\Mapo{E}{Y}$ is locally finitely presented (the proof goes like that of Lemma \ref{lemma:openinclusion}). Note that in the sitiuation of the proof of Theorem \ref{proposition:localization}, the map $[t_0 X^{T(e)}/T]\to [X/T]$ is locally finitely presented, as the map $t_0 X^{T(e)}\rightarrow X$ is a locally finitely presented map of (classical) schemes. Recall indeed that, as $t_0 X^{T(e)}$ and $X$ are varieties over a characteristic zero field $k$, the map $t_0 X^{T(e)}\rightarrow X$ is locally finitely presented (see for example Lemma 29.21.11 in \cite[\href{https://stacks.math.columbia.edu/tag/01TO}{Tag 01TO}]{StacksProj1}). 
\end{remark}

A consequence of the localization formula is the following description of the fibers of the structure map $p':\Mapo{E}{[X/T]}\rightarrow\Pic^0(E)_T$.

\begin{corollary}\label{corollary:fibers}
Let $X$ be a smooth variety over $k$. Then given a closed point $$\bar{e}:\Spec K\rightarrow\Pic^0(E)_T$$ the fiber of the structure map $p':\Mapo{E}{[X/T]}\rightarrow\Pic^0(E)_T$ over the point $\bar{e}$ is given by  the derived fixed locus $X^{T(e)}$, where $e$ is the closed point in $E_T$ corresponding to the composition $\Spec K\xrightarrow{\bar{e}}\Pic^0(E)\to E_T$.
\end{corollary}
\begin{proof}
Since the closed immersion $\bar{e}:\Spec K\rightarrow\Pic^0(E)_T$ factors through $\underline{U_e}=U_e\times BT$, where $U_e$ is the Zariski open determined by Proposition \ref{proposition:localization} for the point $e$, we can apply   the same proposition to reduce the computation to the fiber 
\begin{equation*}
\left (\Mapo{E}{[X/T]}\times_{\Pic^0(E)_T} \underline{U_e}\right )\times_{\underline{U_e}}\bar{e}\simeq\left (\Mapo{E}{[t_0 X^{T(e)}/T]}\times_{\Pic^0(E)_T} \underline{U_e}\right )\times_{\underline{U_e}}\bar{e}
\end{equation*}
which in turn is the fiber of the structure map $p':\Mapo{E}{[t_0 X^{T(e)}/T]}\rightarrow\Pic^0(E)_T$ over the point $\bar{e}$. 

Let $T(e)$ be the subgroup associated to the point $e$ in $E_T$. The map $\bar{e}:\Spec K\rightarrow\Pic^0(E)_T$ will factor through the stack $\Bun^0_{T(e)}(E)$, hence we can compute the fiber using the following pasting of pullback diagrams:
$$
\xymatrix{
F_{e} \ar[r]\ar[d]  &\Mapo{E}{[t_0 X^{T(e)} /T(e)]} \ar[r]\ar[d]  & \Mapo{E}{[t_0 X^{T(e)}/T]}\ar[d] \\
\Spec K \ar[r]^{\bar{e}}                 &\Bun^0_{T(e)}(E) \ar[r]               & \Pic^0(E)_T
}
$$
where we called $F_{e}$ the fiber we are interested in computing. The square on the right is a pullback since $\Mapo{E}{-}$ commutes with limits, and the diagram
$$
\xymatrix{
[t_0 X^{T(e)} /T(e)] \ar[r] \ar[d]^{p_e}  & [t_0 X^{T(e)}/T]\ar[d] \\
[\Spec k/T(e)] \ar[r]               & [\Spec k/T]
}
$$
is a pullback. The calculation of $F_{e}$ follows easily from the observation that, by definition, $T(e)$ acts trivially on $t_0 X^{T(e)}$, and in particular we have that 
\begin{equation*}
 \Mapo{E}{[t_0 X^{T(e)} /T(e)]} \simeq \Mapo{E}{t_0 X^{T(e)} \times BT(e)} \simeq \Mapo{E}{t_0 X^{T(e)}} \times \Bun^0_{T(e)}(E)
\end{equation*} 
hence the map $p_e$ necessarily has fiber over $\bar{e}$ given by $\Mapo{E}{t_0 X^{T(e)}}$, as $p_e$ is isomorphic to the projection to $\Bun^0_{T(e)}(E)$. Then, by Corollary $1.0.1$ in \cite{HChen}  
\begin{equation*}
\Mapo{E}{t_0 X^{T(e)}}\simeq\cL\, t_0 X^{T(e)} \simeq X^{T(e)}
\end{equation*} 
\end{proof}


\begin{remark}
We established Theorem \ref{proposition:localization} for closed points $e\in E_T$, but it holds also for non-closed points $x \in E_T$ with the notion of subgroup $T(x)$ associated to one such point introduced in Remark \ref{rmk:T(x)nonclosed}. Indeed, if $x$ is any point in $E_T$, its closure $\overline{\{x\}}$ contains at least one closed point $e$ such that the two subgroups $T(e)$ and $T(x)$ coincide. We can then declare the Zariski open subset $U_x$ of $E_T$ realizing localization for the point $x$ to be the open subset $U_e$ associated to the closed point $e$, as $x \in U_e$. If we do so, the statement of Theorem \ref{proposition:localization} extends to non-closed points. 
\end{remark}

\subsection{The local structure of the quasi-constant maps}
We now compute the completions of elliptic Hochschild homology at closed points of $E_T$. 

Recall that for a derived stack $\mathcal{X}$, there is a natural map
\begin{equation*}
\mathcal{X}\xrightarrow{\simeq}\Map{\Spec k}{\mathcal{X}}\rightarrow\Mapo{E}{\mathcal{X}}
\end{equation*}
induced by the structure morphism $E\rightarrow \Spec k$. We call the completion of this map \emph{the completion of $\Mapo{E}{\cX}$ at the constant maps} or \emph{formal maps} from $E$ to $\cX$. There is an analogous map for the loop space $\cL\cX$, and its formal completion is usually called the \emph{formal loop space}: we denote it by $\widehat{\cL}\cX$. See Definition \ref{def:formalloops} for reference.
\begin{remark}
The formal completion of $\Mapo{E}{\cX}$ at the constant maps is the same as that of $\Map{E}{\cX}$, as $\Mapo{E}{\cX}$ is  a collection of connected components of $\Map{E}{\cX}$ containing  the constant maps.
\end{remark}
Recall the notion of 
\begin{proposition}\label{proposition:formalcompletions}
Let $\cX$ be a derived stack with affine diagonal over a field $k$ of characteristic zero, and $E$ be an elliptic curve over $k$. There is a natural map between formal completions at the constant maps
\begin{equation*}
\psi:\widehat{\cL}\cX\rightarrow\fMapo{E}{\cX}
\end{equation*}
Further, $\psi$ is an equivalence. 
\end{proposition}
Proposition \ref{proposition:formalcompletions} is closely related to results that have already appeared in the literature in slightly different settings, and in particular to Theorem 6.9 of \cite{BZNLoopConn}. The point is that the deformation theory of quasi-constant maps out of $E$ near the constant maps is controlled by the affinization of $E$. As the latter is equivalent to the affinization of $S^1$ the completion of $\Mapo{E}{\cX}$ is equivalent to the completion of $\cL X$. 

\begin{proof}
We will show that the map 
$$\psi:\widehat{\cL}\cX\simeq\fMapo{\Aff{E}}{\cX}\rightarrow\fMapo{E}{\cX}$$
induces an equivalence of the pullback of the cotangent complexes to the constant maps.

Define the maps
\begin{align*}
& u:\Map{\Aff{E}}{\cX}\rightarrow\Map{E}{\cX} \\
& c':\cX\simeq\Map{\Spec k}{\cX}\rightarrow\Map{\Aff{E}}{\cX}\\
& c:=u\circ c':\cX\simeq\Map{\Spec k}{\cX}\rightarrow\Map{E}{\cX}
\end{align*}
given by composition with the unit of the affinization $E\rightarrow\Aff{E}$ and with the structure maps $\Aff{E}\rightarrow\Spec k$ and $E\rightarrow\Spec k$ respectively. The map $u$ induces a fiber-cofiber sequence
\begin{equation*}
u^\ast\bL_{\Map{E}{\cX}}\rightarrow\bL_{\Map{\Aff{E}}{\cX}}\rightarrow\bL_{u}
\end{equation*}
of quasi-coherent sheaves on $\Map{\Aff{E}}{\cX}$. Here, $\bL_u$ denotes the relative cotangent complex of the map $u$.
Pulling this back along the map $c'$ we get a fiber-cofiber sequence
\begin{equation*}
c^\ast\bL_{\Map{E}{\cX}}\rightarrow c'^\ast\bL_{\Map{\Aff{E}}{\cX}}\rightarrow c'^\ast\bL_{u}
\end{equation*}
of quasi--coherent sheaves on $\cX$. Our goal is to show that $c'^\ast\bL_u$ vanishes. 

To do so, we show the stronger fact that for any constant map $x:\Spec S\to \cX$ the map of the based loop spaces
$$\Omega_x(u):\Omega_x\Map{\Aff{E}}{\cX}\to\Omega_x\Map{E}{\cX}$$
is an equivalence.
Indeed, the cotangent complex of a based loop space is a shift of the original cotangent complex: indeed, let $F:\dAff^\op\to\cS$ be a prestack, and let $x:\Spec S\to F$ be an $S$-point of $F$; then the point $x$ canonically induces a point $\delta_x:\Spec S\to \Omega_x F$, and the following relation holds:
$$\bL_{F,x} \simeq \bL_{\Omega_x F,\delta_x}[-1]$$

To show that $\Omega_x(u)$ is an equivalence, observe that the based loop spaces have a presentation as mapping stacks: in particular we have equivalences
\begin{align*}
& \Omega_x\Map{\Aff{E}}{\cX}\simeq\Map{\Aff{E}}{\Omega_x\cX} \\
& \Omega_x\Map{E}{\cX}\simeq\Map{E}{\Omega_x\cX}
\end{align*}
obtained by applying the tensor-hom adjunction twice on different factors. Since $\cX$ has affine diagonal the based loop space $\Omega_x\cX$ is an affine scheme, hence we have an identification $$\Map{\Aff{E}}{\Omega_x\cX}\simeq\Map{E}{\Omega_x\cX}.$$
This completes the proof.
\end{proof}


\begin{remark}
The equivalence in Proposition \ref{proposition:formalcompletions} is clearly natural in the second variable. Note also that the fact that the map $\Map{\Aff{E}}{\cX}\rightarrow\Map{E}{\cX}$ has vanishing cotangent complex over the constant maps holds in considerable generality, as the only restriction is that the mapping stacks have to admit a cotangent complex over the loci of constant maps. For instance, it remains true if $E$ replaced with any smooth variety.
\end{remark}

We will now apply Proposition \ref{proposition:formalcompletions} to compute the completion of $\cHH_E([X/T])$ at the identity of $E_T$. 
Denote by 
$$
i: \{1_{E_T}\} \to E_T \quad j: \{1_T\} \to T
$$
the inclusion of the identity elements.  
Let $\widehat{E_T}$ be the completion of $E_T$ at $i$ and denote by 
$$\widehat{i}: \widehat{E_T}\rightarrow E_T$$ the natural map. 
Similarly, let $\widehat{T}$ be the completion of $T$ at $j$ and denote by $$\widehat{j}:\widehat{T}\rightarrow T$$ the natural map. 
Following Remark $2.4.9$ in \cite{HChen}, we define the derived completion of a quasi-coherent sheaf $\cF$ on $E_T$ at the identity element as the pull-push $$\widehat{i}_\ast\widehat{i}^\ast \cF$$

From now on we require our variety $X$ to be quasi-projective. This is because, as a consequence of Corollary \ref{corollary:fibers}, the natural map 
$$p:\Mapo{E}{[X/T]}\to E_T$$
is perfect in the sense of Ben-Zvi--Francis--Nadler \cite{BZFN}. Indeed, $X$ quasi-projective implies that the map is (a composition of) maps between perfect stacks, hence it satisfies base-change. This is required in our arguments below. We do believe though that the map in question is perfect even without such quasi-projectivity assumption.

\begin{corollary}\label{corollary:completionatone}
The derived completion of $\cHH_E([X/T])$ at the identity of $E_T$ 
$$\widehat{i}_\ast\widehat{i}^\ast \cHH_E([X/T])$$ 
is the pushforward along $\widehat{i}$ of the completion of $\mathrm{HH}_\ast([X/T])$ at the prime ideal corresponding to the point $1\in T \cong \Spec \mathrm{HH}_\ast([\Spec k/T])$.
\end{corollary}
\begin{proof}
Consider the following pullback diagram:
$$
\xymatrix{
\fMapo{E}{[X/T]} \ar[r]^{\widehat{i}_X}\ar[d]^{\widehat{p}} & \Mapo{E}{[X/T]}\ar[d]^{p} \\
\widehat{E_{T}}\ar[r]^{\widehat{i}} \ar[r] & E_{T}
}
$$
Base-changing along this diagram, we can substitute $\widehat{i}^\ast p_\ast$ with $\widehat{p}_\ast\widehat{i}_{X}^\ast$ in $\widehat{i}_\ast\widehat{i}^\ast \cHH_E([X/T])$. 

There is an analogous pullback square for the loop space
$$
\xymatrix{
\widehat{\cL}[X/T] \ar[r]^{\widehat{j}_X}\ar[d]^{\widehat{q}} & \cL[X/T]\ar[d]^{q} \\
\widehat{T}\ar[r]^{\widehat{j}} \ar[r] & T
}
$$
We may consider similar completions for the loop space, namely $\widehat{j}_\ast\widehat{j}^\ast q_{\ast}\cO_{\cL[X/T]}$. Base-changing along the pullback square for the loop space we may rewrite this completion as
\begin{equation}\label{eq:BCloops}
\widehat{j}_\ast\widehat{j}^\ast q_{\ast}\cO_{\cL[X/T]} \simeq \widehat{j}_\ast \widehat{q}_{\ast}\widehat{j}_{X}^{\ast}\cO_{\cL[X/T]} \simeq \widehat{j}_\ast \widehat{q}_{\ast}\cO_{\widehat{\cL}[X/T]}
\end{equation}
Now, $\widehat{j}_\ast \widehat{q}_{\ast}\cO_{\widehat{\cL}[X/T]}$ is a quasi-coherent sheaf on $T$, which is affine, hence it is completely determined by its global sections, which are given by the completion of the Hochschild homology module of $[X/T]$ at the maximal ideal corresponding to the identity element of the torus $T$. 

Proposition \ref{proposition:formalcompletions} provides an identification of the maps $\widehat{p}$ and $\widehat{q}$. In particular, when evaluating the expression $\widehat{i}_\ast\widehat{i}^\ast \cHH_E([X/T])$ we obtain
\begin{equation*}
\widehat{i}_\ast\widehat{i}^\ast \cHH_E([X/T]) \simeq \widehat{i}_\ast \widehat{p}_\ast\cO_{\fMapo{E}{[X/T]}} \simeq \widehat{i}_\ast \widehat{q}_{\ast}\cO_{\widehat{\cL}[X/T]}
\end{equation*}
\end{proof}

The completions over other closed points $e$ of $E_T$ can be computed from the completion at the identity using the localization formula explained in \ref{subsection:locformula} and the group structure on $E_T$. 

For a closed point $e$ of $E_T$, let $\widehat{i}_e:\widehat{E}_{T,e}\rightarrow E_T$ be the natural map from the derived formal completion of $E_T$ at the closed point $e$ to $E_T$. Moreover, call $\widehat{\mu}_e:\widehat{E}_T\rightarrow\widehat{E}_{T,e}$ the completion of the map of multiplication by $e$, which is an equivalence of formal derived schemes. In particular, the group structure on $E_T$ gives canonical identifications between completions at different closed points.

\begin{theorem}\label{theorem:completionsgeneral}
The (derived) completion of $\cHH_E([X/T])$ at the closed point $e$ of $E_T$, 
$$(\widehat{i}_e)_\ast\widehat{i}_e^\ast \cHH_E([X/T]),$$ is the completion of $\mathrm{HH}_\ast([t_0 X^{T(e)}/T])$ at the prime ideal corresponding to the point 
$$1\in T \cong \Spec \mathrm{HH}_\ast([\Spec k/T]).$$
\end{theorem}
\begin{proof}
Using a similar base change procedure as in the proof of Corollary \ref{corollary:completionatone} we can rewrite the derived completion of $\cHH_E([X/T])$ at $e$ as
\begin{equation*}
(\widehat{i}_e)_\ast\widehat{i}_e^\ast \cHH_E([X/T])\simeq (\widehat{i}_e)_\ast (\widehat{p}_{e})_\ast\cO_{\Mapo{E}{[X/T]}}
\end{equation*}
where $p_{\widehat{e}}:\fMapo{E}{[X/T]}_{e}\rightarrow\widehat{E}_{T,e}$ is the completion at $e$ of the structure map $p$. Proposition \ref{proposition:localization} gives us a Zariski open $U$ of $E_T$ containing $e$ such that $$\Mapo{E}{[t_0 X^{T(e)}/T]}\times_{E_T} U\rightarrow\Mapo{E}{[X/T]}\times_{E_T} U$$ is an equivalence, which implies that the completion at $e$ of $\cHH_E([X/T])$ is equivalent to the completion at the same point of the sheaf given by $\cHH_E([t_0 X^{T(e)}/T])$.

Consider the following pullback diagram:
$$
\xymatrix{
\fMapo{E}{[t_0 X^{T(e)}/T]} \ar[r]^{\widehat{\mu}_{e,X}}\ar[d]^{\widehat{p}} & \fMapo{E}{[t_0 X^{T(e)}/T]}_{e}\ar[d]^{\widehat{p}_e} \\
\widehat{E}_{T}\ar[r]^{\widehat{\mu}_e} \ar[r] & \widehat{E}_{T,e}
}
$$
Since $\widehat{\mu}_e$ is an equivalence, we have the following relation:
\begin{equation*}
(\widehat{p}_{e})_{\ast}\cO_{\fMapo{E}{[t_0 X^{T(e)}/T]}_{e}} \simeq (\widehat{\mu}_{e})_{\ast}\widehat{p}_{\ast}\cO_{\fMapo{E}{[t_0 X^{T(e)}/T]}}
\end{equation*}
Corollary \ref{corollary:completionatone} implies that we can identify $\widehat{p}_{\ast}\cO_{\fMapo{E}{[t_0 X^{T(e)}/T]}}$ with $\widehat{q}_{\ast}\cO_{\widehat{\cL}[t_0 X^{T(e)}/T]}$. By plugging in this equivalence in the previous expression we obtain
\begin{equation*}
(\widehat{p}_{e})_{\ast}\cO_{\fMapo{E}{[t_0 X^{T(e)}/T]}_{e}} \simeq (\widehat{\mu}_{e})_{\ast}\widehat{q}_{\ast}\cO_{\widehat{\cL}[t_0 X^{T(e)}/T]}
\end{equation*}
as quasi-coherent sheaves on the formal completion $\widehat{E}_{T,e}$. In order to obtain the desired quasi-coherent sheaf on $E_T$ we need to take the pushforward  along $\widehat{i}_e$: 
$$
(\widehat{i}_e)_\ast\widehat{i}_e^\ast \cHH_E([X/T]) \simeq (\widehat{i}_e)_{\ast}(\widehat{\mu}_{e})_{\ast}\widehat{q}_{\ast}\cO_{\widehat{\cL}[t_0 X^{T(e)}/T]}
$$
Since the map $\widehat{i}:\widehat{E}_T\rightarrow E_T$ factors as $\widehat{\mu}_e:\widehat{E}_T\rightarrow\widehat{E}_{T,e}$ followed by $\widehat{i}_e:\widehat{E}_{T,e}\rightarrow E_T$, we can rewrite the previous formula as 
\begin{equation*}
(\widehat{i}_e)_\ast\widehat{i}_e^\ast \cHH_E([X/T]) \simeq 
(\widehat{i}_e)_{\ast}(\widehat{\mu}_{e})_{\ast}\widehat{q}_{\ast}\cO_{\widehat{\cL}[t_0 X^{T(e)}/T]}
\simeq 
\widehat{i}_{\ast}\widehat{q}_{\ast}\cO_{\widehat{\cL}[t_0 X^{T(e)}/T]}
\end{equation*}
and this completes the proof.
\end{proof}

Theorem \ref{theorem:completionsgeneral} will play a major role in the proof of our general comparison theorem between a periodic cyclic version of elliptic Hochschild homology and equivariant elliptic cohomology in the sense of Grojnowski. This is the content of Section \ref{section:Tate}.

\begin{remark}
Using equation \eqref{eq:BCloops} we can write
\begin{equation*}
(\widehat{i}_e)_\ast\widehat{i}_e^\ast \cHH_E([X/T]) \simeq \widehat{i}_{\ast}\widehat{q}_{\ast}\cO_{\widehat{\cL}[t_0 X^{T(e)}/T]} \simeq \widehat{i}_{\ast}\widehat{j}^{\ast}q_{\ast}\cO_{\cL[t_0 X^{T(e)}/T]}
\end{equation*}
where $q: \cL [X/T] \to T$ is the natural map. 
\end{remark}
We denote $q_{\ast}\cO_{\cL[t_0 X^{T(e)}/T]}$ as $\cHH([X/T])$. The global sections of this sheaf over 
$$T\cong\Spec \mathrm{HH}([\Spec k/T])$$ are given by $\mHH([X/T])$. 
Call $\widehat{k}:\widehat{\ft}\rightarrow\ft$ the map from the completion at $0$ of the Lie algebra of $T$ to the Lie algebra of $T$, and by $\cH([X/T])$ the quasi-coherent sheaf on $\ft \cong \Spec \mathrm{H}_T(\ast)$ having as global sections  the $\mathbb{Z}_2$-periodized $T$-equivariant cohomology of $X^\an$, $\mathrm{H}_T^{\oplus,\ast}(X^\an)$. Then we have
\begin{align*}
& (\widehat{i}_e)_\ast\widehat{i}_e^\ast \cHH_E([X/T]) \simeq \widehat{i}_{\ast}\widehat{q}_{\ast}\cO_{\widehat{\cL}[t_0 X^{T(e)}/T]} \simeq \widehat{i}_{\ast}\widehat{j}^{\ast}\cHH([t_0 X^{T(e)}/T])\\
& (\widehat{i}_e)_\ast\widehat{i}_e^\ast\El_T(X^\an)\simeq\widehat{i}_{\ast}\widehat{k}^{\ast}\cH([t_0 (X^\an)^{T(e)}/T])
\end{align*}
as the completions of equivariant elliptic cohomology over $E_T$ compute Borel equivariant cohomology. If we replace Hochschild homology by periodic cyclic homology by taking Tate fixed points with respect to the canonical $S^1$-action, the two completions become equivalent by means of Chen's Theorem 4.3.2 in \cite{HChen}.

In the next section we explain how the natural $E$-action on $\Mapo{E}{[X/T]}$ induces $S^1$-actions on the adelic descent data of elliptic Hochschild homology, and use this action to  define the \emph{elliptic periodic cyclic homology} of $[X/T]$. We will show that this object recovers Grojnowski's equivariant elliptic cohomology of the analytification.

\section{The adelic Tate construction}\label{section:Tate}

\subsection{The action of the elliptic curve $E$} 
Let $\cX$ be a derived stack over $k$. The multiplication map $\mu:E\times E\to E$ induces a global $E$-action on $\Mapo{E}{\cX}$.

\begin{remark}
Let $X$ be a variety over a field $k$ of characteristic zero.
As explained in \cite{BZNLoopConn} the unit map
$$E\to\Aff{E}$$
is a group homomorphism. This implies that the canonical equivalence
$$\cL X\simeq\Mapo{\Aff{E}}{X}\to\Mapo{E}{X}$$
intertwines the $E$-action on the mapping stack on the right with the $\Aff{E}\simeq\Aff{S^1}\simeq B\bG_a$-action on the mapping stack on the left, i.e. the ``elliptic HKR isomorphism" of Proposition \ref{affinization}
\[
\Mapo{\Aff{E}}{X}\simeq\Mapo{E}{X}
\]
is equivariant with respect to the relevant actions.
\end{remark}

In the following lemma we characterize the $E$-action on the stack of quasi-constant maps from $E$ to $BT$.

\begin{lemma}\label{Lemma:trivaction}
The $E$-action on $\Mapo{E}{BT}$ induces a trivial action on the coarse moduli space $E_T$.
\end{lemma}
\begin{proof}
Without loss of generality, let us restrict to the case when $T$ has rank $1$. 
The triviality of the action on the coarse moduli space is a consequence of the fact that degree zero line bundles on elliptic curves can be presented as maps of abelian groups from $E$ to $B\bG_m$. 
In particular, consider the following maps:
\begin{itemize}
\item the action map 
$$E\times \Pic^0(E)\to\Pic^0(E)$$ 
that is adjoint to the map
$$\Mapo{E}{B\bG_m}\to\Mapo{E\times E}{B\bG_m}$$
induced by composition with the multiplication map $\mu : E\times E\to E$;
\item the map that classifies the box product of line bundles
$$\boxtimes:E\times\Pic^0(E)\to\Pic^0(E)$$
\end{itemize}
The latter is constructed by adjunction from the map
$$\alpha:\Mapo{E}{B\bG_m}\to\Mapo{E\times E}{B\bG_m}$$
which is obtained from the following composition:
$$E\times\Mapo{E}{B\bG_m}\times E\times\Mapo{E}{B\bG_m}\xrightarrow{\mathrm{ev}\times\mathrm{ev}}B\bG_m\times B\bG_m\xrightarrow{m}B\bG_m$$
In the above, $\mathrm{ev}$ is the evaluation map, and $m:B\bG_m\times B\bG_m\to B\bG_m$ is the multiplication on $B\bG_m$. Adjoining the composition $m\circ (\mathrm{ev}\times \mathrm{ev})$, we obtain
$$\beta:\Mapo{E}{B\bG_m}\times\Mapo{E}{B\bG_m}\to\Mapo{E\times E}{B\bG_m}$$
We further compose the above map with the diagonal $\Delta$:
$$\Mapo{E}{B\bG_m}\xrightarrow{\Delta}\Mapo{E}{B\bG_m}\times\Mapo{E}{B\bG_m}\xrightarrow{\beta}\Mapo{E\times E}{B\bG_m}$$
to obtain the map $\alpha$ whose adjoint $\boxtimes$ classifies the box product. The map $\boxtimes$ becomes equivalent to the projection to $\Pic^0(E)$, i.e. the trivial action map,
$$E\times\Pic^0(E)\to\Pic^0(E)$$
after passing to the coarse moduli space.

The condition that degree zero line bundles on elliptic curves correspond to maps of abelian groups from $E$ to $B\bG_m$ implies that the maps on the coarse moduli spaces induced by the action map and the map classifying the box product are isomorphic.
\end{proof}

\begin{remark}\label{remark:pEequivariant}
Since the $E$-action on $\Mapo{E}{\cX}$ is induced by the group structure on $E$, for any map of derived stacks $$f:\cX\to \cY$$ the map induced by composition 
$$\Mapo{E}{\cX}\to \Mapo{E}{\cY}$$
is $E$-equivariant. In particular, the structure map
$$p:\Mapo{E}{[X/T]}\to E_T$$
is $E$-equivariant. In the above, $E$ acts trivially on $E_T$ as this is the coarse moduli space of $\Pic^0(E)_T$, on which the action is trivial by Lemma \ref{Lemma:trivaction}.
\end{remark}

\subsection{Adelic Tate construction for elliptic cohomology}
We now describe how the global $E$-action allows us to perform a Tate construction on the sheaf $\cHH_{E}([X/T])$. 

Consider $\pi:E\times E_T\to E_T$ as a group scheme over $E_T$; this object acts on $\Mapo{E}{[X/T]}$ in the category of derived stacks over $E_T$. 
\begin{remark}
The pushforward  $\pi_{\ast}\cO_{E\times E_T}$ is a sheaf of Hopf algebras on $E_T$. The action relative to $E_T$ gives to $p_{\ast}\cO_{\Mapo{E}{[X/T]}}=\cHH_{E}([X/T])$ the structure of a comodule over this sheaf of Hopf algebras in $\QCoh(E_T)$.
\end{remark}
Over a field $k$ of characteristic zero, the equivalence $\Aff{E}\simeq\Aff{S^1}$ induces an equivalence 
$$\pi_{\ast}\cO_{E\times E_T}\simeq\pi_{\ast}\cO_{S^1\times E_T}$$
where $\pi:S^1\times E_T\to E_T$ is viewed as a group scheme over $E_T$. In particular, $\cHH_{E}([X/T])$ receives a comodule structure over $\pi_{\ast}\cO_{S^1\times E_T}$.

The elliptic curve $E$ acts trivially on $E_T$ itself according to lemma \ref{Lemma:trivaction}. This action induces an $S^1$-action on the category $\QCoh(E_T)$, by letting $S^1$ act trivially on $E_T$. In particular, by the discussion above, the $E$-action on $\Mapo{E}{[X/T]}$ induces a lift of $\cHH_{E}([X/T])$ to the $S^1$-invariant category $\QCoh(E_T)^{S^1}$. 

\begin{definition}
The \emph{elliptic periodic cyclic homology} of $[X/T]$
$$\cHP_{E}([X/T])$$
is the image of the pair given by $\cHH_{E}([X/T])$ and its $S^1$-action in the category $\QCoh(E_T)_{\bZ_2}\coloneqq\QCoh(E_T)\otimes_{k}k[\beta,\beta^{-1}]$.
\end{definition}

\begin{remark}
The $E$-action on $E_T$ restricts to each term of its adelic decoposition $\Spec\bA_{E_T}^\bullet$, and moreover the adelic descent data $\bA_{E}^{\bullet}(\cHH_{E}([X/T]))$ obtain an action of $S^1$ by similar arguments to those above. 
Let $\bA_{E}^{\bullet}(\cHH_{E}([X/T]))^{tS^1}$ be the cosimplicial object obtained by applying level-wise the Tate construction with respect to this $S^1$-action. We have the following equivalence:
$$\bA_{E}^{\bullet}(\cHP_{E}([X/T]))\simeq\bA_{E}^{\bullet}(\cHH_{E}([X/T]))^{tS^1}$$
This behaviour justifies the name \emph{adelic} Tate constuction.
\end{remark}

\begin{remark}\label{remark:coherence}
Under the assumption that $X$ is a smooth variety over a field $k$, the elliptic periodic cyclic homology of $[X/T]$ is a $\bZ_2$-periodic perfect complex on $E_T$. Indeed, by the discussion in section \ref{section:HELL}, the coherence of this complex is controlled by finite generation of $H_{T^\an}(X^\an)$ as a $H_{T^\an}(\ast)$-module.
As explained in the discussion in the proof of Corollary 4.3.21 in \cite{HChen}, $X$ being a quasi-compact algebraic space is sufficient for this finite generation requirement, as the $T^\an$-equivariant cohomology can be computed by a double complex whose $E_1$-page is given by $H(X)\otimes H_{T^\an}(\ast)$.
Moreover, as soon as $k=\bC$ the analytification of $X$ has the same homotopy type of a finite CW-complex, ensuring that Grojnowski's equivariant elliptic cohomology is also an object in the $\bZ_2$-periodic category $\Perf(E_T)_{\bZ_2}$. 
\end{remark}
\begin{remark}
On the other hand, by the same arguments as in Remark \ref{remark:coherence}, the coherence of the complex $\cHH_E([X/T])$ is controlled by finite generation of the Hochschild homology of $X$. In particular, a sufficient condition for the coherence of $\cHH_E([X/T])$ is that $X$ is proper. 
\end{remark}

\subsubsection{The rank one case}
Since the proof of the comparison theorem is an induction on the rank of the torus, we start by proving it for tori of rank $1$. 

\begin{proposition}[Comparison theorem, rank one case]\label{prop:Comparisonrk1}
Let $k=\bC$. Let $T$ be an algebraic torus of rank $1$ acting on a smooth variety $X$.
We have an isomorphism of $\bZ_2$-periodic coherent sheaves on $E$
$$\cHP_E([X/T])\simeq\El_{T^{\an}}(X^{\an})$$
where $\El_{T^{\an}}(X^{\an})$ denotes complexified equivariant elliptic cohomology  of the analytification of $X$. Moreover, this equivalence is natural with respect to $X$.
\end{proposition}
\begin{proof}
The proof of this theorem is based on adelic descent in dimension one. 

Let us start with the ad\`eles with respect to closed points $e$ of $E$. In this case, the ad\`eles are given by completion at such points. In particular, Theorem \ref{theorem:completionsgeneral} gives us the desired equivalence. The ad\`ele given by completion at the generic point corresponds to the generic fiber. The equivalence of such ad\`eles comes from observing that the generic fiber can be computed via the localization theorem. Let $c:E\to\Spec k$ be the structure morphism. Then we have a canonical isomorphism
$$\tilde{j}_{\eta}\cHP_{E}([X/T])\simeq \tilde{j}_{\eta}c^{\ast}\mHP(t_0 X^{T})$$
for elliptic Hochschild homology, and similarly 
$$\tilde{j}_{\eta}\El_{T^{\an}}(X^{\an})\simeq \tilde{j}_{\eta}c^{\ast}C_{dR}^{\oplus,\ast}((t_0 X^{T})^{\an})$$
for equivariant elliptic cohomology. Indeed, the localization theorem gives us a canonical equivalence
$$\tilde{j}_{\eta}\cHP([X/T])\simeq\tilde{j}_{\eta}\cHP([t_0X^{T(\eta)}/T])$$
and since $T(\eta)=T$, $T$ acts trivially on $t_0 X^T$ and we have canonical equivalences
$$\tilde{j}_{\eta}\cHP([t_0X^{T(\eta)}/T])\simeq\tilde{j}_{\eta}c^\ast\cHP([t_0X^{T(\eta)}/T'(\eta)])\simeq\tilde{j}_{\eta}c^{\ast}\mHP(t_0 X^{T})$$
The HKR theorem, as in Proposition 4.4 of \cite{BZNLoopConn} induces an equivalence
$$c^{\ast}\mHP(t_0 X^{T})\simeq c^{\ast}C_{dR}^{\oplus,\ast}((t_0 X^{T})^{\an})$$
which in turn gives an equivalence of the ad\`eles with respect to the generic point. 

Similarly, for a chain $\Delta=(\eta,x)$ where $\eta$ is the generic point and $x$ is a closed point, note that 
$$\bA_{E}(\Delta, \cF)=\bA_{E}(x, \tilde{j}_{\eta}\cF)$$
for a perfect complex $\cF$.
%
The same HKR theorem induces a canonical equivalence
$$\tilde{j}_{\eta}\cHP([X/T])\simeq\tilde{j}_{\eta}\El_{T^{\an}}(X^{\an})$$
and thus we have that 
$$\bA_{E}(\Delta, \cHP_E([X/T]))\simeq \bA_{E}\left (\Delta, \El_{T^{\an}}(X^{\an})\right )$$
Since this equivalence is induced by a canonical isomorphism of perfect complexes, it is compatible with the coface map corresponding to removing the point $x$ from the chain $\Delta$. Compatibility with the coface map induced by removing the point $\eta$ has to be tested separately. In particular, we need to check that the following diagram commutes:
$$\xymatrix{
\bA_{E}\left ((x), \cHP_E([X/T])\right )\ar[r]\ar[d] & \bA_{E}((x), \El_{T^{\an}}(X^{\an}))\ar[d] \\
\bA_{E}(\Delta, \cHP_E([X/T]))\ar[r] & \bA_{E}(\Delta, \El_{T^{\an}}(X^{\an}))
}
$$
We can rewrite the above diagram as
$$\xymatrix{
\cHP_E([X/T])_{\widehat{x}}\ar[r]\ar[d] & \El_{T^{\an}}(X^{\an})_{\widehat{x}}\ar[d] \\
\mHP(t_0(X^T))\otimes_{k}\Frac\cO_{E,\widehat{x}}\ar[r] & C_{dR}^{\oplus,\ast}(t_0(X^T)^\an)\otimes_{k}\Frac\cO_{E,\widehat{x}}
}
$$
since the ad\`eles with respect to the chain $\Delta$ correspond to
$$\left (\tilde{j}_{\eta}\cHP_E([X/T])\right )_{\widehat{x}}\simeq\left (\mHP(t_0(X^T))\otimes_{k}k(E)\right )\otimes_{\cO_{E,x}}\cO_{E,\widehat{x}}\simeq\mHP(t_0(X^T))\otimes_{k}\Frac\cO_{E,\widehat{x}}$$
Commutativity of this diagram then follows from the compatibility of the HKR isomorphisms as in Proposition 4.4 of \cite{BZNLoopConn} and Theorem 4.3.2 of \cite{HChen} with pullbacks and with the base change to $k(E)$. 
Indeed, we can rewrite $\mHP(t_0(X^T))\otimes_{k}\Frac\cO_{E,\widehat{x}}$ as 
$$\mHP(t_0(X^T))\otimes_{k}\Frac\cO_{E,\widehat{x}}\simeq\mHP([t_0(X^T)/T])_{\widehat{1}}\otimes_{\cO_{E,x}}k(E)$$
and $C_{dR}^{\oplus,\ast}(t_0(X^T)^\an)\otimes_{k}\Frac\cO_{E,\widehat{x}}$ as 
$$C_{dR}^{\oplus,\ast}(t_0(X^T)^\an)\otimes_{k}\Frac\cO_{E,\widehat{x}}\simeq C_{dR}^{\prod,\ast}(t_0(X^T)^\an)\otimes_{\cO_{E,x}}k(E)$$

We now show the naturality of the equivalence with respect to $X$. We need to show that if $f: Y\to X$ is a map, the following diagram commutes
$$\xymatrix{
\cHP_E([X/T])\ar[r]\ar[d] & \El_{T^{\an}}(X^{\an})\ar[d] \\
\cHP_E([Y/T])\ar[r] & \El_{T^{\an}}(Y^{\an})
}
$$
The corresponding diagram of ad\`eles commutes as the vertical maps are given by pullback in periodic cyclic homology (for the left arrow) and de Rham cohomology (for the right arrow); commutativity follows specifically from the compatibility of the HKR theorem with pullbacks, both in its form as Proposition 4.4 in \cite{BZNLoopConn} and as Theorem 4.3.2 in \cite{HChen}.
\end{proof}

\subsubsection{The general case}

We proceed with the proof of the main result of this Section. As already anticipated, the general case follows inductively from the rank 1 case.

\begin{theorem}[Comparison theorem]\label{theorem:comparisongeneral}
Let $k=\bC$. Let $T$ be an algebraic torus of rank $n$ acting on a smooth quasi-projective variety $X$.
There is an isomorphism of $\bZ_2$-periodic coherent sheaves on $E$
$$\cHP_E([X/T])\simeq\El_{T^{\an}}(X^{\an})$$
where $\El_{T^{\an}}(X^{\an})$ denotes the complexified equivariant elliptic cohomology   of the analytification of $X$. Moreover, this equivalence is natural with respect to $X$.
\end{theorem}

\begin{proof}
We reason by induction, as anticipated. 

Our inductive hypothesis gives us an equivalence $$\cHP_E([X/K])\simeq\El_{K^{\an}}(X^{\an})$$ for any torus $K$ of rank strictly smaller than $n$. Further, this equivalence is natural with respect to $X$. The base case when $\mathrm{rk}(T)=1$ was proved as Proposition \ref{prop:Comparisonrk1}.
Now let $T$ be an algebraic torus of rank $n$. We will produce an equivalence 
$$\cHP_E([X/T])\simeq\El_{T^{\an}}(X^{\an})$$
natural with respect to $X$. 
As in the rank one case, we use adelic descent. 
In the case of closed points $e\in E_T$, the equivalence
$$\bA_{E_T}\left ((e), \cHP_E([X/T])\right )\simeq\bA_{E_T}\left ((e), \El_{T^{\an}}(X^{\an})\right )$$
is Theorem \ref{theorem:completionsgeneral} together with Theorem 4.3.2 in \cite{HChen}. Now let $\Delta=(x>x_1>\dots>x_k)\in |E_T|_{k}$ be a chain of length $k>1$ on $E_T$, and let $\Delta '=(x_1>\dots>x_k)$. By definition we have
$$\bA_{E_T}\left (\Delta, \cHP_E([X/T])\right )\simeq \lim_{r\geq 0}\bA_{E_T}\left (\Delta', \tilde{j}_{rx}\cHP_E([X/T])\right )$$
The localization theorem provides an equivalence
$$\tilde{j}_{rx}\cHP_E([X/T])\simeq\tilde{j}_{rx}c_x^\ast\cHP_E([t_o X^{T(x)}/T'(x)])$$
By the inductive hypothesis there is an equivalence 
$$\cHP_E([t_o X^{T(x)}/T'(x)])\simeq\El_{T'(x)^{\an}}(t_0(X^{T(x)})^{\an})$$
Indeed $x$ is not a closed point, hence the rank of $T'(x)$ is necessarily smaller than $n$. Finally, this equivalence of quasi-coherent complexes induces an equivalence
$$\lim_{r\geq 0}\bA_{E_T}\left (\Delta', \tilde{j}_{rx}c_x^\ast\cHP_E([X/T])\right )\simeq\lim_{r\geq 0}\bA_{E_T}\left (\Delta', \tilde{j}_{rx}c_x^\ast\El_{T'(x)^\an}(t_0(X^{T(x)})^\an)\right )$$
which, by the computation carried out above, means that we obtain by composition a canonical isomorphism
$$\phi_{\Delta}:\bA_{E_T}\left (\Delta, \cHP_E([X/T])\right )\simeq \bA_{E_T}\left (\Delta, \El_{T^\an}(X^\an)\right )$$
All the equivalences between adelic groups produced by the above argument come from equivalences of objects in $\Perf(E_T)_{\bZ_2}$, thus they are all compatible with the coface maps induced by the operation of removing a point $x_i$ from the chain $\Delta$, for $i\in \{1,\dots,k\}$. The coface map induced by removing $x$ has to be treated separately: in this case we cannot reduce to the sheaves $\tilde{j}_{rx}\cHP_E([X/T])$ for one of the two adelic groups involved. 
Indeed, by definition, the ad\`ele $\bA_{E_T}\left (\Delta',\cHP_E([X/T])\right )$ is a limit of ad\`eles of the sheaves $\tilde{j}_{rx_1}\cHP_E([X/T])$ rather than $\tilde{j}_{rx}\cHP_E([X/T])$.

Let us consider now the case of the coface map induced by removing the point $x$ from the chain $\Delta$. For this specific case, we switch to the complexes of ad\`eles as opposed to their global sections. 
We need to check the commutativity of the following diagram:
$$\xymatrix{
\bfA_{E_T}(\Delta', \cHP_E([X/T]))\ar[r]\ar[d] & \bfA_{E_T}(\Delta', \El_{T^{\an}}(X^{\an}))\ar[d] \\
\bfA_{E_T}(\Delta, \cHP_E([X/T]))\ar[r] & \bfA_{E_T}(\Delta, \El_{T^{\an}}(X^{\an}))
}
$$
where
\begin{itemize}
\item $\Delta'$ is obtained from $\Delta$ by removing $x$  \item  the horizontal arrows are the isomorphisms obtained above via the inductive hypothesis 
\item the vertical arrows are the coface maps
\end{itemize} For the sheaves of ad\`eles we have a decomposition
$$\bfA_{X}(\Delta, \cF)\simeq\cF\otimes_{\cO_X}\bfA_{X}(\Delta, \cO_X)$$
which allows us to write the above diagram as 
$$\xymatrix{
\bfA_{E_T}(\Delta', \cHP_E([X/T]))\ar[r]^{\simeq}\ar[d]^{\simeq} & \bfA_{E_T}(\Delta', \El_{T^{\an}}(X^{\an}))\ar[d]^{\simeq} \\
c_{x_1}^{\ast}\cHP_E([t_0(X^{T(x_1)})/T'(x_1)])\otimes_{\cO_X}\bfA_{X}(\Delta', \cO_X)\ar[r]^{\simeq}\ar[d] & c_{x_1}^{\ast}\El_{T'(x_1)^{\an}}(t_0(X^{T(x_1)})^{\an})\otimes_{\cO_X}\bfA_{X}(\Delta', \cO_X)\ar[d] \\
c_{x}^{\ast}\cHP_E([t_0(X^{T(x)})/T'(x)])\otimes_{\cO_X}\bfA_{X}(\Delta, \cO_X)\ar[r]^{\simeq}\ar[d] & c_{x}^{\ast}\El_{T'(x)^{\an}}(t_0(X^{T(x)})^{\an})\otimes_{\cO_X}\bfA_{X}(\Delta, \cO_X)\ar[d] \\
\bfA_{E_T}(\Delta, \cHP_E([X/T]))\ar[r]^{\simeq} & \bfA_{E_T}(\Delta, \El_{T^{\an}}(X^{\an}))
}
$$
Since $x>x_1$, $c_x$ factors as the composition
$$c_x:E_T\xrightarrow{c_{x_1}}E_{T'(x_1)}\xrightarrow{c_{x_1,x}}E_{T'(x)}$$
In particular, the two middle vertical maps in the diagram factor as the tensor product of a pullback map 
along the inclusion $t_0(X^{T(x)})\to t_0(X^{T(x_1)})$
and the coface map for the ad\`eles of the structure sheaf. If $k>1$ or $k=1$ and $x_1$ is not closed, the bottom and top squares commute. By inductive hypothesis on naturality, the middle square also commutes, and thus we obtain the desired equivalence.

If $k=1$ and $x_1$ is closed, $T'(x_1)$ might have the same rank as $T$. In this situation, we apply Proposition 3.2.1 in \cite{HubAd}:
$$\bA_{E_T}\left ((x_0,x_1),\cHP_{E}([X/T])\right )=C_{x_0}S_{x_0}^{-1}C_{x_1}S_{x_1}^{-1}\cHP_{E}([X/T])$$
where $S_{p}^{-1}$ is localization at $p$ and $C_{p}$ is the functor that sends a quasi-coherent complex 
$$M \simeq \colim_{i}N_i$$ with $N_i$ perfect complex, to $\colim_{i}N_{i,\widehat{p}}$. The coface map is the natural map
$$C_{x_1}S_{x_1}^{-1}\cHP_{E}([X/T])\to C_{x_0}S_{x_0}^{-1}C_{x_1}S_{x_1}^{-1}\cHP_{E}([X/T])$$
and similarly for elliptic cohomology.
The inductive hypothesis and localization identify the isomorphism $C_{x_1}S_{x_1}^{-1}\cHP_{E}([X/T])\simeq C_{x_1}S_{x_1}^{-1}\El_{T^\an}(X^\an)$ with the HKR isomorphism (together with Theorem \ref{theorem:completionsgeneral}), which implies that the relevant diagram commutes.

To finish the argument, we need to prove naturality.
This follows from the inductive hypothesis in the following way: choose $f:Y\to X$ and consider the diagram of ad\`eles associated to naturality, for a chain $\Delta\in |E_T|_k$:
$$\xymatrix{
\bA_{E_T}\left (\Delta,\cHP_E([X/T])\right )\ar[r]\ar[d] & \bA_{E_T}\left (\Delta,\El_{T^{\an}}(X^{\an})\right )\ar[d] \\
\bA_{E_T}\left (\Delta,\cHP_E([Y/T])\right )\ar[r] & \bA_{E_T}\left (\Delta,\El_{T^{\an}}(Y^{\an})\right )
}
$$
By localization, the above diagram is obtained by applying the functor $\bA_{E_T}(\Delta,-)$ to the diagram
$$\xymatrix{
c_x^\ast\cHP_E([t_0 X^{T(x)}/T'(x)])\ar[r]\ar[d] & c_x^\ast\El_{T'(x)^{\an}}((t_0 X^{T(x)})^{\an})\ar[d] \\
c_x^\ast\cHP_E([t_0 Y^{T(x)}/T'(x)])\ar[r] & c_x^\ast\El_{T'(x)^{\an}}((t_0 Y^{T(x)})^{\an})
}
$$
which commutes by the inductive hypothesis. 
\end{proof}

From Theorem \ref{theorem:comparisongeneral} we deduce the following corollary:

\begin{corollary}
Let $k=\bC$. Let $T$ be an algebraic torus of rank $n$ acting on a smooth quasi-projective variety $X$.
We have an isomorphism of $\bZ_2$-periodic coherent sheaves on $E$
$$\pi_{\ast}\cHP_E([X/T])\simeq\pi_\ast\El_{T^{\an}}(X^{\an})=\El_{T^{\an}}^{\ast}(X^{\an})$$
where $\pi_{\ast}$ denotes homotopy sheaves; $\El_{T^{\an}}^{\ast}(X^{\an})$ denotes the collection of homotopy sheaves of complexified equivariant elliptic cohomology of the analytification of $X$, i.e. the classical version of equivariant elliptic cohomology due to Grojnowski. Moreover, this equivalence is natural with respect to $X$. 
\end{corollary}

\begin{remark}
We expect elliptic Hochschild homology to encode 2-categorical information on the stack $[X/T]$.   This is the most exciting future direction of our work, as it could shed light on the much studied problem of constructing geometric representatives of elliptic cocycles. As a  reality check we remark that, in contrast with ordinary Hochschild homology, $\cHH_E([X/T])$ and $\cHP_E([X/T])$  are \emph{not} invariants of $\Perf([X/T])$. This follows immediately from Theorem \ref{theorem:comparisongeneral} and the main result of \cite{SchSib}. The category $\Perf([X/T])$  can be viewed, in a sense, as the  universal recipient of 1-categorical information on $[X/T]$: hence, the fact that equivariant elliptic cohomology is not an invariant of $\Perf([X/T])$  confirms the expectation that elliptic Hochschild homology detects information which is not 1-categorical in nature.  
\end{remark}

\section{The case of reductive G} 
\label{sect:ReductiveG}
In this section we extend our construction from the case of an algebraic torus to that of a reductive group $G$. We present two approaches. The first one follows closely Grojnowski's original definition (Definition \ref{def:HHTW}). We will then sketch an alternative construction, which is arguably preferable, as it is fully intrinsic (Definition \ref{def:HHTW2}). We stress however that in this paper we will limit ourselves to give a sketch of the comparison between these two definitions of equivariant elliptic Hochschild homology; a complete argument will appear in future work.  

Let $k$ be a field of characteristic 0. Let $G$ be a reductive group over  $k$ acting on a smooth quasi-projective $k$-variety $X$, and let $T$ be a maximal torus. Following Grojnowski's approach (see Section \ref{subsection:PerlimGroj} for notation), we give the following definition:

\begin{definition}\label{def:HHTW}
We define the \emph{elliptic Hochschild homology} of $[X/G]$ to be
$$\cHH_E([X/G]):=\cHH_E([X/T])^W\in \QCoh(E_G)$$
where $W$ is the Weyl group of $G$.
Similarly, we define the \emph{elliptic periodic cyclic homology} of $[X/G]$ as
$$\cHP_E([X/G]):=\cHP_E([X/T])^W$$
\end{definition}
As a consequence of Theorem \ref{theorem:comparisongeneral} and of Definition \ref{def:HHTW} we obtain the following Corollary.
\begin{corollary}\label{cor:comparisonG}
Let $X$ be a smooth quasi-projective variety with an action of a reductive group $G$. 
Then there is an equivalence
$$\cHP_E([X/G])\simeq\El_G(X^\an)$$
\end{corollary}

We proceed next to sketch a more intrinsic approach to  $G$-equivariant elliptic Hochschild homology in terms of quasi-constant maps to $[X/G]$. 
First of all, we consider the stack 
$$\Mapo{E}{[X/G]}$$
of quasi-constant maps from $E$ to $[X/G]$. This stack is defined analogously to Definition \ref{quasiconstant}, and naturally lives over $\Bun^0_G(E)$.
The next step is to restrict our attention to the stack of \emph{semistable} quasi-constant maps which is given by the pullback along the inclusion 
$$\Bun_G^{0}(E)^{ss}\to\Bun_G^0(E)$$
of the semi-stable locus in the stack of degree zero principal $G$-bundles on $E$
$$
\xymatrix{
\Mapo{E}{[X/G]}^{ss} \ar[r] \ar[d] & \Mapo{E}{[X/G]} \ar[d] \\
\Bun_G^{0}(E)^{ss}\ar[r] & \Bun_G^0(E)
}
$$

In Definition \ref{def:HHTW2} below, we give an alternative approach to the  elliptic Hochschild homology  of $[X/G]$  in terms of semistable quasi-constant maps to $[X/G]$.  We stress that the identification between Definition \ref{def:HHTW2} and Definition \ref{def:HHTW} has to be regarded for the moment as conjectural, as we will only explain the main ingredients that should go into the comparison; a complete argument will  appear in future work.

The next definition makes sense for any algebraic group $G$, not necessarily reductive. Denote by $p$ the composition
\[
p:\Mapo{E}{X/G}^{ss}\to\Mapo{E}{BG}^{ss}\to E_G
\]
where $E_G$ denotes the coarse moduli space of semistable degree $0$ principal $G$ bundles on $E$.
\begin{definition}
\label{def:HHTW2}
The (intrinsic) \emph{elliptic Hochschild homology} of $[X/G]$ is defined as 
$$\cHH_E^{int}([X/G]):=p_\ast\cO_{\Mapo{E}{[X/G]}^{ss}}\in\QCoh(E_G)$$
and the (intrinsic) \emph{elliptic periodic cyclic homology} of $[X/G]$, 
$$\cHP_E^{int}([X/G])$$
is the image of the pair $\cHH_E^{int}([X/G])$ and its natural $S^1$-action in the category $\QCoh(E_G)_{\bZ_2}$.
\end{definition}
The $S^1$-action appearing in the previous definition is constructed as in Section \ref{section:Tate}.

\begin{remark}
When $G=T$ is a torus, 
$$\cHH_E^{int}([X/T])\simeq\cHH_E([X/T])$$
\end{remark}

\begin{remark}
Similarly to the case of abelian groups, $G$-equivariant elliptic Hochschild and periodic cyclic homology satisfy all expected functorialities. Namely, for every $G$-equivariant map $X\to Y$ we obtain a map  
\[
\cHH_E([Y/G])\to \cHH_E([X/G])
\]
Further, a group homomorphism $f:H\to G$  induces a map $\phi_f:E_H\to E_G$,
\[
\cHH_E([X/G])\to (\phi_f)_\ast\cHH_E([X/H])
\]
and analogously for elliptic periodic cyclic homology.
\end{remark}

We conjecture the following relation between Definitions \ref{def:HHTW} and \ref{def:HHTW2}.

\begin{conjecture}\label{conjecture:HPG}
Let $G$ be a reductive group. Let $X$ be a smooth quasi-projective variety acted on by an algebraic group $G$. Then there is an equivalence of objects in $\Coh(E_G)_{\bZ_2}$
$$\cHP_E([X/G])\simeq\cHP_E^{int}([X/G])$$
In particular,
$$\cHP_E^{int}([X/G])\simeq\El_G(X^\an)$$
\end{conjecture}

We plan to return to Conjecture \ref{conjecture:HPG}  in future work. We conclude the paper by explaining in some detail a proof   strategy towards Conjecture \ref{conjecture:HPG}, which is inspired by classical  work of Thomason in the context of algebraic K-theory and G-theory. The argument that we will present crucially relies on Conjectures \ref{conjecture:HHEB} and \ref{conjecture:HHEG} below, which successively reduce the problem from the group $G$ to a Borel subgroup $B$ and finally to a maximal torus $T$.

In what follows, we fix a maximal torus $T$ of $G$ and a Borel $B$ containing $T$.

\begin{conjecture}[Reduction from $B$ to $T$]\label{conjecture:HHEB}
Let $X$ be a smooth quasi-projective variety over $k$ with an action of $G$. Then natural map
$$\cHP_E^{int}([X/B])\to\cHP_E([X/T])$$
is an equivalence.
\end{conjecture}
A proof of Conjecture \ref{conjecture:HHEB} would follow immediately from the homotopy-invariance properties of elliptic periodic cyclic homology, which we intend to establish  in future work.

The next step is the reduction of $\cHH_E^{int}([X/G])$ to the quotient by the Borel subgroup. 
\begin{conjecture}[Reduction from $G$ to $B$]\label{conjecture:HHEG}
Let $X$ be a variety over $k$ with an action of $G$. Let $W$ be the Weyl group. The natural map
$$\cHH_E^{int}([X/G])\to\phi_\ast\cHH_E^{int}([X/B])$$
induces an equivalence
$$\cHH_E^{int}([X/G])\simeq\cHH_E^{int}([X/B])^W.$$
\end{conjecture}
In the statement above $\phi_*$ is the pushforward  along the map
$$\phi:E_T\simeq E_B\to E_G$$
 which is induced by the inclusion $T < G$. Also, we view the invariants $\cHH_E^{int}([X/B])^W$ as an object of $\QCoh(E_G)$ via the identification  
$E_G\simeq E_T//W$, which is  
due to Laszlo \cite{Laszlo}. 
\begin{remark}\label{remark:Intrinsic_T}
Note that the identification $E_T\simeq E_B$, which is used in the discussion above, holds because the coarse moduli space of the stack of bundles is insensitive to unipotent bundles, hence it  does not distinguish $T$ from $B$.
\end{remark}
We expect Conjecture \ref{conjecture:HHEG} to follow from the geometric properties of the equivariant elliptic Grothendieck--Springer resolution
$$\Bun_{B}^{0}(E)^{ss}\to\Bun_{G}^{0}(E)^{ss}.$$ Let us explain this in some more detail. We start by observing that we have an identification 
$$[X/B]\simeq [X\times^{B}G/G]=[(X\times G/B)/G]$$
which implies that there is an equivalence
$$\phi_\ast\cHH_E^{int}([X/B])\simeq\cHH_E^{int}([(X\times G/B)/G]).$$ 
Since $\Mapo{E}{-}$ preserves pullbacks, we obtain a pullback diagram
$$\xymatrix{
\Mapo{E}{[(X\times G/B)/G]}^{ss}\ar[r]^{\beta} \ar[d]_{\alpha}  & \Mapo{E}{[(G/B)/G]}^{ss} \ar[d]^{\gamma} \\
\Mapo{E}{[X/G]}^{ss} \ar[r]^{p''} & \Bun^{0}_G(E)^{ss}
}
$$
of the restriction to the semistable locus. We observe that, as 
$$\Mapo{E}{[(G/B)/G]}^{ss} \simeq \Bun_B^{0}(E)^{ss}$$
the map $\gamma$ is the equivariant elliptic Grothendieck--Springer resolution of \cite{BZNEllipticGS}. 

The crux of the argument would be to show that the elliptic Grothendieck--Springer resolution has the property that 
$$\gamma_\ast\cO_{\Mapo{E}{[(G/B)/G]}^{ss}}\simeq\cO_{\Bun^{0}_G(E)^{ss}}\otimes_{k} \mHH(G/B)$$
We  will give a rigorous proof of this fact in future work; here we will limit ourselves to explain how this implies the desired claim. Note indeed that, setting   $p'=p''\circ\alpha$, we obtain an  equivalence
$$p'_\ast\cO_{\Mapo{E}{[(X\times G/B)/G]}^{ss}}\simeq p''_\ast\cO_{\Mapo{E}{[X/G]}^{ss}}\otimes_{k}\mHH(G/B)$$
The conclusion follows from the computation of the $W$-invariants on both sides of this identification, and then by applying pushforward along the map $\chi : \Bun^{0}_G(E)^{ss}\to E_G$ 
$$\cHH_E^{int}([X/G])\simeq\cHH_E^{int}([X/B])^W.$$

 Conjecture \ref{conjecture:HPG}  follows immediately from the arguments above. Indeed, 
combining Conjecture \ref{conjecture:HHEB} and  \ref{conjecture:HHEG} we obtain an equivalence
$$\cHP_E^{int}([X/G])\simeq\cHP_E^{int}([X/T])^W.$$
By Remark \ref{remark:Intrinsic_T}
$$\cHP_E^{int}([X/T])^W\simeq\cHP_E([X/T])^W$$
which combined with Definition \ref{def:HHTW} yields
$$\cHP_E^{int}([X/G])\simeq\cHP_E([X/T])^W=\cHP_E([X/G])\simeq\El_G(X^\an)$$
as desired.

\end{document}